\documentclass[11pt]{amsart}
\usepackage{amsmath,amssymb,amsthm,a4wide,scalerel}
\usepackage{mathabx}
\usepackage{color}
\usepackage{duckuments}
\usepackage{xcolor,graphicx}
\usepackage{mathrsfs}
\usepackage{tikz,pgfplots}
\usepackage{subcaption}
\usepackage[normalem]{ulem}
\usepackage{hyperref}
\usepackage{dsfont}
\usepackage{comment}
\usepackage{enumitem}
\usepackage{ctable} 
\usepackage{cases}

\newtheorem{theorem}{Theorem}
\newtheorem{lemma}[theorem]{Lemma}
\newlist{sublemmas}{enumerate}{1}
\setlist[sublemmas]{
  label=\textit{(\roman*)},
  ref=\thelemma\textit{(\roman*)},
  leftmargin=*,
  align=right,
  labelsep=0.5em
}
\newtheorem{proposition}[theorem]{Proposition}
\newtheorem{remark}[theorem]{Remark}
\newlist{subremarks}{enumerate}{1}
\setlist[subremarks]{
  label=\textit{(\roman*)},
  ref=\theremark\textit{(\roman*)},
  leftmargin=*,
  align=right,
  labelsep=0.5em
}
\newtheorem{corollary}[theorem]{Corollary}
\newtheorem{defi}[theorem]{Definition}

\newtheorem{assumption}{Assumption}
\newlist{subassumptions}{enumerate}{1}
\setlist[subassumptions]{
  label=\textit{(\roman*)},
  ref=\theassumption\textit{(\roman*)},
  leftmargin=*,
  align=right,
  labelsep=0.5em
}
\newtheorem{example}{Example}

\newcommand{\dt}{\tau}
\newcommand{\uk}{{\underline{k}}}

\newcommand{\R}{{\mathbb R}}
\newcommand{\N}{{\mathbb N}}
\newcommand{\E}{{\mathbb E}}
\newcommand{\PP}{{\mathbb P}}

\newcommand{\e}{\mathrm{e}}
\newcommand{\abs}[1]{\left\lvert #1 \right\rvert}
\providecommand{\norm}[1]{\left\lVert#1\right\rVert}
\newcommand*\dd{\mathop{}\!\mathrm{d}}

\newcommand{\mse}[1]{\norm{ #1 }_{L^2(\Omega)}}

\newcommand{\dr}{b}
\newcommand{\drf}{a}
\newcommand{\di}{g}
\newcommand{\dif}{f}
\newcommand{\distar}{\di^{\star}}
\newcommand{\drstar}{\dr^{\star}}
\newcommand{\difstar}{\dif^{\star}}
\newcommand{\drfstar}{\drf^{\star}}
\newcommand{\ustar}{u^{\star}}
\newcommand{\util}{u^{\text{DP}}}
\newcommand{\Ztil}{\widetilde{Z}}
\newcommand{\vtil}{\widetilde{v}}
\newcommand{\dbt}{\delta\beta}
\newcommand{\UMY}{U_m(y)}
\newcommand{\Util}{U^{\text{DP}}}
\newcommand{\Ustar}{U^\star}

\usepackage{color}

\author{Charles-Edouard Br\'ehier}
              \address{Universite de Pau et des Pays de l'Adour, E2S UPPA, CNRS, LMAP, Pau, France}
              \email{charles-edouard.brehier@univ-pau.fr}

\author{David Cohen}
              \address{Department of Mathematical Sciences,
              Chalmers University of Technology and University of Gothenburg, 41296~Gothenburg, Sweden}
              \email{\tt david.cohen@chalmers.se}

\author{Gijs Custers}
              \address{Department of Mathematical Sciences,
              Chalmers University of Technology and University of Gothenburg, 41296~Gothenburg, Sweden}
              \email{\tt gijs.custers@chalmers.se}

\begin{document}

\title[Domain preserving schemes for SPDE\MakeLowercase{s} driven by a standard Brownian motion]{Domain preserving splitting schemes for a class of SPDE\MakeLowercase{s} driven by a standard Brownian motion}

\begin{abstract}
We consider a class of SPDEs driven by a standard real-valued Brownian motion, with drift and diffusion coefficients such that there exists a unique mild solution taking values in the interval $[-1,1]$ almost surely.
To preserve this qualitative property of the exact solution, we propose a domain preserving Lie--Trotter splitting scheme: for any choice of the time-step size, the numerical solution takes values in the interval $[-1,1]$ almost surely. Furthermore, we prove mean-square convergence with rate $1/2-$ for the domain preserving Lie--Trotter scheme. These theoretical results are illustrated with numerical experiments.
\end{abstract}

\maketitle

{\small\noindent
{\bf AMS Classification.} 60H10, 60H15, 60H35, 65C30.

\bigskip\noindent{\bf Keywords.} Stochastic partial differential equations. Geometric numerical integration. Explicit splitting schemes. Domain preserving numerical schemes. Strong convergence.

\section{Introduction}\label{sec-intro}

Stochastic partial differential equations (SPDEs) are popular models used to describe a wide range of processes in various scientific fields,
see e.\,g. \cite{MR4458524,MR3288853,MR3236753,MR1500166,DalangSanz2026,MR4807583,MR3308418,MR2329435}.
In several applications, the solutions of such SPDEs take values in a given bounded interval almost surely. 
A first example is the stochastic Nagumo equation, see for instance~\cite{MR3907917}.
This SPDE is used to model action potential propagation on nerve fibres, which indicates that its solutions take values in the interval $[0,1]$ almost surely.
The particular noise term used in this model is biologically motivated in~\cite{MR3987335}.
Another example is the stochastic Fisher--KPP equation, which models population growth and wave propagation, see~\cite{MR4259661}: if $K>0$ is the habitat capacity, then, the solution to this SPDE takes values in the interval $[0,K]$ almost surely. Ensuring that numerical approximations of these SPDEs remain in the appropriate interval is important. Standard time integrators such as the semi-implicit Euler--Maruyama scheme or the stochastic exponential Euler scheme may leave the given interval.
This highlights the need for domain preserving schemes for this type of SPDEs, which guarantee that numerical approximations remain in a given interval almost surely.

In this article, we consider parabolic semilinear SPDEs (interpreted in the It{\^o} sense) in the $d$-dimensional spatial domain $\mathcal{D}=(0,1)^d$, with homogeneous Dirichlet boundary conditions on $\partial\mathcal{D}$, driven by a standard real-valued Brownian motion $\beta$, and an initial value $u_0$:
\begin{equation}\label{eq-SPDEintro}
\left\lbrace
\begin{aligned}
&        \dd u(t,x) = \bigl( \Delta u(t,x) + \dr (t,x,u(t,x)) \bigr)\dd t + \di (t,x,u(t,x))\dd\beta(t), \quad t\in(0,T], ~x\in\mathcal{D}, \\
&        u(t,x) = 0, \quad t\in[0,T],~x\in\partial\mathcal{D}, \\
&        u(0,x) = u_0(x), \quad x\in\overline{\mathcal{D}}.
\end{aligned}
\right.
\end{equation}
Under the condition that $\dr(t,x,\pm1)=\di(t,x,\pm1)=0$ for all $t\ge 0$ and $x\in\overline{\mathcal{D}}$, and under appropriate regularity assumptions on the coefficients $\dr$ and $\di$, if the initial value $u_0$ takes values in the domain $[-1,1]$, then the solution to~\eqref{eq-SPDEintro} is guaranteed to verify $u(t,x)\in[-1,1]$ for all $t\ge 0$ and $x\in\overline{\mathcal{D}}$ almost surely.

The aim of this work is to design and analyse a class of time integrators for the SPDE~\eqref{eq-SPDEintro} taking values in the domain $[-1,1]$ almost surely. Such time integrators are said to be domain preserving. The proposed numerical method is based on a Lie--Trotter splitting strategy: the SPDE~\eqref{eq-SPDEintro} is decomposed into the subsystems
\begin{align}
&    \dd v_1(t,x) = \dr(t,x,v_1(t,x))\dd t + \di(t,x,v_1(t,x))\dd\beta(t), \label{eq-SDEintro} \\
&    \dd v_2(t,x) = \Delta v_2(t,x)\dd t, \label{eq-HEintro}
\end{align}
where the linear heat equation~\eqref{eq-HEintro} is combined with homogeneous Dirichlet boundary conditions as in~\eqref{eq-SPDEintro}.
The stochastic differential equation (SDE)~\eqref{eq-SDEintro} is parameterised by the position variable $x\in\overline{\mathcal{D}}$ and is discretised by applying a domain preserving numerical scheme introduced in the recent work~\cite{manu}. The deterministic linear heat equation~\eqref{eq-HEintro} is solved exactly. The domain preserving Lie--Trotter (DPLT) splitting scheme is given by~\eqref{eq-dplt}, described in more detail in Section~\ref{sec-bp}.

In this article, we prove that the time integrator~\eqref{eq-dplt} is domain preserving for any choice of the time-step size, see Corollary~\ref{corodpdplt},
and that it converges in the mean-square sense with rate $1/2-$ to the solution of the SPDE~\eqref{eq-SPDEintro}, see Corollary~\ref{lecorollaire}.

Note that the existence and uniqueness of a mild solution~\eqref{eq-SPDEintro} with values in $[-1,1]$ (see Definition~\ref{defi-mild}) is, to the best of our knowledge, not known in the literature. This property, see Proposition~\ref{prop-exuniqdp}, is one of the main contributions of this article. To prove it, we consider a modified version~\eqref{eq-SPDEstar} of the SPDE~\eqref{eq-SPDEintro} with extensions for $u\in\R$ denoted by $\drstar(t,x,u)$ and $\distar(t,x,u)$ of the drift and diffusion coefficients $\dr(t,x,u)$ and $\di(t,x,u)$ which are defined only for $u\in[-1,1]$. We refer to Section~\ref{sec-sol} for the details. We first define a version of the DPLT splitting scheme for the modified SPDE~\eqref{eq-SPDEstar}, and by domain preservation (see Proposition~\ref{prop-dplt}) we show it coincides with the DPLT scheme given by~\eqref{eq-dplt}, defined using the original coefficients $\dr$ and $\di$. This argument also shows that the DPLT splitting scheme~\eqref{eq-dplt} is domain preserving, see Corollary~\ref{corodpdplt}. We then prove the main technical result of this work, Theorem~\ref{thm-mainthm}, which shows convergence of the DPLT scheme for the modified SPDE~\eqref{eq-SPDEstar} with rate $1/2-$. Combining the convergence result given by Theorem~\ref{thm-mainthm} with the domain preservation property given by Proposition~\ref{prop-dplt} then shows Proposition~\ref{prop-exuniqdp}. The proof of the mean-square convergence of the DPLT splitting scheme~\eqref{eq-dplt} to the solution of the SPDE~\eqref{eq-SPDEintro} then follows by a straightforward argument.

To the best of our knowledge, this article is the first where strong convergence of a domain preserving numerical scheme is established for the general class~\eqref{eq-SPDEintro} of SPDEs. A weak approximation domain preserving scheme has been proposed in the work~\cite{ulander2025boundary}, using a different approach.
However, the literature contains other relevant contributions on domain and positivity preserving schemes for SDEs and SPDEs.
Our construction of the DPLT scheme for the SPDE~\eqref{eq-SPDEintro} heavily relies on the domain preserving schemes studied in the recent work~\cite{manu} for stochastic differential equations. Other related recent works are~\cite{MR4993436,MR4737060}, and we refer to the list of references of~\cite{manu} for other related works.
For SPDEs, positivity preserving temporal and spatial discretisation schemes have been studied in several recent works. First, a fully-discrete positivity preserving scheme for a SPDE in spatial dimension $1$ driven by space-time white noise has been proposed and studied in~\cite{MR4780408}, using finite differences for the spatial approximation. For SPDEs driven by a purely time-dependent Brownian motion $\beta$, a positivity preserving time integrator has been presented in~\cite{MR4729657}. The article~\cite{Djurdjevac_2026} is focused on the preservation of positivity for the spatial approximation performed by finite element methods, and~\cite{hearder2026fullydiscretenonnegativitypreservingfem} then studies a fully-discrete version.

This article is organised as follows. Section~\ref{sec-set} is devoted to the presentation of the setting, and in particular with the statement of the assumptions on the coefficients $\dr$ and $\di$, the description of the modified version~\eqref{eq-SPDEstar} of the SPDE~\eqref{eq-SPDE}, and definitions and properties of the mild solutions to these equations. In Section~\ref{sec-bp}, we present the construction of the domain preserving Lie--Trotter splitting scheme~\eqref{eq-dplt}, and its main properties: domain preservation (Proposition~\ref{prop-dplt}) and convergence (Theorem~\ref{thm-mainthm} and Corollary~\ref{lecorollaire}). We also present numerical illustrations of the main results in Section~\ref{subNumExp}. The proofs of auxiliary results and of Theorem~\ref{thm-mainthm}, which rely on a continuous-time version of the scheme, are provided in Section~\ref{sec-proofs}.
Section~\ref{sec:genZ} presents some generalisations of the construction and analysis of the DPLT splitting scheme in three directions: for SPDEs with a Stratonovich interpretation of the noise (Section~\ref{sec:Strato}), for SPDEs driven by multidimensional noise (Section~\ref{sec:mulNoise}) and for systems of SPDEs (Section~\ref{sec:syst}). Covering the second and third generalisations would require minor modifications in the analysis, but this is omitted for simplicity of notation. We validate the properties of the DPLT splitting scheme in these contexts with numerical experiments.

\section{Setting}\label{sec-set}

In this section, we present the setting and necessary assumptions on the considered SPDE. Next, we recall some useful properties of the heat kernel in ${[0,1]}^d$.
We then present mild solutions to the SPDE and their properties. Finally, we introduce an extended version of the coefficients and of the SPDE that will be used to define the proposed numerical scheme in the next section.

\subsection{Introducing the SPDE}

Let $T\in(0,\infty)$ denote a finite time horizon. The spatial domain is denoted by $\mathcal{D}={(0,1)}^d$ in arbitrary dimension $d\in\N$, and $\Delta = \partial^2_{11} + \ldots + \partial^2_{dd}$ denotes the Laplace operator. In addition, the closure of $\mathcal{D}$ is denoted by $\overline{\mathcal{D}}=[0,1]^d$ and its boundary is given by $\partial\mathcal{D}=\overline{\mathcal{D}}\setminus\mathcal{D}$.
The SPDE is driven by a standard real-valued Brownian motion denoted by $\beta$ and defined on a probability space $(\Omega,\mathcal{F},\PP)$ satisfying the usual conditions. We consider the following class of stochastic partial differential equations interpreted in the Itô sense:
\begin{equation}\label{eq-SPDE}
\left\lbrace
\begin{aligned}
&        \dd u(t,x) = \bigl( \Delta u(t,x) + \dr (t,x,u(t,x)) \bigr)\dd t + \di (t,x,u(t,x))\dd\beta(t), \quad t\in(0,T], ~x\in\mathcal{D}, \\
&        u(t,x) = 0, \quad t\in[0,T],~x\in\partial\mathcal{D}, \\
&        u(0,x) = u_0(x), \quad x\in\overline{\mathcal{D}}.
\end{aligned}
\right.
\end{equation}
Further assumptions on the drift coefficient $\dr$, the diffusion coefficient $\di$, and the initial value $u_0$ are given below.

Let us denote the standard Euclidean norm on $\mathbb R^d$ by $\norm{\cdot}$. Further, we denote by $\mathcal{C}_0^0(\overline{\mathcal{D}})$ the space of continuous functions $h:\overline{\mathcal{D}}\to\R$ that vanish on the boundary of $\mathcal{D}$, i.\,e. such that $h(x)=0$ for all $x\in\partial\mathcal{D}$.

Throughout this paper, we impose the following assumptions on the initial value $u_0$ of the SPDE~\eqref{eq-SPDE}.
\begin{assumption}\label{ass-ic}
    The initial value $u_0$ is a deterministic function, which satisfies $u_0\in \mathcal{C}_0^0(\overline{\mathcal{D}})$, and verifies $u_0(x)\in[-1,1]$ for all $x\in\overline{\mathcal{D}}$.
\end{assumption}

\begin{remark}\label{rmk-generalisations1}
One could consider random $\mathcal{F}_0$-measurable initial conditions $u_0$ with some minor modifications in the setting and proofs below.
\end{remark}

We will also need the following assumptions on the drift and diffusion coefficients of the SPDE~\eqref{eq-SPDE}.
\begin{assumption}\label{ass-driftdiff}
    We assume that the drift and diffusion coefficients $\dr,\di\colon[0,\infty)\times\overline{\mathcal{D}}\times[-1,1]\to\R$ of the SPDE~\eqref{eq-SPDE} are continuous mappings, which satisfy the conditions
    \begin{subassumptions}
    \item\label{ass-driftdiff-vanish} $\dr(t,x,\pm1) = \di(t,x,\pm1) = 0$ for all $t\in[0,\infty)$ and $x\in\overline{\mathcal{D}}$.
    \item[]\hspace{-1cm} In addition, we impose the following regularity for the coefficients of the SPDE~\eqref{eq-SPDE}:
    \item\label{ass-driftdiff-regt} For all $x\in\overline{\mathcal{D}}$ and $u\in[-1,1]$, the mappings $\dr(\cdot,x,u)$ and $\di(\cdot,x,u)$ are H{\"o}lder continuous with exponent $\frac{1}{2}$ on $[0,\infty)$.
    In addition, for all $T\in(0,\infty)$, there exists $C(T)\in(0,\infty)$ such that for all $s,t\in[0,T]$, one has
    \begin{equation*}
        \sup_{x\in\overline{\mathcal{D}}}~\sup_{u\in[-1,1]}~\Bigl(\abs{\dr(t,x,u) - \dr(s,x,u)} + \abs{\di(t,x,u) - \di(s,x,u)}\Bigr) \leq C(T)\abs{t-s}^\frac{1}{2}.
    \end{equation*}
    \item\label{ass-driftdiff-regu} For all $t\in[0,\infty)$ and $x\in\overline{\mathcal{D}}$, the mappings $\dr(t,x,\cdot)$ and $\di(t,x,\cdot)$ are of class $\mathcal{C}^2$ on $[-1,1]$, and the mappings $\partial_u^2\dr$ and $\partial_u^2\di$ are continuous on $[0,\infty)\times\overline{\mathcal{D}}\times[-1,1]$.
    In addition, for all $T\in(0,\infty)$, there exists $C(T)\in(0,\infty)$ such that
    \begin{equation*}
    \sup_{t\in[0,T]}~\sup_{x\in\overline{\mathcal{D}}}~\sup_{u\in[-1,1]}~\Bigl(\abs{\partial_u^2\dr(t,x,u)}+\abs{\partial_u^2\di(t,x,u)} \Bigr)\le C(T).
    \end{equation*}
    \end{subassumptions}
\end{assumption}

\begin{remark}\label{rmk-generalisations2}
Note that, instead of the interval $[-1,1]$, one could consider initial values and solutions with values in more general domains $[L,R]$, where the condition $L\leq0\leq R$ is required to ensure compatibility with the homogeneous Dirichlet boundary conditions. This would require Assumption~\ref{ass-ic} and~\ref{ass-driftdiff} to be modified. To simplify the presentation we only consider the case $L=-1$ and $R=1$.
\end{remark}

Under Assumption~\ref{ass-driftdiff}, we can factorise the drift and diffusion coefficients $\dr$ and $\di$ as specified in the following lemma, which is a direct consequence of~\cite[Lemma~2]{manu}. To this end, we define the function $\sigma(u)=(u-1)(u+1)$ for $u\in[-1,1]$.

\begin{lemma}\label{la-driftdiffdecomp}
    Under Assumption~\ref{ass-driftdiff}, there exist continuous mappings $\drf,\dif\colon[0,\infty)\times\overline{\mathcal{D}}\times[-1,1]\to\R$ such that for all $(t,x,u)\in[0,\infty)\times\overline{\mathcal{D}}\times[-1,1]$, one has
    \begin{equation*}
        \dr(t,x,u) = \drf(t,x,u)\sigma(u)\quad\text{and}\quad \di(t,x,u) = \dif(t,x,u)\sigma(u).
    \end{equation*}
	In addition, for all $T\in(0,\infty)$, there exists $C(T)\in(0,\infty)$ such that 
	\begin{equation*}
	\sup_{t\in[0,T]}~\sup_{x\in\overline{\mathcal{D}}}~\sup_{u\in[-1,1]}~\Bigl(\abs{\drf(t,x,u)}+\abs{\dif(t,x,u)} \Bigr)\le C(T).
	\end{equation*}
\end{lemma}

\begin{proof}
For all $t\in[0,\infty)$ and $x\in\overline{\mathcal{D}}$, the mappings $\dr(t,x,\cdot)$ and $\di(t,x,\cdot)$ are of class $\mathcal{C}^2$ owing to Assumption~\ref{ass-driftdiff-regu}, and satisfy Assumption~\ref{ass-driftdiff-vanish}. Therefore, owing to the proof of~\cite[Lemma~2]{manu}, the factorisation is verified for the mappings $\drf$ and $\dif$ given as follows: for all $(t,x,u)\in [0,\infty)\times\overline{\mathcal{D}}\times[-1,1]$, one has
\begin{align*}
\drf(t,x,u)&=\int_0^1\int_0^1 \partial_u^2\dr\bigl(t,x,\zeta\eta u+(1-\zeta)\eta+\eta-1\bigr)\eta\dd \zeta\dd \eta\\
\dif(t,x,u)&=\int_0^1\int_0^1 \partial_u^2\di\bigl(t,x,\zeta\eta u+(1-\zeta)\eta+\eta-1\bigr)\eta\dd \zeta\dd \eta.
\end{align*}
The continuity of the mappings $\drf$ and $\dif$ and the uniform upper bound on $\drf$ and $\dif$ are direct consequences of Assumptions~\ref{ass-driftdiff-regt} and~\ref{ass-driftdiff-regu} respectively.
\end{proof}

\subsection{Heat kernel}\label{sec-HK}

In this subsection, we introduce the heat kernel $G_d$ and state some of its properties that will be useful for the definition and analysis of mild solutions and for the numerical analysis presented below. Some preliminary notation needs to be introduced. For all $k\in\N$ and $\xi\in[0,1]$, set $e_k(\xi) = \sqrt{2}\sin(k\pi\xi)$. For all $\uk = \left(k_1,\ldots,k_d\right)\in\N^d$ and $x = \left(x_1,\ldots,x_d\right)\in\overline{\mathcal{D}}$, set $e_\uk(x) = \prod_{j=1}^d e_{k_j}(x_j)$.
The $d$-dimensional heat kernel $G_d$ with homogeneous Dirichlet boundary conditions, see for instance~\cite[Theorem 2.1.4]{MR990239}, is defined as
\begin{equation}\label{eq-HK}
    G_d(t,x,y) = \sum_{\uk\in\N^d}\exp\left(-\norm{\uk}^2\pi^2t\right)e_\uk(x)e_\uk(y), \quad \forall~t\in(0,\infty), ~x,y\in\overline{\mathcal{D}}.
\end{equation}

Lemma~\ref{la-propHK} provides properties of the heat kernel $G_d$ defined in equation~\eqref{eq-HK}.

\begin{lemma}\label{la-propHK}
    The heat kernel $G_d$, defined by~\eqref{eq-HK}, satisfies the following properties:
    \begin{sublemmas}
        \item\label{la-propHK-bdd} For all $t\in(0,\infty)$ and $x,y\in\overline{\mathcal{D}}$, one has
        \begin{equation*}
        G_d(t,x,y) \geq 0\quad\text{and}\quad \sup_{t\in(0,\infty)}~\sup_{x\in\overline{\mathcal{D}}}~\int_{\mathcal{D}}G_d(t,x,y)\dd y \leq 1.
        \end{equation*}
        \item\label{la-propHK-semi} For all $s,t\in(0,\infty)$ and $x,y\in\overline{\mathcal{D}}$, one has
        \begin{equation*}
            G_d(t+s,x,y) = \int_\mathcal{D}G_d(t,x,z)G_d(s,z,y)\dd z.
        \end{equation*}
        \item\label{la-propHK-reg} There exists $C_d\in(0,\infty)$ such that for all $\alpha\in[0,1]$ and $0<s<t<\infty$, one has
        \begin{equation*}
            \underset{x\in\overline{\mathcal{D}}}\sup~\int_\mathcal{D}\bigl|G_d(t,x,y) - G_d(s,x,y)\bigr|\dd y \leq C_d\frac{\abs{t-s}^\alpha}{s^\alpha}.
        \end{equation*}
    \end{sublemmas}
\end{lemma}
\begin{proof}
    The properties~\textit{(i)} and~\textit{(ii)} are satisfied in dimension $d=1$, and in arbitrary dimension $d$ they then follow from the decomposition $G_d(t,x,y) = \prod_{j=1}^dG_1(t,x_j,y_j)$, see~\cite[Proposition 3.3.1]{DalangSanz2026} and~\cite{brehier2026} for instance.
    To prove the property~\textit{(iii)}, note that by e.\,g.~\cite[Lemma~3]{brehier2026} there exist $C_d,c_d\in(0,\infty)$ such that
    \begin{equation}\label{eq-HKregt}
        \bigl|G_d(t,x,y) - G_d(s,x,y)\bigr| \leq C_d\frac{\abs{t-s}^\alpha}{s^\alpha}\left( s^{-\frac{d}{2}}e^{-c_d\frac{\norm{y-x}^2}{2s}} + t^{-\frac{d}{2}}e^{-c_d\frac{\norm{y-x}^2}{2t}}\right),
    \end{equation}
    and that for any $c,\gamma\in(0,\infty)$ and $x\in\overline{\mathcal{D}}$, one has
    \begin{equation}\label{eq-gaussint}
        \gamma^\frac{-d}{2}\int_\mathcal{D}\exp\left(-c\frac{\norm{y-x}^2}{\gamma}\right)\dd y \leq \gamma^\frac{-d}{2}\int_{\R^d}\exp\left(-c\frac{\norm{z}^2}{\gamma}\right)\dd z = \left(\frac{\pi}{c}\right)^{\frac{d}{2}}.
    \end{equation}
    Property~\textit{(iii)} then follows by combining equations~\eqref{eq-HKregt} and~\eqref{eq-gaussint}.
\end{proof}

\subsection{Mild solutions and auxiliary SPDE}\label{sec-sol}

In this subsection, we present the notion of mild solutions to the SPDE~\eqref{eq-SPDE} and a well-posedness result.
\begin{defi}\label{defi-mild}
Let $T\in(0,\infty)$ and let Assumptions~\ref{ass-ic} and~\ref{ass-driftdiff} be satisfied.
A continuous adapted process $\left(u(t,x)\right)_{t\in[0,T], x\in\overline{\mathcal{D}}}$ is a mild solution to the SPDE~\eqref{eq-SPDE} if almost surely one has $u(0,x)=u_0(x)$ for all $x\in\overline{\mathcal{D}}$, if
    \begin{equation}\label{eq-mildsolDP}
        \PP\bigg(u(t,x)\in[-1,1], \forall t\in[0,T], \forall x\in\overline{\mathcal{D}}\bigg) = 1,
    \end{equation}
and if almost surely for all $t\in(0,T]$ and $x\in\overline{\mathcal{D}}$ one has
\begin{equation}
    \begin{split}\label{eq-mildsol}
    u(t,x) &= \int_\mathcal{D}G_d(t,x,y)u_0(y)\dd y + \int_0^t\int_\mathcal{D}G_d(t-s,x,y)\dr\left(s,y,u(s,y)\right)\dd y\dd s \\
    &\quad+ \int_0^t\int_\mathcal{D}G_d(t-s,x,y)\di\left(s,y,u(s,y)\right)\dd y\dd\beta(s).
    \end{split}
\end{equation}
\end{defi}

One of the main results of this article is the existence and uniqueness of a mild solution to the SPDE~\eqref{eq-SPDE}.
\begin{proposition}\label{prop-exuniqdp}
    Under Assumptions~\ref{ass-ic} and~\ref{ass-driftdiff}, there exists a unique mild solution $u$ to the SPDE~\eqref{eq-SPDE}. 
\end{proposition}

The proof of Proposition~\ref{prop-exuniqdp} is given in Section~\ref{sec-proofexuniqdp}. The proof of uniqueness follows from a standard application of the Gr\"onwall inequality, whereas the proof of existence is more challenging: since the drift and diffusion coefficients $\dr$ and $\di$ are defined only on $[0,\infty)\times\overline{\mathcal{D}}\times [-1,1]$, it is necessary to prove the almost sure preservation of
the domain $[-1,1]$ from~\eqref{eq-mildsolDP} in order to make sense of the mild formulation~\eqref{eq-mildsol}. To establish the existence of a mild solution, we first introduce the modified version~\eqref{eq-SPDEstar} of the SPDE~\eqref{eq-SPDE} with extensions $\drstar$ and $\distar$ of the coefficients $\dr$ and $\di$ defined on $[0,\infty)\times\overline{\mathcal{D}}\times \R$. The mappings $\drstar$ and $\distar$ satisfy usual regularity conditions, such as global Lipschitz continuity with respect to the third variable, to ensure the existence of a unique mild solution. Second, we introduce a domain preserving time integrator and we prove its convergence to the solution of the modified SPDE~\eqref{eq-SPDEstar}. Finally, letting the time-step size go to $0$ then shows that the mild solution of~\eqref{eq-SPDEstar} satisfies the condition~\eqref{eq-mildsolDP}, and is thus a mild solution to the SPDE~\eqref{eq-SPDE} in the sense of Definition~\ref{defi-mild}.

There is no unique way to define the extensions $\drstar$ and $\distar$ of the coefficients $\dr$ and $\di$. The main requirements are to ensure that for all $t\in[0,\infty)$ and $x\in\overline{\mathcal{D}}$ one has
\[
\drstar(t,x,u)=\dr(t,x,u)\quad\text{and}\quad\distar(t,x,u)=\di(t,x,u),\qquad \forall~u\in[-1,1],
\]
and that the mappings $\drstar(t,x,\cdot)$ and $\distar(t,x,\cdot)$ are (uniformly) globally Lipschitz continuous on $\R$. In this article, the extensions $\drstar$ and $\distar$ of $\dr$ and $\di$ are defined as follows: for all $t\in[0,\infty)$, $x\in\overline{\mathcal{D}}$ and $u\in\R$, set 
\begin{equation}\label{eq-starcoeff}
    \drstar(t,x,u)=\mathds{1}_{|u|\leq1}\dr(t,x,u) + \mathds{1}_{|u|>1}\sin(\pi u) \quad \text{and} \quad \distar(t,x,u)=\mathds{1}_{|u|\leq1}\di(t,x,u) + \mathds{1}_{|u|>1}\sin(\pi u).
\end{equation}
The modified version of the SPDE~\eqref{eq-SPDE} is then defined as the following SPDE with globally Lipschitz coefficients:
\begin{equation}\label{eq-SPDEstar}
    \begin{cases}
        \dd \ustar(t,x) = \left( \Delta \ustar(t,x) + \drstar (t,x,\ustar(t,x)) \right)\dd t + \distar (t,x,\ustar(t,x))\dd\beta(t), \quad t\in(0,T], x\in\mathcal{D}, \\
        \ustar(t,x) = 0, \quad t\in[0,T], x\in\partial\mathcal{D}, \\
        \ustar(0,x) = u_0(x), \quad x\in\overline{\mathcal{D}}.
    \end{cases}
\end{equation}

The coefficients $\drstar$ and $\distar$ of the modified SPDE~\eqref{eq-SPDEstar} then inherit the regularity of the coefficients $\dr$ and $\di$ of the original SPDE~\eqref{eq-SPDE}
stated in Assumption~\ref{ass-driftdiff}. This is made more precise in the following remark.

\begin{remark}\label{rmk-driftdiffstar}
    The modified drift and diffusion coefficients $\drstar,\distar\colon[0,\infty)\times\overline{\mathcal{D}}\times\R\to\R$ of the modified SPDE~\eqref{eq-SPDEstar}, defined by~\eqref{eq-starcoeff}, are continuous and inherit the following properties from $\dr$ and $\di$:
    \begin{subremarks}
        \item\label{rmk-driftdiffstar-agree} For $(t,x,u)\in[0,\infty)\times\overline{\mathcal{D}}\times[-1,1]$, one has $\drstar(t,x,u) = \dr(t,x,u)$ and $\distar(t,x,u) = \di(t,x,u)$. In particular, for all $t\in[0,\infty)$ and $x\in\overline{\mathcal{D}}$, it holds that $\drstar(t,x,\pm1) = \distar(t,x,\pm1) = 0$.
        \item\label{rmk-driftdiffstar-regt} For all $T\in(0,\infty)$, there exists $C(T)\in(0,\infty)$ such that for all $s,t\in[0,T]$, one has
        \begin{equation*}
            \sup_{x\in\overline{\mathcal{D}}}\sup_{u\in\R}\bigg(\abs{\drstar(t,x,u) - \drstar(s,x,u)} + \abs{\distar(t,x,u) - \distar(s,x,u)}\bigg) \leq C(T)\abs{t-s}^\frac{1}{2}.
        \end{equation*}
    \end{subremarks}
\end{remark}
Furthermore, the mappings $\drstar,\distar$ defined by~\eqref{eq-starcoeff} satisfy the following global Lipschitz continuity property: for all $T\in(0,\infty)$, there exists $C(T)\in(0,\infty)$ such that for all $u,v\in\R$ one has
\begin{equation}\label{eq-Lipstar}
\underset{t\in[0,T]}\sup~\underset{x\in\overline{\mathcal{D}}}\sup~\bigg(\abs{\drstar(t,x,u) - \drstar(t,x,v)} + \abs{\distar(t,x,u) - \distar(t,x,v)}\bigg)\leq C(T)\abs{u-v}.
\end{equation}

Note that the mapping $\vartheta=\sin(\pi \cdot)$ in~\eqref{eq-starcoeff} satisfies the condition $\vartheta(\pm 1)=0$, and can thus be factorised by applying~\cite[Lemma~2]{manu}: there exists a continuous mapping $\theta:\R\to\R$ such that
\[
\sin(\pi u)=\vartheta(u)=\theta(u)\sigma(u),\qquad \forall~u\in\R.
\]
More precisely, owing to the proof of~\cite[Lemma~2]{manu}, the expression for $\theta$ is given by
\[
\theta(u)=\int_0^1\int_0^1 \vartheta''\bigl(\zeta\eta u+(1-\zeta)\eta+\eta-1\bigr)\eta\dd \zeta\dd \eta.
\]
Since $\vartheta''=-\pi^2\vartheta=-\pi^2\sin(\pi\cdot)$, the mapping $\theta$ is continuous and bounded on $\R$.

As a result, we can define the mappings $\drfstar,\difstar\colon [0,\infty)\times\overline{\mathcal{D}}\times\R\to\R$ such that for all $t\in[0,\infty)$, $x\in\overline{\mathcal{D}}$ and $u\in\R$ one has 
\begin{equation}\label{eq-afstar}
\drfstar(t,x,u)=\mathds{1}_{|u|\leq1}\drf(t,x,u) + \mathds{1}_{|u|>1}\frac{\sin(\pi u)}{\sigma(u)} \quad \text{and} \quad \difstar(t,x,u)=\mathds{1}_{|u|\leq1}\dif(t,x,u) + \mathds{1}_{|u|>1}\frac{\sin(\pi u)}{\sigma(u)}.
\end{equation}
The mappings $\drfstar$ and $\difstar$ are extensions of the coefficients $\drf$ and $\dif$ introduced in Lemma~\ref{la-driftdiffdecomp}.
Note that these mappings are continuous on $[0,\infty)\times\overline{\mathcal{D}}\times[-1,1]$. In addition, treating the cases $u\in[-1,1]$ and $u\notin[-1,1]$ separately shows that the following upper bounds are satisfied: for all $T\in(0,\infty)$, there exists $C(T)\in(0,\infty)$ such that
	\begin{equation}\label{eq-boundafstar}
	\sup_{t\in[0,T]}~\sup_{x\in\overline{\mathcal{D}}}~\sup_{u\in\R}~\Bigl(\abs{\drfstar(t,x,u)}+\abs{\difstar(t,x,u)} \Bigr)\le C(T).
	\end{equation}

Given arbitrary $T\in(0,\infty)$, we recall that a continuous adapted process $\left(\ustar(t,x)\right)_{t\in[0,T], x\in\overline{\mathcal{D}}}$ is a mild solution (in the classical sense)
to the SPDE~\eqref{eq-SPDEstar} if almost surely one has $u(0,x)=u_0(x)$ for all $x\in\overline{\mathcal{D}}$, and if almost surely for all $t\in(0,T]$ and $x\in\overline{\mathcal{D}}$ one has
\begin{equation}
    \begin{split}
    \ustar(t,x) &= \int_\mathcal{D}G_d(t,x,y)u_0(y)\dd y + \int_0^t\int_\mathcal{D}G_d(t-s,x,y)\drstar\left(s,y,\ustar(s,y)\right)\dd y\dd s \\
    &\quad+ \int_0^t\int_\mathcal{D}G_d(t-s,x,y)\distar\left(s,y,\ustar(s,y)\right)\dd y\dd\beta(s).
    \end{split}
\end{equation}
Applying a standard fixed point argument, using in particular the regularity condition from Remark~\ref{rmk-driftdiffstar-regt} of $\drstar$ and $\distar$ with respect to the first variable and the uniform global Lipschitz continuity property~\eqref{eq-Lipstar} with respect to the third variable, the following well-posedness result for the modified SPDE~\eqref{eq-SPDEstar} is obtained.

\begin{proposition}\label{prop-exuniqmod}
    Under Assumptions~\ref{ass-ic} and~\ref{ass-driftdiff}, assume that the coefficients $\drstar$ and $\distar$ are defined by~\eqref{eq-starcoeff}. For any $T\in(0,\infty)$, there exists a unique mild solution $\ustar$ to the modified SPDE~\eqref{eq-SPDEstar}.
\end{proposition}
Note that the domain preserving scheme introduced in Section~\ref{sec-bp} is shown first to converge to the mild solution~$\ustar$ of the modified SPDE~\eqref{eq-SPDEstar}. This result is then used to show that, almost surely for all $t\in[0,T]$ and $x\in\overline{\mathcal{D}}$, one has $\ustar(t,x)\in[-1,1]$, and that finally $u=\ustar$ is a mild solution to the SPDE~\eqref{eq-SPDE} in the sense of Definition~\ref{defi-mild}.

\section{Domain preserving Lie--Trotter splitting scheme}\label{sec-bp}
In this section, we derive a class of explicit numerical schemes, first for the time discretisation of the modified SPDE~\eqref{eq-SPDEstar}, and ultimately of the SPDE~\eqref{eq-SPDE}.
These time integrators are domain preserving, which means that the numerical approximations remain, for any time-step size, almost surely, in the domain $[-1,1]$ of the SPDE~\eqref{eq-SPDEstar}
(see Proposition~\ref{prop-dplt} below). This geometric property of the numerical schemes is then used to prove   mean-square convergence with order $\frac12-$ to the mild solution of the modified SPDE~\eqref{eq-SPDEstar} (see Theorem~\ref{thm-mainthm}).
In Corollary~\ref{lecorollaire} below, we show the mean-square convergence of the domain preserving schemes with order $\frac12-$ to the mild solution of the original SPDE~\eqref{eq-SPDE}. 
We conclude this section with some numerical experiments illustrating the main qualitative and convergence properties of the proposed domain preserving schemes.

\subsection{A class of explicit domain preserving schemes}\label{subExp}

The objective of this subsection is to present the considered class of explicit domain preserving Lie--Trotter time integrators applied to the SPDE~\eqref{eq-SPDEstar}.
Given a final time $T\in(0,\infty)$ and an integer $M\in\N$, the time-step size is defined as $\dt=T/M$. Moreover, the equidistant time grid is defined by $t_m=m\dt$ for $m\in\{0,1,\ldots,M\}$, and the
Brownian increments are defined as $\dbt_m=\beta(t_{m+1})-\beta(t_m)$ for $m\in\{0,1,\ldots,M-1\}$.

The proposed domain preserving scheme is based on a Lie--Trotter splitting strategy combined with
the work~\cite{manu} on explicit domain preserving integrators for SDEs.
Given the numerical approximation
$u^\star_m\approx\ustar(t_m,\cdot)$ of the solution to the modified SPDE~\eqref{eq-SPDEstar} at the time-grid point $t_m$, for some $m\in\{0,1,\ldots,M-1\}$,
the numerical approximation $u^\star_{m+1}\approx\ustar(t_{m+1},\cdot)$ at the next time-grid point $t_{m+1}=t_m+\dt$ is obtained as follows.
\begin{itemize}
\item First, apply a domain preserving scheme to the auxiliary It\^o SDEs parametrised by $x\in\overline{\mathcal{D}}$ considered on the subinterval $[t_m,t_{m+1}]$:
\begin{equation}\label{eq:scheme-subsystem1}
\left\lbrace
\begin{aligned}
\dd v_{1}(t,x)&=\drstar(t_m,x,v_{1}(t,x))\dd t+\distar(t_m,x,v_{1}(t,x))\dd\beta(t)~,\quad t\in[t_m,t_{m+1}],~x\in\overline{\mathcal{D}},\\
v_{1}(t_m,x)&=u_{m}^\star(x),\quad x\in\overline{\mathcal{D}},
\end{aligned}
\right.
\end{equation}
with initial value $v_{1}(t_m,\cdot)=u_{m}^\star$ at the time-grid point $t_m$. The details of this critical step are given below.
\item Second, solve exactly the homogeneous deterministic heat equation with homogeneous Dirichlet boundary conditions considered on the subinterval $[t_m,t_{m+1}]$:
\begin{equation}\label{eq:scheme-subsystem2}
\left\lbrace
\begin{aligned}
\dd v_{2}(t,x)&=\Delta v_{2}(t,x)\,\dd t~,\quad t\in (t_m,t_{m+1}],~x\in{\mathcal{D}},\\
v_{2}(t,x)&=0~,\quad t\in [t_m,t_{m+1}],~x\in\partial\mathcal{D},\\
v_{2}(t_m,x)&=v_1(t_{m+1},x),\quad x\in\overline{\mathcal{D}},
\end{aligned}
\right.
\end{equation}
with initial value $v_{2}(t_m,\cdot)=v_1(t_{m+1},\cdot)$ at the time-grid point $t_m$, given by the numerical approximation at the time-grid point $t_{m+1}$ of the solution to the auxiliary subsystem~\eqref{eq:scheme-subsystem1} considered in the first step above.
\item Finally, the numerical approximation $u^\star_{m+1}$ at the time-grid point $t_{m+1}$ of the solution
to the modified SPDE~\eqref{eq-SPDEstar} is defined as
\[u^\star_{m+1}(x)=v_{2}(t_{m+1},x),\quad x\in\overline{\mathcal{D}},
\]
and is given by the solution at the time-grid point $t_{m+1}$ to the auxiliary subsystem~\eqref{eq:scheme-subsystem2} considered in the second step above.
\end{itemize}

For the second step, note that the exact solution to the homogeneous heat equation~\eqref{eq:scheme-subsystem2} has the following expression: for all $t\in(t_m,t_{m+1}]$ and $x\in\overline{\mathcal{D}}$, one has
$$
v_{2}(t,x)=\int_{\mathcal{D}}G_d(t-t_m,x,y)v_{2}(t_m,y)\dd y=\int_{\mathcal{D}}G_d(t-t_m,x,y)v_1(t_{m+1},y)\dd y.
$$
For the first step, the numerical approximation of the auxiliary subsystem~\eqref{eq:scheme-subsystem1} parametrised by $x\in\overline{\mathcal{D}}$ is performed with a domain preserving scheme from~\cite{manu}: for all $x\in\overline{\mathcal{D}}$, set
$$
v_{1,m+1}(x)=\Phi\left(\difstar(t_m,x,u^\star_m(x))^2\dt,\drfstar(t_m,x,u^\star_m(x))\dt+\difstar(t_m,x,u^\star_m(x))\dbt_m,u^\star_m(x)\right),
$$
where the integrator $\Phi\colon[0,\infty)\times\R\times[-1,1]\to\R$ is a mapping which satisfies the following assumptions.
\begin{assumption}\label{ass-Phi}
    The mapping $\Phi\colon[0,\infty)\times\R\times[-1,1]\to\R$ is of class $\mathcal{C}^3$ on $[0,\infty)\times\R\times[-1,1]$ and satisfies the following properties:
    \begin{subassumptions}
        \item\label{ass-Phi-dp} $\Phi(s,\gamma,v)\in[-1,1]$ for all $s\geq0,\gamma\in\R,v\in[-1,1]$.
        \item\label{ass-Phi-0} $\Phi(0,0,v)=v$ for all $v\in[-1,1]$.
        \item\label{ass-Phi-dg} $\partial_\gamma\Phi(0,0,v) = \sigma(v)$ for all $v\in[-1,1]$.
        \item\label{ass-Phi-Psi} $\partial_s\Phi(0,0,v) + \frac{1}{2}\partial_\gamma^2\Phi(0,0,v)=0$ for all $v\in[-1,1]$.
        \item\label{ass-Phi-bdd} There exists a constant $C\in(0,\infty)$ such that for all $s\geq0, \gamma\in\R, v\in[-1,1]$:
        \begin{equation*}
            \sum_{\alpha_s + \alpha_\gamma + \alpha_v \leq 3}\abs{\partial_s^{\alpha_s}\partial_\gamma^{\alpha_\gamma}\partial_v^{\alpha_v}\Phi(s,\gamma,v)} \leq C\e^{C(\abs{\gamma} + s)}.
        \end{equation*}
    \end{subassumptions}
\end{assumption}
For ease of presentation, for all $z=(s,\gamma,v)\in[0,\infty)\times\mathbb R\times[-1,1]$, set
\begin{equation}\label{defpsi}
\psi(z) = \partial_s\Phi(z) + \frac{1}{2}\partial_\gamma^2\Phi(z).
\end{equation}

Combining the ingredients described above yields the definition of the numerical approximation $u^\star_{m+1}\approx\ustar(t_{m+1},\cdot)$ of the solution
to the modified SPDE~\eqref{eq-SPDEstar} at the time-grid point $t_{m+1}$, and one obtains the domain preserving Lie--Trotter (DPLT) splitting scheme: for all $m\in\{0,1,\ldots,M-1\}$, set
\begin{numcases}{}
    u^\star_{m+1}(x)=\int_{\mathcal D}G_d(\dt,x,y)\Phi\big(\Ztil^\star_{m+1}(y)\big)\dd y,\quad x\in\overline{\mathcal D},\label{eq-dpltstar}\\
    \Ztil^\star_{m+1}(y)= \left(\difstar(t_m,y,u^\star_m(y))^2\dt,\drfstar(t_m,y,u^\star_m(y))\dt+\difstar(t_m,y,u^\star_m(y))\dbt_m,u^\star_m(y)\right).\label{eq-Ztilmstar}
\end{numcases}

Example~\ref{exa} provides an example of an integrator $\Phi$ which satisfies Assumption~\ref{ass-Phi}, and which is employed in the numerical experiments below. We refer to~\cite{manu} for a derivation of this integrator $\Phi$ and for other examples.
\begin{example}\label{exa}
An example of an integrator $\Phi$ in the DPLT splitting scheme~\eqref{eq-dpltstar} satisfying Assumption~\ref{ass-Phi} is given by
the composition
$$
\Phi(s,\gamma,v)=\varphi(\gamma,\phi(s,v)),
$$
where the mappings $\varphi,\phi$, for all $s\in\R$ and $v\in[-1,1]$, are given by
\begin{equation*}
\phi(s,v)=\frac{v}{\sqrt{v^2+(1-v^2)e^{-2s}}}\quad\text{and}\quad
\varphi(s,v)=\frac{(1-e^{2s})+(1+e^{2s})v}{(1+e^{2s})+(1-e^{2s})v}.
\end{equation*}
Note that $|\phi(s,v)|\le 1$ and $|\varphi(s,v)|\le 1$ for all $s\in\R$ and $v\in[-1,1]$.
\end{example}

\subsection{Main results on the DPLT splitting scheme}\label{subMain}
In this subsection, we first state that the proposed time integrator~\eqref{eq-dpltstar} is domain preserving, i.\,e. that $\ustar_m(x)\in[-1,1]$ for all $x\in\overline{\mathcal{D}}$ and all $m\in\{1,\ldots,M\}$, see Proposition~\ref{prop-dplt}. Second, we state that it converges with mean-square order $\frac12-$ to the solution of the modified SPDE~\eqref{eq-SPDEstar} (with coefficients $\drstar$ and $\distar$), see Theorem~\ref{thm-mainthm}. Finally, as a consequence of the previous results, we show that one can simplify the construction of the scheme by considering the original mappings $\drf$ and $\dif$ instead of their extensions $\drfstar$ and $\difstar$, and that the domain preserving Lie--Trotter scheme~\eqref{eq-dplt} converges with mean-square order $\frac12-$ to the solution of the original SPDE~\eqref{eq-SPDE} (with coefficients $\dr$ and $\di$).
The proofs of the convergence results are given in Sections~\ref{sec-mainproof} and~\ref{sec-prooflecoro}.

First, let us show that the DPLT splitting scheme~\eqref{eq-dpltstar} is domain preserving.
\begin{proposition}\label{prop-dplt}
    Let Assumptions~\ref{ass-ic},~\ref{ass-driftdiff}, and~\ref{ass-Phi} be satisfied. Then, for any time-step size $\dt=T/M$ with $T\in(0,\infty)$ and $M\in\N$, the numerical solution given by the DPLT splitting scheme~\eqref{eq-dpltstar} takes values in the domain $[-1,1]$ almost surely:
    \begin{equation}\label{eq-dpdplt}
        \PP\left(u^\star_m(x)\in[-1,1],~\forall m\in\{0,\ldots,M\},~\forall x\in\overline{\mathcal{D}}\right) = 1.
    \end{equation}
    Furthermore, for all $m\in\{0,\ldots,M\}$, $u^\star_m$ is almost surely continuous on $\overline{\mathcal{D}}$.
\end{proposition}

\begin{proof}
	The property~\eqref{eq-dpdplt} is established by induction with respect to the index $m\in\{0,\ldots,M\}$. First, note that owing to Assumption~\ref{ass-ic}, $u^\star_0$ is continuous on $\overline{\mathcal{D}}$ and one has
	\[
	\underset{x\in\overline{\mathcal{D}}}\sup~|\ustar_0(x)|=\underset{x\in\overline{\mathcal{D}}}\sup~|u_0(x)|\le 1.
	\]
	Next, assume that the property is satisfied for the index $m$, i.\,e. that, almost surely, $u^\star_m$ is continuous on $\overline{\mathcal{D}}$ and one has
	\[
	\underset{x\in\overline{\mathcal{D}}}\sup~|\ustar_m(x)|\le 1.
	\]
	Since the mappings $\drfstar$ and $\difstar$ are continuous on $[0,\infty)\times\overline{\mathcal{D}}\times[-1,1]$ and the integrator $\Phi$ is continuous, the random functions  $y\mapsto \Ztil^\star_{m+1}(y)$ defined by~\eqref{eq-Ztilmstar} and $y\in\overline{\mathcal{D}}\mapsto \Phi\bigl(\Ztil^\star_{m+1}(y)\bigr)$ are continuous almost surely. In addition, applying
	 Assumption~\ref{ass-Phi-dp} on the integrator $\Phi$, almost surely one has
	\[
	\underset{y\in\overline{\mathcal{D}}}\sup~|\Phi\bigl(\Ztil^\star_{m+1}(y)\bigr)|\le 1.
	\]
	As a result, by continuity of the heat kernel, $u^\star_{m+1}$ is almost surely continuous on $\overline{\mathcal{D}}$.
    Finally, applying Lemma~\ref{la-propHK-bdd} to the definition~\eqref{eq-dpltstar} of the numerical solution $\ustar_{m+1}$, one obtains almost surely
	\[
	\underset{x\in\overline{\mathcal{D}}}\sup~|\ustar_{m+1}(x)|\le \underset{x\in\overline{\mathcal{D}}}\sup~\int_{\mathcal D}G_d(\dt,x,y)\dd y~\underset{y\in\overline{\mathcal{D}}}\sup~|\Phi\bigl(\Ztil^\star_{m+1}(y)\bigr)|\le 1.
	\]
	The property~\eqref{eq-dpdplt} is thus established for the index $m+1$. The application of an induction argument then concludes the proof of Proposition~\ref{prop-dplt}.
\end{proof}

Owing to Proposition~\ref{prop-dplt}, in the definition~\eqref{eq-Ztilmstar} of $\Ztil_{m+1}^\star(y)$, the extensions $\drfstar$ and $\difstar$ can be replaced by the original mappings $\drf$ and $\dif$: for all $m\in\{0,\ldots,M-1\}$ and $y\in\overline{\mathcal{D}}$, one has
\[
\Ztil_{m+1}^\star(y)=\left(\dif(t_m,y,u_m^\star(y))^2\dt,\drf(t_m,y,u_m^\star(y))\dt+\dif(t_m,y,u_m^\star(y))\dbt_m,u_m^\star(y)\right).
\]
As a result, the domain preserving Lie--Trotter scheme~\eqref{eq-dpltstar} can be redefined directly at the original level, i.\,e. for the original SPDE~\eqref{eq-SPDE}, with the coefficients $\drf$ and $\dif$. In other words, for all $m\in\{0,1,\ldots,M-1\}$ and $x\in\overline{\mathcal D}$, set
\begin{equation}\label{eq-dplt}
\begin{cases}
u_{m+1}(x)=\int_{\mathcal D}G_d(\dt,x,y)\Phi\big(\Ztil_{m+1}(y)\big)\dd y, \\
\Ztil_{m+1}(y)=\left(\dif(t_m,y,u_m(y))^2\dt,\drf(t_m,y,u_m(y))\dt+\dif(t_m,y,u_m(y))\dbt_m,u_m(y)\right).
\end{cases}
\end{equation}
The numerical solutions given by~\eqref{eq-dpltstar}-\eqref{eq-Ztilmstar} and \eqref{eq-dplt} coincide: almost surely, one has
$$
u_m=u^\star_m\quad\text{and}\quad \Ztil_{m}=\Ztil^\star_{m},\quad \forall~m\in\{0,1,\ldots,M\}.
$$
As a straightforward consequence of Proposition~\ref{prop-dplt}, the DPLT splitting scheme~\eqref{eq-dplt} is domain preserving.
\begin{corollary}\label{corodpdplt}
Let Assumptions~\ref{ass-ic},~\ref{ass-driftdiff}, and~\ref{ass-Phi} be satisfied. Then, for any time-step size $\dt=T/M$ with $T\in(0,\infty)$ and $M\in\N$, the numerical solution given by the DPLT splitting scheme~\eqref{eq-dplt} takes values in the domain $[-1,1]$ almost surely:
	\begin{equation}\label{eq-dpdpltnostar}
        \PP\left(u_m(x)\in[-1,1],~\forall m\in\{0,1,\ldots,M\},~\forall x\in\overline{\mathcal{D}}\right) = 1.
    \end{equation}
\end{corollary}
In practice, the implementation of the DPLT splitting scheme is performed using the formulation~\eqref{eq-dplt}.

We now state the main result of this article: the DPLT splitting scheme~\eqref{eq-dpltstar} converges with mean square order $\frac12-$ to the solution of~\eqref{eq-SPDEstar}. To simplify notation, we introduce the norm $\mse{\cdot} = \E[\abs{\cdot}^2]^{\frac12}$ on the space $L^2(\Omega)$ of square-integrable real-valued random variables.

\begin{theorem}\label{thm-mainthm}
    Let $T\in(0,\infty)$. Consider the mild solution $\bigl(\ustar(t,\cdot))_{t\in[0,T]}$ to the modified SPDE~\eqref{eq-SPDEstar} with coefficients $\drstar$ and $\distar$ defined by~\eqref{eq-starcoeff}, and the numerical solution $\bigl(\ustar_m\bigr)_{0\le m\le M}$ given by the DPLT splitting scheme~\eqref{eq-dpltstar} with time-step size $\dt=T/M$.
    Under Assumptions~\ref{ass-ic},~\ref{ass-driftdiff}, and~\ref{ass-Phi}, the mean-square error converges to $0$ at order $\frac12-$: for all $T\in(0,\infty)$ and $\varepsilon\in(0,\frac12)$, there exists $C_\varepsilon(T)\in(0,\infty)$ such that, for all $\dt=T/M$, one has
    \begin{equation}\label{convDiscont}
        \sup_{0\leq m\leq M}\sup_{x\in\overline{\mathcal{D}}} \mse{u^\star_m(x)-\ustar(t_m,x)}
        \leq C_\varepsilon(T)\dt^{\frac12-\varepsilon}.
    \end{equation}
\end{theorem}
The proof of Theorem~\ref{thm-mainthm} is given in Section~\ref{sec-mainproof}.
Note that an auxiliary continuous-time process $\bigl(\util(t,x)\bigr)_{t\in[0,T],x\in\overline{\mathcal{D}}}$  defined by equation~\eqref{eq-auxproc} below, which coincides with the solution of the DPLT splitting scheme~\eqref{eq-dpltstar} at the time-grid points $t_m$ for $0\le m\le M$, is employed in the proof. In fact, the proof provided in Section~\ref{sec-mainproof} provides a stronger convergence result, which is valid at all time $t\in[0,T]$, see equation~\eqref{eq-timecontresult}: with the same notation as in Theorem~\ref{thm-mainthm}, one has the mean-square error estimates
\begin{equation}\label{convTimeCont}
\sup_{t\in[0,T]}~\sup_{x\in\overline{\mathcal{D}}}
~\mse{\util(t,x)-\ustar(t,x)}
\leq C_\varepsilon(T)\dt^{\frac12-\varepsilon}.
\end{equation}

Finally, verifying that $\util(t,x)\in[-1,1]$ almost surely for all $t\in[0,T]$ and $x\in\overline{\mathcal{D}}$, the convergence result~\eqref{convTimeCont} shows that  $\ustar(t,x)\in[-1,1]$ almost surely for all $t\in[0,T]$ and $x\in\overline{\mathcal{D}}$. As a consequence, as explained in Section~\ref{sec-sol}, one obtains the existence of a mild solution $\bigl(u(t,x)\bigr)_{t\in[0,T],x\in\overline{\mathcal{D}}}$ to the SPDE~\eqref{eq-SPDE} in the sense of Definition~\eqref{defi-mild} (see Proposition~\ref{prop-exuniqdp}). In addition, Corollary~\ref{lecorollaire} states the convergence in the mean-square sense with order $\frac12-$ of the DPLT splitting scheme~\eqref{eq-dplt} to the solution of the SPDE~\eqref{eq-SPDE}.

\begin{corollary}\label{lecorollaire}
Let $T\in(0,\infty)$. Consider the mild solution $\bigl(u(t,\cdot))_{t\in[0,T]}$ to the SPDE~\eqref{eq-SPDE}, and the numerical solution $\bigl(u_m\bigr)_{0\le m\le M}$ given by the DPLT splitting scheme~\eqref{eq-dplt} with time-step size $\dt=T/M$.
    Under Assumptions~\ref{ass-ic},~\ref{ass-driftdiff}, and~\ref{ass-Phi}, the mean-square error converges to $0$ at order $\frac12-$: for all $T\in(0,\infty)$ and $\varepsilon\in(0,\frac12)$, there exists $C_\varepsilon(T)\in(0,\infty)$ such that, for all $\dt=T/M$, one has
    \begin{equation*}
        \sup_{0\leq m\leq M}\sup_{x\in\overline{\mathcal{D}}} \mse{u_m(x)-u(t_m,x)}
        \leq C_\varepsilon(T)\dt^{\frac12-\varepsilon}.
    \end{equation*}
\end{corollary}
The proof of this result is given in Section~\ref{sec-prooflecoro}.

\subsection{Numerical experiments}\label{subNumExp}

In this subsection, we provide numerical illustrations of the qualitative behavior of
the DPLT splitting scheme~\eqref{eq-dplt}\footnote{The codes are available under \url{https://doi.org/10.5281/zenodo.18785719}.} (denoted
by {\upshape DPLT} below) when applied to the
SPDE~\eqref{eq-SPDE} and of its mean-square convergence. In addition, we compare the proposed time integrator with the following classical numerical schemes for SPDEs 
(which are combined with a spatial discretisation, see below):
\begin{itemize}
\item the Euler--Maruyama scheme (under a CFL condition)
$$
u^{\text{EM}}_{m+1}(x)=u^{\text{EM}}_{m}(x)+\dt\Delta u^{\text{EM}}_{m}(x)+\dt\dr(t_m,x,u^{\text{EM}}_{m}(x))+\di(t_m,x,u^{\text{EM}}_{m}(x))\dbt_m,
$$
which will be denoted {\upshape EM} below, see for instance \cite{MR1803132},
\item the semi-implicit Euler--Maruyama scheme,
$$
u^{\text{sEM}}_{m+1}(x)=u^{\text{sEM}}_{m}(x)+\dt\Delta u^{\text{sEM}}_{m+1}(x)+\dt\dr(t_m,x,u^{\text{sEM}}_{m}(x))+\di(t_m,x,u^{\text{sEM}}_{m}(x))\dbt_m,
$$
which will be denoted {\upshape sEM} below, see for instance \cite{MR1699161},
\item the stochastic exponential Euler scheme,
$$
u^{\text{SEXP}}_{m+1}(x)=\e^{\dt\Delta}\left(u^{\text{SEXP}}_{m}(x)+
\dt\dr(t_m,x,u^{\text{SEXP}}_{m}(x))+\di(t_m,x,u^{\text{SEXP}}_{m}(x))\dbt_m\right),
$$
which will be denoted {\upshape SEXP} below, see for instance \cite{MR3308418}.
\end{itemize}

To perform the numerical experiments, we discretise the SPDE~\eqref{eq-SPDE} in space using standard finite differences with mesh size $h=1/N$
for some $N\in\N$ on the spatial grid defined by $x_n=nh$ for $n=0,1,\ldots,N$, see for instance~\cite{MR1683281,MR1644183}. This spatial discretisation is known to be domain preserving, see e.\,g.~\cite[Lemma 5]{ulander2025boundary}.

Let us first illustrate the domain preserving property of {\upshape DPLT} as shown in Proposition~\ref{prop-dplt} in spatial dimension one.
For a fixed spatial discretisation of the SPDE~\eqref{eq-SPDE} with mesh size $h=2^{-8}$,
we apply the four time integrators defined above, with time-step size $\dt=2^{-3}$, on the time interval $[0,20]$.
We consider the initial value $u_0(x)=\sin(4\pi x)$ for $x\in[0,1]$ and several choices for the drift $\dr=\drf\sigma$ and
diffusion $\di=\dif\sigma$ coefficients of the SPDE. We then compute $100$ independent realisations of each numerical approximation and compute the proportion of these realisations which remain in the domain $[-1,1]$. Table~\ref{tabDP} illustrates the fact that all realisations of the DPLT splitting scheme
take values in the domain $[-1,1]$, which validates Proposition~\ref{prop-dplt}. On the contrary, observe that for the classical time integrators some realisations exit the domain $[-1,1]$.

\begin{table}[h]
\begin{center}
\begin{tabular}{|c | c | c | c| c| c|}
  \hline
  $\drf(t,x,u)$ & $\dif(t,x,u)$ & {\upshape DPLT} & {\upshape EM} & {\upshape sEM} & {\upshape SEXP} \\
  \specialrule{.15em}{.05em}{.05em}
  $u+\cos(x)$ & $u+2\sin(t)$ & $100/100$  & $0/100$ & $60/100$ & $99/100$\\
  \hline
  $u$ & $u$ & $100/100$ & $0/100$ & $100/100$ & $100/100$\\
  \hline
  $20u$ & $20u$ & $100/100$ & $0/100$ & $0/100$ & $80/100$\\
  \hline
  $\exp(-2u)$ & $\exp(-2u)$ & $100/100$ & $0/100$ & $16/100$ & $88/100$\\
  \hline
  $\sin(u\pi/2)$ & $2\cos(u\pi/2)$ & $100/100$ & $0/100$ & $51/100$ & $99/100$\\
  \hline
  $-u$ & $2(1-u^2)$ & $100/100$ & $0/100$ & $52/100$ & $99/100$\\
  \hline
\end{tabular}
\end{center}
\caption{Proportion of realisations in the domain $[-1,1]$ for $100$ simulated sample paths for different time integrators:
DPLT splitting scheme~\eqref{eq-dplt} ({\upshape DPLT}), Euler--Maruyama scheme ({\upshape EM}), semi-implicit Euler--Maruyama scheme ({\upshape sEM}), and stochastic exponential Euler scheme ({\upshape SEXP}).}
\label{tabDP}
\end{table}

Next, we illustrate the result of Theorem~\ref{thm-mainthm} on the mean-square rate of convergence of the DPLT splitting scheme. In particular, we discretise the SPDE~\eqref{eq-SPDE} in spatial dimension one, with initial condition $u_0(x) = \sin(2\pi x)$ for $x\in[0,1]$ using finite differences with mesh size $h = 2^{-8}$. We further discretise in time using the time integrators
{\upshape sEM}, {\upshape SEXP}, and {\upshape DPLT}. Note that the standard Euler--Maruyama scheme is not a suitable time integrator in the considered setting and is thus omitted. 
For the {\upshape sEM} and {\upshape SEXP} time integrators, the coefficients $\drf$ and $\dif$ are replaced by their extensions $\drfstar$ and $\difstar$ defined by~\eqref{eq-starcoeff} since, as shown in Table~\ref{tabDP}, the numerical solution may exit the domain $[-1,1]$. This guarantees the mean-square convergence of these integrators, since $\drstar=\drfstar\sigma$ and $\distar=\difstar\sigma$ satisfy the global Lipschitz continuity condition~\eqref{eq-Lipstar}. For the coefficients $\drf$ and $\dif$, we consider two examples: first $\drf(t,x,u) = 0$ and $\dif(t,x,u) = 2e^u$, and second $\drf(t,x,u) = u$ and $\dif(t,x,u) = 5u$.
The maximum mean-square error,
\begin{equation*}
    \sup_{0\leq m\leq M}\sup_{0\leq n\leq N}\mse{u_m(x_n) - u^{\mathrm{ref}}(t_m,x_n)},
\end{equation*}
is displayed in Figure~\ref{fig-strongconv1d} for different values of the time-step size $\dt = 2^{-4},\ldots,2^{-15}$ on the time interval $[0,1]$. The numerical approximation $u_m(x_n)$ refers to the {\upshape sEM}, {\upshape SEXP}, or
{\upshape DPLT} time integrator. To compute the reference solution $u^{\mathrm{ref}}$, we use the DPLT splitting scheme with time-step size $\dt^{\mathrm{ref}} = 2^{-16}$. In this experiment, $200$ samples were used to approximate the expectations. This number is sufficiently large to ensure that the Monte Carlo error is negligible. 
The observed order of convergence is indeed $\frac12$, which confirms the result from Theorem~\ref{thm-mainthm}.

\begin{figure}[h]
   \centering
   \begin{subfigure}{.45\textwidth}
     \centering
     \includegraphics[width=\textwidth]{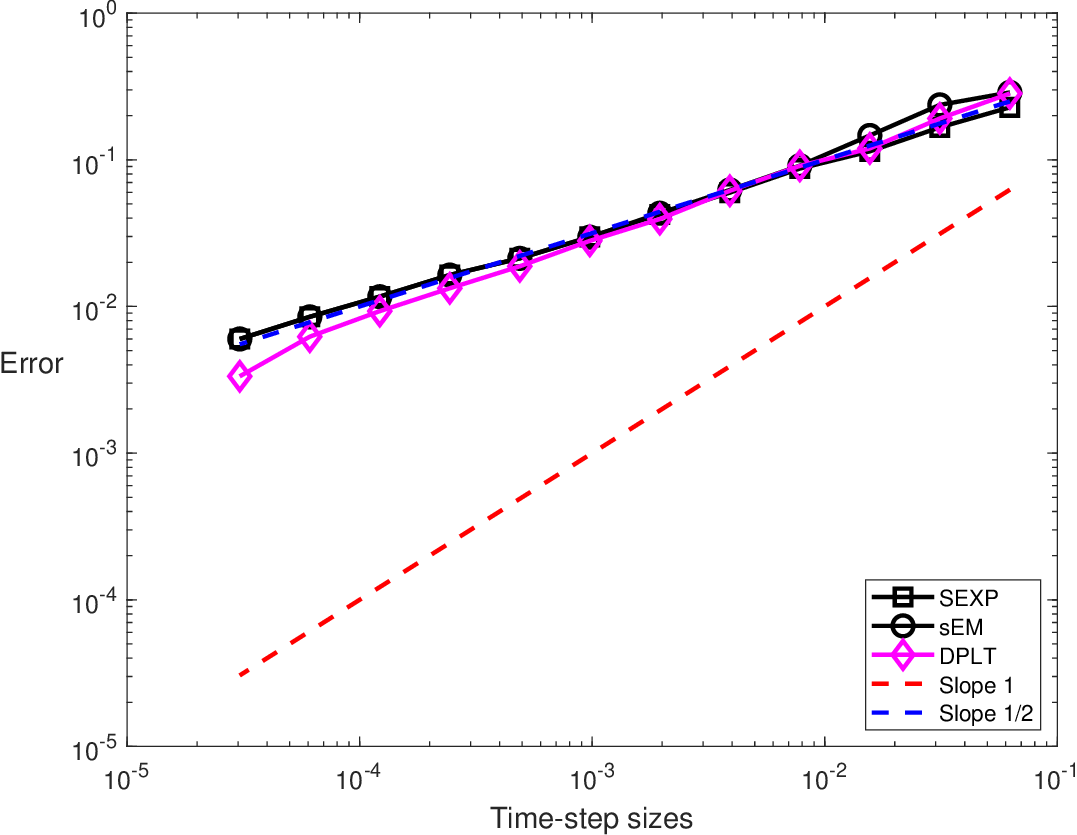}
     \caption{$\drf(t,x,u) = 0$, $\dif(t,x,u) = 2e^u$}
   \end{subfigure}
   \begin{subfigure}{.45\textwidth}
     \centering
     \includegraphics[width=\textwidth]{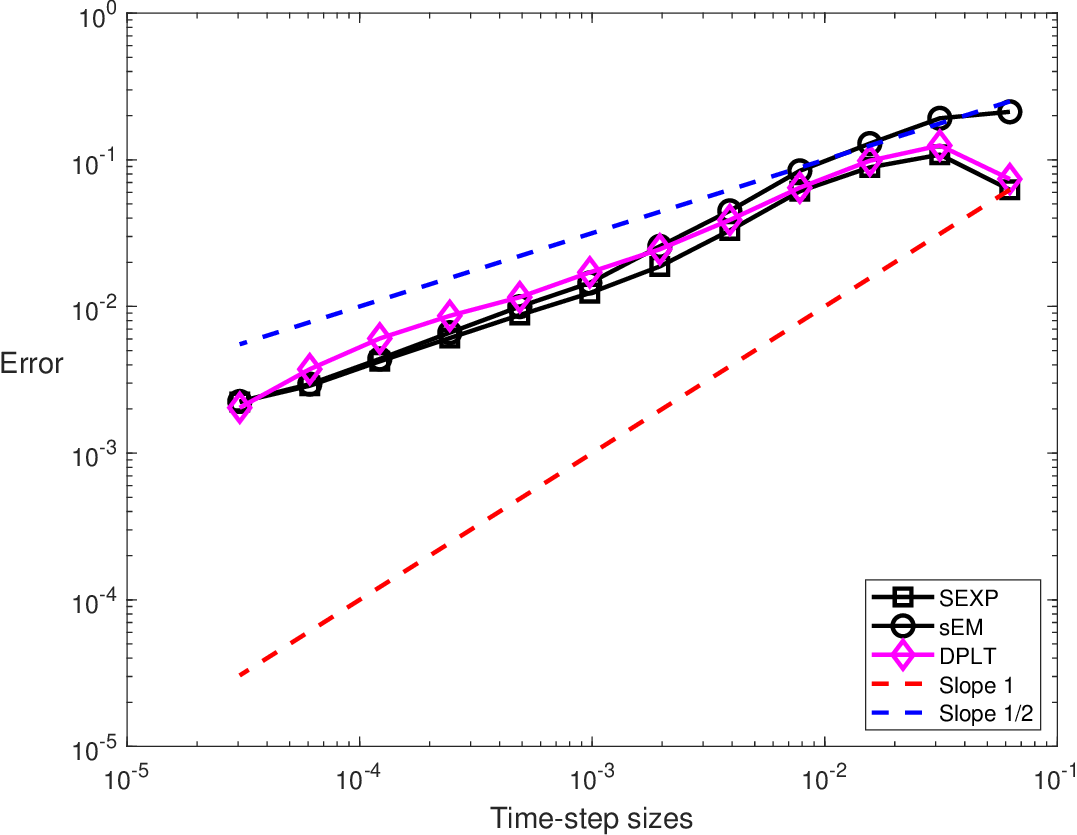}
     \caption{$\drf(t,x,u) = u$, $\dif(t,x,u) = 5u$}
   \end{subfigure}%
   \caption{Maximum mean-square errors for the {\upshape sEM}, {\upshape SEXP}, and {\upshape DPLT} time integrators on the time interval $[0,1]$, for $d=1$, with space-mesh size $h=2^{-8}$.}
   \label{fig-strongconv1d}
\end{figure}

In Figure~\ref{fig-strongconv2d}, the experiment is repeated in spatial dimension $d=2$, on the domain $\mathcal D=(0,1)^2$. The initial condition is given by $u_0(x_1,x_2) = \sin(2\pi x_1)\sin(2\pi x_2)$ for $(x_1,x_2)\in[0,1]^2$ and finite differences with mesh size $h=2^{-6}$ are used for the spatial discretisation. We choose the coefficients given by $\drf(t,x_1,x_2,u) = u^2$ and $\dif(t,x_1,x_2,u) = 2\left(u^3+t+x_1+x_2\right)$. We use $200$ samples to approximate the expectation and consider again the time-step sizes $\dt = 2^{-4},\ldots,2^{-15}$ on the time interval $[0,1]$. For the reference solution $u^{\mathrm{ref}}$ we consider the DPLT splitting scheme with time-step size $\dt^{\mathrm{ref}} = 2^{-16}$.  the order of convergence is confirmed to be $\frac12$, as stated in Theorem~\ref{thm-mainthm}.

\begin{figure}[h]
    \centering
    \includegraphics[width=0.5\textwidth]{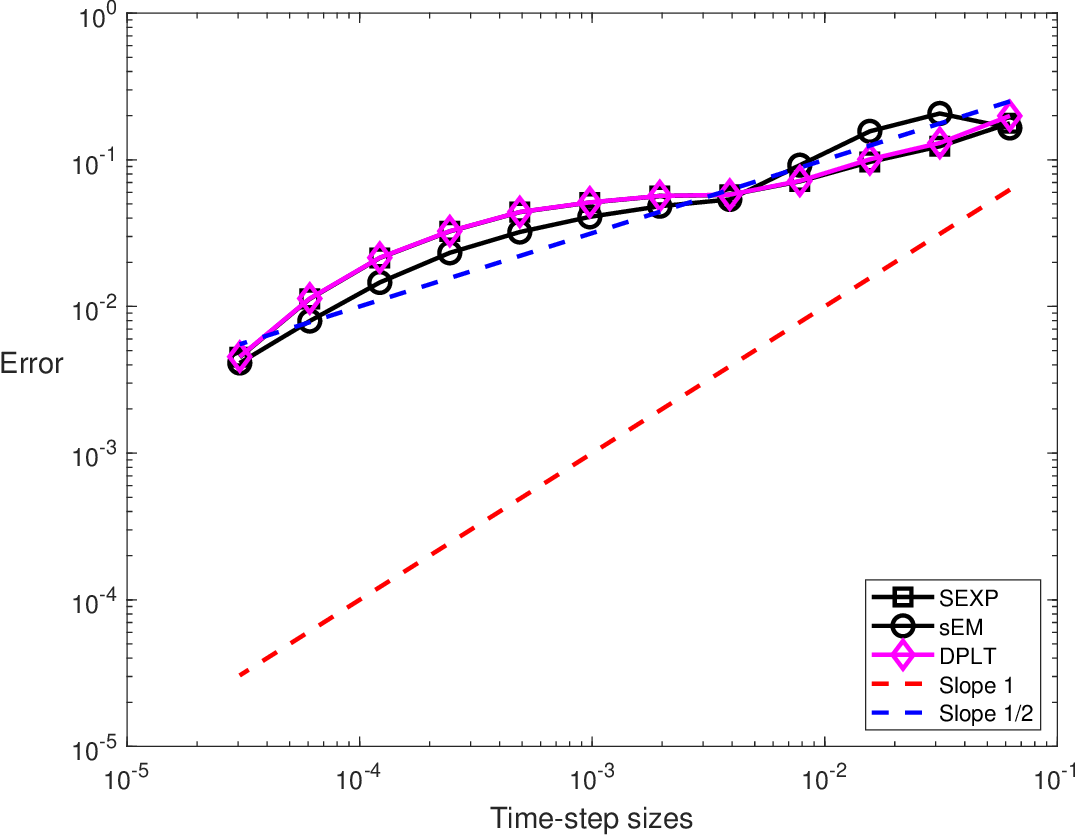}
    \caption{Maximum mean-square errors for the {\upshape sEM}, {\upshape SEXP}, and {\upshape DPLT} time integrators on the time interval $[0,1]$ for $d=2$, with space-mesh size $h=2^{-6}$.}
   \label{fig-strongconv2d}
\end{figure}

In the final set of numerical experiments, we illustrate that the observed order of convergence of the maximum mean-square error from Theorem~\ref{thm-mainthm} is independent of the space-mesh size $h$ used in the finite difference discretisation of space in dimension $d=1$ and $d=2$.

In Figure~\ref{fig-unif1d} we consider the same setting as in Figure~\ref{fig-strongconv1d}, in dimension $d=1$, with space-mesh sizes $h=2^{-4},2^{-6},2^{-8},2^{-10}$. We observe that the maximum mean-square errors are independent of the space-mesh size.

\begin{figure}[h]
    \centering
    \begin{subfigure}{.45\textwidth}
      \centering
      \includegraphics[width=\textwidth]{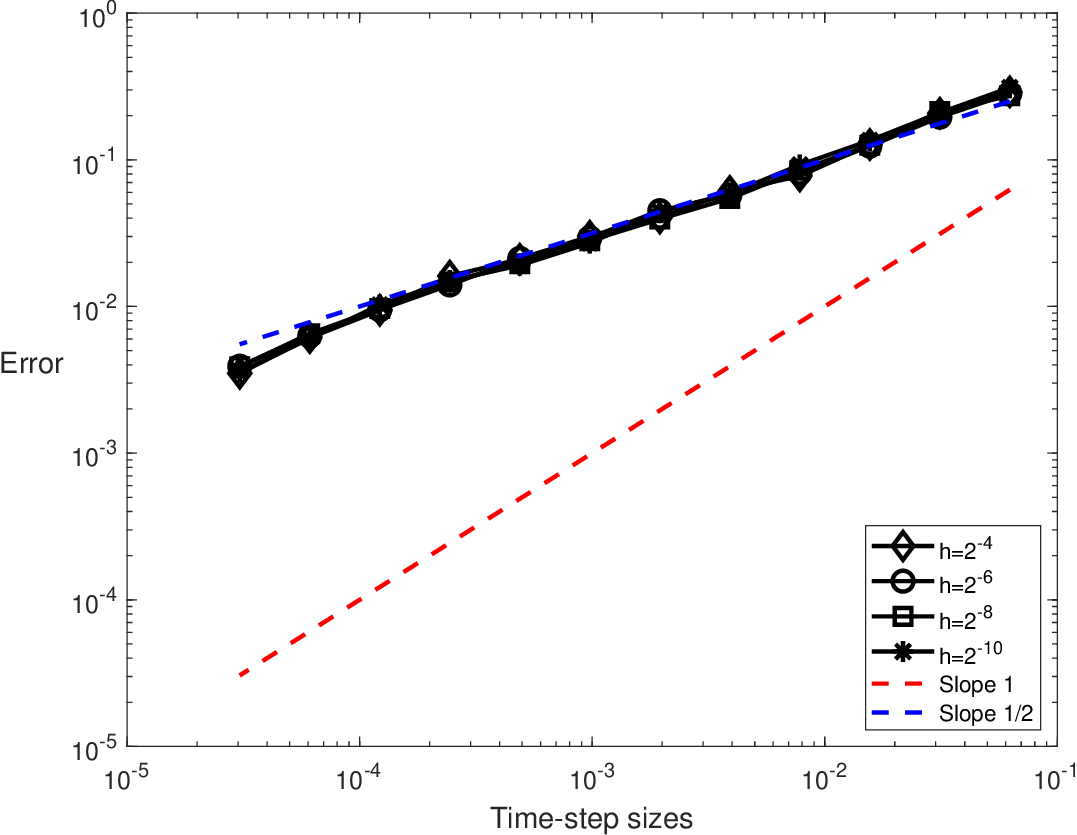}
      \caption{$\drf(t,x,u) = 0$, $\dif(t,x,u) = 2e^u$}
    \end{subfigure}
    \begin{subfigure}{.45\textwidth}
      \centering
      \includegraphics[width=\textwidth]{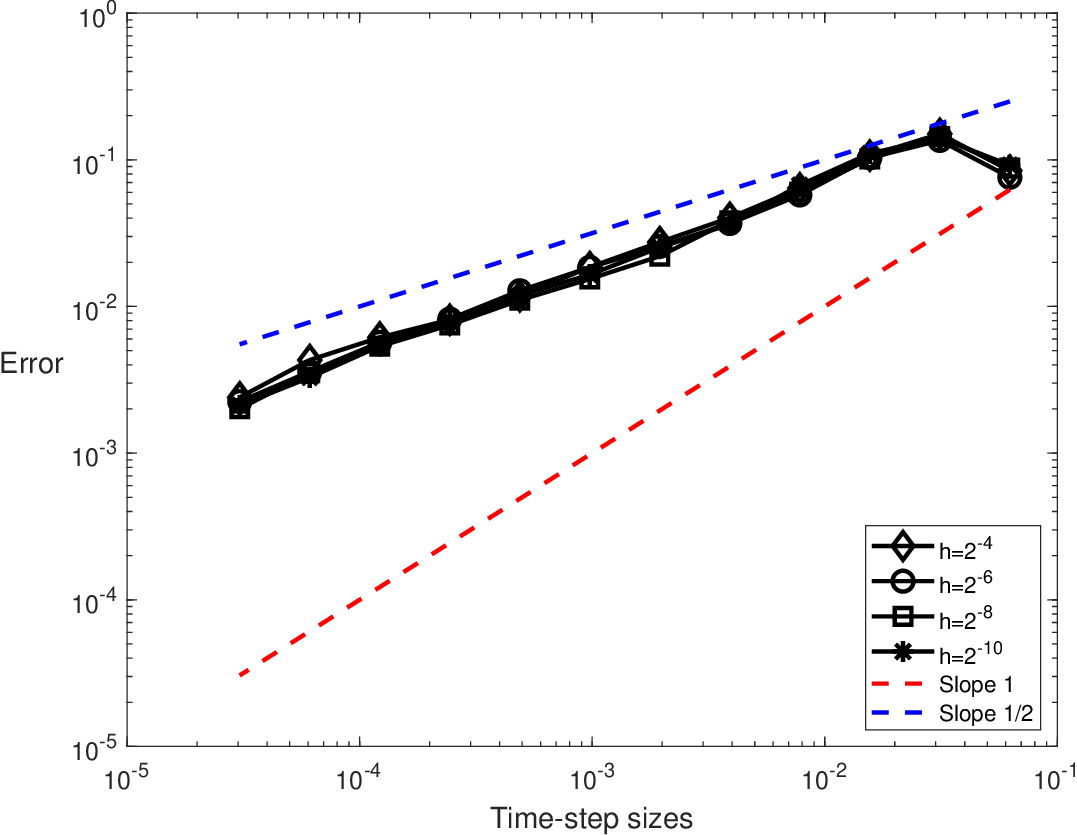}
      \caption{$\drf(t,x,u) = u$, $\dif(t,x,u) = 5u$}
    \end{subfigure}%
    \caption{Maximum mean-square errors for the {\upshape DPLT} time integrator on the time interval $[0,1]$,
    for $d=1$, with several different space-mesh sizes.}
   \label{fig-unif1d}
\end{figure}

Finally, in Figure~\ref{fig-unif2d}, we consider the same setting as in Figure~\ref{fig-strongconv2d}, in dimension $d=2$, for several different values of the space-mesh size $h=2^{-4},2^{-6},2^{-8}$. Again, we observe that the maximum mean-square error is independent of the space-mesh size $h$.

\begin{figure}[h]
    \centering
    \includegraphics[width=0.5\textwidth]{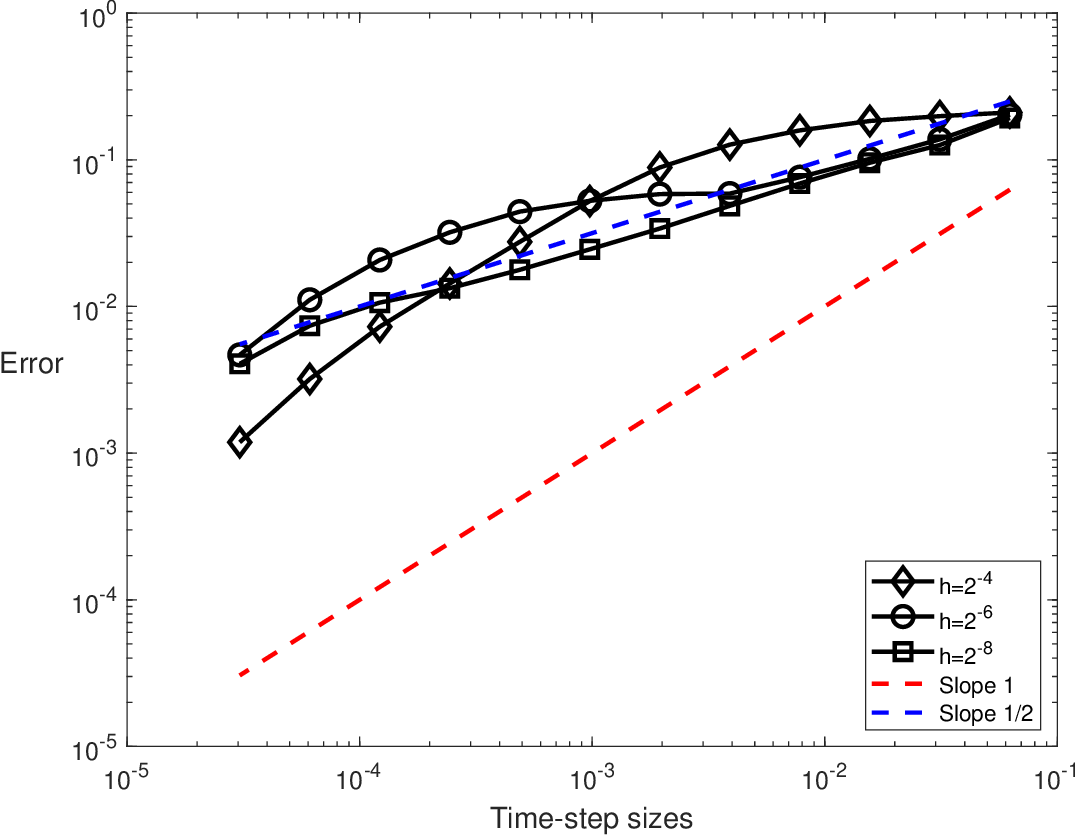}
    \caption{Maximum mean-square errors for the {\upshape DPLT} time integrator on the time interval $[0,1]$,
    for $d=2$, with several different space-mesh sizes.}
   \label{fig-unif2d}
\end{figure}

\section{Proof of the main results}\label{sec-proofs}
In this section we prove the existence and uniqueness of solutions to the SPDE~\eqref{eq-SPDE} as well as the convergence results of this article. In Section~\ref{sec-auxpr}, we introduce an auxiliary process which will be a useful tool in the proof of Theorem~\ref{thm-mainthm}. The proof of this theorem is then given in Section~\ref{sec-mainproof}. In Section~\ref{sec-proofexuniqdp}, we prove existence and uniqueness of a mild solution $u$ to the original SPDE~\eqref{eq-SPDE} (see Proposition~\ref{prop-exuniqdp}).
Finally, in Section~\ref{sec-prooflecoro}, we prove convergence of the DPLT splitting scheme~\eqref{eq-dplt} to this mild solution (see Corollary~\ref{lecorollaire}).

\subsection{Auxiliary processes}\label{sec-auxpr}

Let the value $T\in(0,\infty)$ of the final time be fixed.
As mentioned below the statement of Theorem~\ref{thm-mainthm}, we prove the convergence result~\eqref{convTimeCont} instead of~\eqref{convDiscont}, for an auxiliary continuous-time process $\left(\util(t,x)\right)_{t\in[0,T],x\in\overline{\mathcal{D}}}$ which coincides with the solution $u_m^\star$ given by the DPLT splitting scheme~\eqref{eq-dpltstar} at the time-grid points $t_m$: for all $m\in\{0,\ldots,M\}$ and all $x\in\overline{\mathcal{D}}$ one has $\util(t,x)=u_m^\star(x)$.

The definition of the auxiliary process requires us to introduce additional auxiliary processes. To simplify notation, their dependence with respect to the time-step size $\dt$ is not indicated.

First, let us introduce the process $\left(\Ztil(t,x)\right)_{t\in[0,T],x\in\overline{\mathcal{D}}}$ that is connected to~\eqref{eq-Ztilmstar}. To simplify notation, for all $m\in\{0,\ldots,M\}$ and $x\in\overline{\mathcal{D}}$, set
\begin{equation}\label{Util}
\Util_m(x)=(t_m,x,u^\star_m(x)).
\end{equation}
Then for all $m\in\{0,\ldots,M\}$, $t\in[t_m,t_{m+1})$ and $x\in\overline{\mathcal{D}}$, set
\begin{equation}\label{defZtilda}
\Ztil(t,x)=\left(\difstar(\Util_m(x))^2(t-t_m),\drfstar(\Util_m(x))(t-t_m)+\difstar(\Util_m(x))\bigl(\beta(t)-\beta(t_m)\bigr),u^\star_m(x)\right).
\end{equation}
In addition, set $\Ztil(T,x) = \left(0, 0, u^\star_M(x)\right)$ for all $x\in\overline{\mathcal{D}}$.

Note that by the definition~\eqref{defZtilda} of the process $\Ztil$, for all $m\in\{0,\ldots,M\}$ and all $x\in\overline{\mathcal{D}}$, one has $\Ztil(t_m,x)=\bigl(0,0,u_m^\star(x)\bigr)$. Moreover, recalling the definition~\eqref{eq-Ztilmstar} of $\Ztil^\star_{m+1}$, for all $m\in\{0,\ldots,M\}$ and all $x\in\overline{\mathcal{D}}$ one has
\[
\Ztil^\star_{m+1}(x)=\underset{t\to t_{m+1}^-}\lim~\Ztil(t,x).
\]
A temporal regularity property of the auxiliary process $\Ztil$ is given in Proposition~\ref{prop-Ztil} below.

Next, for any $m\in\{0,\ldots,M-1\}$, we define a process $\left(\vtil_m(t,x)\right)_{t\in[t_m,t_{m+1}],x\in\overline{\mathcal{D}}}$ defined on the time interval $[t_m,t_{m+1}]$: for all $x\in\overline{\mathcal{D}}$, set
\begin{equation}\label{eq-vtil}
    \vtil_m(t,x) = 
    \begin{cases}
        \Phi\big(\Ztil(t,x)\big), \quad t\in[t_m,t_{m+1}), \\
        \Phi\big(\Ztil^\star_{m+1}(x)\big), \quad t=t_{m+1},
    \end{cases}
\end{equation}
with $\Ztil(t,x)$ given by~\eqref{defZtilda} and $\Ztil^\star_{m+1}$ given by~\eqref{eq-Ztilmstar}. Note that the auxiliary process $\vtil_m$ is almost surely continuous on the interval $[t_m,t_{m+1}]$. More precisely, the mapping $(t,x)\in[t_m,t_{m+1}]\times\overline{\mathcal{D}}\mapsto \vtil_m(t,x)$ is almost surely continuous. Since the mapping $\Phi$ is assumed to satisfy Assumption~\ref{ass-Phi-0}, one has
\begin{equation}\label{eq-gridvtil}
    \vtil_m(t_m,x) = u^\star_m(x), \quad \forall~m\in\{0,\ldots,M-1\}, x\in\overline{\mathcal{D}}.
\end{equation}
In addition, by the definition~\eqref{eq-dpltstar} of the DPLT splitting scheme, one obtains the following identity: for all $m\in\{0,\ldots,M-1\}$ and $ x\in\overline{\mathcal{D}}$, one has
\begin{equation}\label{eq-vtildplt}
    u^\star_{m+1}(x) = \int_\mathcal{D}G_d(\dt,x,y)\vtil_m(t_{m+1},y)\dd y.
\end{equation}
Recall the definition~\eqref{defpsi} of the auxiliary mapping $\psi$ associated to the integrator $\Phi$. Applying It\^{o}'s formula for arbitrary $x\in\overline{\mathcal{D}}$ on the interval $[t_m,t_{m+1}]$ for all $m\in\{0,\ldots,M-1\}$, the auxiliary process $t\in[t_m,t_{m+1}]\mapsto \vtil_m(t,x)$ is a solution to the stochastic differential equation
\begin{equation}
    \begin{split}\label{eq-itovtil}
    \dd\vtil_m(t,x) &= \drfstar\left(\Util_m(x)\right)\partial_\gamma\Phi\big(\Ztil(t,x)\big)\dd t \\
    &\quad + \difstar\left(\Util_m(x)\right)\partial_\gamma\Phi\big(\Ztil(t,x)\big)\dd\beta(t) \\
    &\quad + \difstar\left(\Util_m(x)\right)^2\psi\big(\Ztil(t,x)\big)\dd t,
    \end{split}
\end{equation}
where we recall the definition~\eqref{Util} of $\Util_m(x)$ and the definition~\eqref{defpsi} of $\psi$.

For all $t\in[0,T]$, set $\ell(t) = \sup\{t_m: t_m\leq t, 0\le m\le M\}$.
We now provide the definition of the auxiliary process $\util = \left(\util(t,x)\right)_{t\in[0,T],x\in\overline{\mathcal{D}}}$:
for all $x\in\overline{\mathcal{D}}$, set
\begin{equation}\label{eq-auxproc0}
\util(0,x) = u_0(x),
\end{equation}
and for all $t\in(0,T]$, set 
\begin{equation}
    \begin{split}\label{eq-auxproc}
    \util(t,x) &= \int_\mathcal{D}G_d(t,x,y)u_0(y)\dd y \\
    &\quad + \int_0^t\int_\mathcal{D}G_d\left(t-\ell(s),x,y\right)\drfstar\left(\ell(s),y,\util\left(\ell(s),y\right)\right)\partial_\gamma\Phi\big(\Ztil(s,y)\big)\dd y\dd s \\
    &\quad + \int_0^t\int_\mathcal{D}G_d\left(t-\ell(s),x,y\right)\difstar\left(\ell(s),y,\util\left(\ell(s),y\right)\right)\partial_\gamma\Phi\big(\Ztil(s,y)\big)\dd y\dd\beta(s) \\
    &\quad + \int_0^t\int_\mathcal{D}G_d\left(t-\ell(s),x,y\right)\difstar\left(\ell(s),y,\util\left(\ell(s),y\right)\right)^2\psi\big(\Ztil(s,y)\big)\dd y\dd s.
    \end{split}
\end{equation}

To simplify notation in the proofs below, for all $s\in[0,T]$ and $y\in\overline{\mathcal{D}}$, set
	\begin{equation}\label{Utilbis}
	\Util(s,y)=\bigl(s,y,\util(s,y)\bigr).
	\end{equation}

In Proposition~\ref{prop-utilexpr} we provide two key results. First, we show that the auxiliary process $\util$ coincides with the DPLT splitting scheme~\eqref{eq-dpltstar} at the time-grid points $t_m$. Second, we provide an expression of $\util(t,x)$ depending on the auxiliary process $\vtil_m$ defined by~\eqref{eq-vtil}, which is a generalisation of the expression~\eqref{eq-vtildplt} to arbitrary times $t\in[0,T]$.

\begin{proposition}\label{prop-utilexpr}
    Let $\bigl(u_m^\star\bigr)_{0\le m\le M}$ denote the numerical solution given by the  DPLT splitting scheme~\eqref{eq-dpltstar}. 

    For the auxiliary process $\util$ defined by~\eqref{eq-auxproc} one has, almost surely, the identity
    \begin{equation}\label{utilexpr1}
    \util(t_m,x) = u^\star_m(x),\quad\forall~m\in\{0,\ldots,M\},~x\in\overline{\mathcal{D}},
    \end{equation}
	meaning that the auxiliary continuous-time process $\util$ coincides with the numerical scheme at the time-grid points.

	Furthermore, for all $m\in\{0,\ldots,M-1\}$, for the auxiliary process $\vtil_m$ defined by~\eqref{eq-vtil}, one has almost surely the identity
    \begin{equation}\label{utilexpr2}
        \util(t,x) = \int_\mathcal{D}G_d(t-t_m,x,y)\vtil_m(t,y)\dd y,\quad \forall~t\in(t_m,t_{m+1}), ~x\in\overline{\mathcal{D}}.
    \end{equation}
    \end{proposition}

\begin{proof}
    We prove the first claim~\eqref{utilexpr1} by induction on the index $m$.
    First, it follows from~\eqref{eq-auxproc0} that the claim~\eqref{utilexpr1} is verified for the index $m=0$.
    Now, consider a positive integer $m\in\{1,\ldots,M\}$ and assume that the claim~\eqref{utilexpr1} holds for all $k\in\{0,\ldots,m-1\}$, i.\,e. that for all $k\in\{0,\ldots,m-1\}$ and all $x\in\overline{\mathcal{D}}$ one has $\util(t_k,x) = u^\star_k(x)$. Let us prove that the claim~\eqref{utilexpr1} then holds for the integer $m$.
    
    Applying the definition~\eqref{eq-auxproc} of $\util$ at time $t_m$ and the induction assumption $\util(t_k,x) = u^\star_k(x)$ for $k\in\{0,\ldots,m-1\}$, one obtains
    \begin{align*}
        \util(t_m,x) &= \int_\mathcal{D}G_d(t_m,x,y)u_0(y)\dd y \\
        &\quad + \int_0^{t_m}\int_\mathcal{D}G_d\left(t_m-\ell(s),x,y\right)\drfstar\left(\ell(s),y,\util\left(\ell(s),y\right)\right)\partial_\gamma\Phi\big(\Ztil(s,y)\big)\dd y\dd s \\
        &\quad + \int_0^{t_m}\int_\mathcal{D}G_d\left(t_m-\ell(s),x,y\right)\difstar\left(\ell(s),y,\util\left(\ell(s),y\right)\right)\partial_\gamma\Phi\big(\Ztil(s,y)\big)\dd y\dd\beta(s) \\
        &\quad + \int_0^{t_m}\int_\mathcal{D}G_d\left(t_m-\ell(s),x,y\right)\difstar\left(\ell(s),y,\util\left(\ell(s),y\right)\right)^2\psi\big(\Ztil(s,y)\big)\dd y\dd s \\
        &= \int_\mathcal{D}G_d(t_m,x,y)u_0(y)\dd y \\
        &\quad + \sum_{k=0}^{m-1}\int_\mathcal{D}G_d(t_m-t_k,x,y)\int_{t_k}^{t_{k+1}}\drfstar\left(t_k,y,u^\star_k(y)\right)\partial_\gamma\Phi\big(\Ztil(s,y)\big)\dd s\dd y \\
        &\quad + \sum_{k=0}^{m-1}\int_\mathcal{D}G_d(t_m-t_k,x,y)\int_{t_k}^{t_{k+1}}\difstar\left(t_k,y,u^\star_k(y)\right)\partial_\gamma\Phi\big(\Ztil(s,y)\big)\dd\beta(s)\dd y \\
        &\quad + \sum_{k=0}^{m-1}\int_\mathcal{D}G_d(t_m-t_k,x,y)\int_{t_k}^{t_{k+1}}\difstar\left(t_k,y,u^\star_k(y)\right)^2\psi\big(\Ztil(s,y)\big)\dd s\dd y.
    \end{align*}
    For all $k\in\{0,\ldots,m-1\}$, using the stochastic differential equation~\eqref{eq-itovtil} satisfied by the auxiliary mapping $\vtil_k$ on the interval $[t_k,t_{k+1}]$, one obtains the expression
    \begin{equation*}
        \util(t_m,x) = \int_\mathcal{D}G_d(t_m,x,y)u_0(y)\dd y + \sum_{k=0}^{m-1}\int_\mathcal{D}G_d(t_m-t_k,x,y)\Bigl(\vtil_k(t_{k+1},y) - \vtil_k(t_k,y)\Bigr)\dd y.
    \end{equation*}
    Using the semigroup property of the heat kernel from Lemma~\ref{la-propHK-semi}, one obtains
    \begin{align*}
        \util(t_m,x) &= \int_\mathcal{D}G_d(t_m,x,y)u_0(y)\dd y \\
        &\quad + \mathds{1}_{m\geq2}\sum_{k=0}^{m-2}\int_\mathcal{D}G_d(t_m-t_{k+1},x,y)\int_\mathcal{D}G_d(\dt,y,z)\vtil_k(t_{k+1},z)\dd z\dd y \\
        &\quad - \mathds{1}_{m\geq2}\sum_{k=0}^{m-2}\int_\mathcal{D}G_d(t_m-t_k,x,y)\vtil_k(t_k,y)\dd y \\
        &\quad + \int_\mathcal{D}G_d(\dt,x,y)\Bigl(\vtil_{m-1}(t_m,y) - \vtil_{m-1}(t_{m-1},y)\Bigr)\dd y.
    \end{align*}
	Taking into account the properties~\eqref{eq-gridvtil} and~\eqref{eq-vtildplt} which provide expressions of $u^\star_{k}$ and $u^\star_{k+1}$ depending on $\vtil_k(t_k,\cdot)$ and $\vtil_k(t_{k+1},\cdot)$, and applying a telescoping sum argument, one obtains
    \begin{align*}
        \util(t_m,x) &= \int_\mathcal{D}G_d(t_m,x,y)u_0(y)\dd y \\
        &\quad + \mathds{1}_{m\geq2}\sum_{k=0}^{m-2}\left(\int_\mathcal{D}G_d(t_m-t_{k+1},x,y)u^\star_{k+1}(y)\dd y - \int_\mathcal{D}G_d(t_m-t_k,x,y)u^\star_k(y)\dd y\right) \\
        &\quad + u^\star_m(x) - \int_\mathcal{D}G_d(\dt,x,y)u^\star_{m-1}(y)\dd y \\
        &= u^\star_m(x).
    \end{align*}
    As a result, the claim~\eqref{utilexpr1} holds for the index $m$. Applying the recursion argument shows that the first claim~\eqref{utilexpr1} thus holds for all $m\in\{0,\ldots,M\}$.

    Let us now proceed with the proof of the second claim~\eqref{utilexpr2}.
    Let $m\in\{0,\ldots,M-1\}$, $t\in(t_m,t_{m+1}),  x\in\overline{\mathcal{D}}$. 
    Then, using the definition of the auxiliary process $\util$~\eqref{eq-auxproc}, one has
    \begin{align*}
        \util(t,x) &= \int_\mathcal{D}G_d(t,x,y)u_0(y)\dd y \\
        &\quad + \mathds{1}_{m\geq1}\sum_{k=0}^{m-1}\int_\mathcal{D}G_d(t-t_k,x,y)\int_{t_k}^{t_{k+1}}\drfstar\left(t_k,y,u^\star_k(y)\right)\partial_\gamma\Phi\big(\Ztil(s,y)\big)\dd s\dd y \\ 
        &\quad + \mathds{1}_{m\geq1}\sum_{k=0}^{m-1}\int_\mathcal{D}G_d(t-t_k,x,y)\int_{t_k}^{t_{k+1}}\difstar\left(t_k,y,u^\star_k(y)\right)\partial_\gamma\Phi\big(\Ztil(s,y)\big)\dd\beta(s)\dd y \\
        &\quad + \mathds{1}_{m\geq1}\sum_{k=0}^{m-1}\int_\mathcal{D}G_d(t-t_k,x,y)\int_{t_k}^{t_{k+1}}\difstar\left(t_k,y,u^\star_k(y)\right)^2\psi\big(\Ztil(s,y)\big)\dd s\dd y \\
        &\quad + \int_\mathcal{D}G_d(t-t_m,x,y)\int_{t_m}^{t}\drfstar\left(t_m,y,u^\star_m(y)\right)\partial_\gamma\Phi\big(\Ztil(s,y)\big)\dd s\dd y \\
        &\quad + \int_\mathcal{D}G_d(t-t_m,x,y)\int_{t_m}^{t}\difstar\left(t_m,y,u^\star_m(y)\right)\partial_\gamma\Phi\big(\Ztil(s,y)\big)\dd\beta(s)\dd y \\
        &\quad + \int_\mathcal{D}G_d(t-t_m,x,y)\int_{t_m}^{t}\difstar\left(t_m,y,u^\star_m(y)\right)^2\psi\big(\Ztil(s,y)\big)\dd s\dd y.
    \end{align*}
    Using the stochastic differential equation~\eqref{eq-itovtil} on the interval $[t_k,t_{k+1}]$ for the auxiliary function $\vtil_k$ when $k\in\{0,\ldots,m-1\}$, and on the interval $[t_m,t]$ for the auxiliary function $\vtil_m$, one then obtains
    \begin{align*}
        \util(t,x) &= \int_\mathcal{D}G_d(t,x,y)u_0(y)\dd y \\
        &\quad + \mathds{1}_{m\geq1}\sum_{k=0}^{m-1}\int_\mathcal{D}G_d(t-t_k,x,y)\Bigl(\vtil_k(t_{k+1},y) - \vtil_k(t_k,y)\Bigr)\dd y \\
        &\quad + \int_\mathcal{D}G_d(t-t_m,x,y)\Bigl(\vtil_m(t,y)-\vtil_m(t_m,y)\Bigr)\dd y.
    \end{align*}
    By the semigroup property of the heat kernel from Lemma~\ref{la-propHK-semi}, one has
    \begin{align*}
        \util(t,x) &= \int_\mathcal{D}G_d(t,x,y)u_0(y)\dd y \\
        &\quad + \mathds{1}_{m\geq1}\sum_{k=0}^{m-1}\int_\mathcal{D}G_d(t-t_{k+1},x,y)\int_{\mathcal{D}}G_d(\dt,y,z)\vtil_k(t_{k+1},z)\dd z\dd y \\
        &\quad - \mathds{1}_{m\geq1}\sum_{k=0}^{m-1}\int_{\mathcal{D}}G_d(t-t_k,x,y)\vtil_k(t_k,y)\dd y \\
        &\quad + \int_\mathcal{D}G_d(t-t_m,x,y)\bigg(\vtil_m(t,y)-\vtil_m(t_m,y)\bigg)\dd y.
    \end{align*}
    Finally, taking into account the properties~\eqref{eq-gridvtil} and~\eqref{eq-vtildplt} which provide expressions of $u^\star_{k}$ and $u^\star_{k+1}$ depending on $\vtil_k(t_k,\cdot)$ and $\vtil_k(t_{k+1},\cdot)$, and applying a telescoping sum argument, one obtains
    \begin{align*}
        \util(t,x) &= \int_\mathcal{D}G_d(t,x,y)u_0(y)\dd y \\
        &\quad + \mathds{1}_{m\geq1}\sum_{k=0}^{m-1}\left(\int_\mathcal{D}G_d(t-t_{k+1},x,y)u^\star_{k+1}(y)\dd y - \int_\mathcal{D}G_d(t-t_k,x,y)u^\star_k(y)\dd y\right) \\
        &\quad + \int_\mathcal{D}G_d(t-t_m,x,y)\vtil_m(t,y)\dd y - \int_\mathcal{D}G_d(t-t_m,x,y)u^\star_m(y)\dd y \\
        &= \int_\mathcal{D}G_d(t-t_m,x,y)\vtil_m(t,y)\dd y.
    \end{align*}
    The proof of the second claim~\eqref{utilexpr2} is thus completed.
    
    This concludes the proof of Proposition~\ref{prop-utilexpr}.
\end{proof}

Next, Proposition~\ref{prop-Ztil} provides several properties of the auxiliary process $\Ztil$ defined by~\eqref{defZtilda}.
We refer to~\cite{manu} for similar statements in the analysis of domain preserving schemes for stochastic differential equations.
\begin{proposition}\label{prop-Ztil}
    Let Assumptions~\ref{ass-ic},~\ref{ass-driftdiff}, and~\ref{ass-Phi} be satisfied. Let $T\in(0,\infty)$.
    
    For all $p\in\N$, there exists $C_p(T)\in(0,\infty)$ such that 
    for all $m\in\{0,\ldots,M-1\}$ and all $t\in[t_m,t_{m+1})$, one has
    \begin{equation}\label{eq-regZtil}
        \underset{x\in\overline{\mathcal{D}}}\sup~\E\left[\big\|\Ztil(t,x) - \Ztil(t_m,x)\big\|^{2p}\right] \leq C_p(T)\abs{t-t_m}^p.
    \end{equation}
    Moreover, there exists $C(T)\in(0,\infty)$ such that one has
    \begin{equation}\label{eq-momZtil}
        \underset{t\in[0,T]}\sup~\underset{x\in\overline{\mathcal{D}}}\sup~\E\left[\big\|\Ztil(t,x)\big\|^2\right] \leq C(T).
    \end{equation}
    Finally, for all $p\in\N$, there exists $C_p(T)\in(0,\infty)$ such that for all $m\in\{0,\ldots,M-1\}$ and all $t\in[t_m,t_{m+1})$, one has
    \begin{equation}\label{eq-derPhi}
     \sum_{\alpha_s + \alpha_\gamma + \alpha_v \leq 3}~\underset{x\in\overline{\mathcal{D}}}\sup~\E\left[\sup_{c\in[0,1]}\abs{\partial_s^{\alpha_s}\partial_\gamma^{\alpha_\gamma}\partial_v^{\alpha_v}\Phi\big(c\Ztil(t,x) + (1-c)\Ztil(t_m,x)\big)}^{2p}\right] \leq C_p(T).
    \end{equation}
\end{proposition}

\begin{proof}
    Let $m\in\{0,\ldots,M-1\}, t\in[t_m,t_{m+1}), x\in\overline{\mathcal{D}}$. Recall the notation $\Util_m(x)=(t_m,x,u^\star_m(x))$. Then, by the definition~\eqref{defZtilda}
    of the process $\Ztil$, one obtains 
\[
\Ztil(t,x)-\Ztil(t_m,x)=\left(\difstar(\Util_m(x))^2(t-t_m),\drfstar(\Util_m(x))(t-t_m) + \difstar(\Util_m(x))\bigl(\beta(t) - \beta(t_m)\bigr),0\right).
\]
Owing to Proposition~\ref{prop-dplt}, almost surely one has $|u^\star_m(x)|\leq1$ for all $m\in\{0,\ldots,M\}$ and all $x\in\overline{\mathcal{D}}$. Moreover, the mappings $\drfstar$ and $\difstar$ defined by~\eqref{eq-afstar} are bounded on $[0,\infty)\times\overline{\mathcal{D}}\times\R$, see equation~\eqref{eq-boundafstar}. As a result, for all $p\in\N$ there exist $C_p,C_p(T)\in(0,\infty)$ such that one has
    \begin{align*}
        \E\left[\big\|\Ztil(t,x) - \Ztil(t_m,x)\big\|^2\right] &\leq C_p\E\left[\abs{\difstar(t_m,x,u^\star_m(x))^2(t-t_m)}^{2p}\right] \\
        &\quad +C_p\E\left[\abs{\drfstar(t_m,x,u^\star_m(x))(t-t_m)}^{2p} \right] \\
        &\quad +C_p\E\left[\abs{\difstar(t_m,x,u^\star_m(x))\left(\beta(t) - \beta(t_m)\right)}^{2p}\right] \\
        &\leq C_p(T)\left(\abs{t-t_m}^{2p} + \E\bigl[\abs{\beta(t)-\beta(t_m)}^{2p}\bigr]\right).
    \end{align*}
Note that $\E\bigl[\abs{\beta(t)-\beta(t_m)}^{2p}\bigr]=\abs{t-t_m}^p\E[\abs{\mathcal{Z}}^{2p}]$ if $\mathcal{Z}\sim\mathcal{N}(0,1)$ is a standard real-valued Gaussian random variable, and that one has the upper bound $\abs{t-t_m}^{2p}\le T^p\abs{t-t_m}^{p}$ for all $t,t_m\in[0,T]$. As a result, one obtains the temporal regularity property~\eqref{eq-regZtil} for the process $\Ztil$.

Next, let us prove the moment bound~\eqref{eq-momZtil}. Recalling that $\Ztil(t_m,x)=\bigl(0,0,u^\star_m(x)\bigr)$, for all $m\in\{0,\ldots,M-1\}$, $t\in[t_m,t_{m+1})$ and $x\in\overline{\mathcal{D}}$, one has
    \begin{align*}
        \E\bigl[\|\Ztil(t,x)\|^2\bigr] &\leq 2\left(\E\left[\big\|\Ztil(t,x) - \Ztil(t_m,x)\big\|^2\right] + \E\left[\big\|\Ztil(t_m,x)\big\|^2\right]\right) \\
        &= 2\left(\E\left[\big\|\Ztil(t,x) - \Ztil(t_m,x)\big\|^2\right] + \E\left[\abs{u^\star_m(x)}^2\right]\right).
    \end{align*}
Applying the temporal regularity property~\eqref{eq-regZtil} proved above, and recalling that owing to Proposition~\ref{prop-dplt} almost surely one has $|u^\star_m(x)|\leq1$ for all $m\in\{0,\ldots,M\}$ and all $x\in\overline{\mathcal{D}}$, one obtains
    \[
    \E\bigl[\|\Ztil(t,x)\|^2\bigr]\leq C(T)(\dt+1)\leq C(T).
    \]
   The proof of the moment bounds~\eqref{eq-momZtil} is thus completed.

   It remains to prove the upper bounds~\eqref{eq-derPhi}. Let $\alpha_s,\alpha_v,\alpha_\gamma$ be nonnegative integers such that $\alpha_s+\alpha_v+\alpha_\gamma\le 3$, and let $p\in\N$.
   By the definition~\eqref{defZtilda} of $\Ztil$ and owing to Assumption~\ref{ass-Phi-bdd} on the mapping $\Phi$, for all $m\in\{0,\ldots,M-1\}$, $t\in[t_m,t_{m+1})$ and $x\in\overline{\mathcal{D}}$, one has
   \begin{align*}
    &\E\left[\abs{\sup_{c\in[0,1]}\partial_s^{\alpha_s}\partial_\gamma^{\alpha_\gamma}\partial_v^{\alpha_v}\Phi\big(c\Ztil(t,x)+(1-c)\Ztil(t_m,x)\big)}^{2p}\right] \\
    &\quad \leq C_p\E\bigg[\sup_{c\in[0,1]}\e^{C_p\bigl(c\abs{\drfstar(t_m,x,u^\star_m(x))}(t-t_m) + c|\difstar(t_m,x,u^\star_m(x))(\beta(t)-\beta(t_m))| + c\difstar(t_m,x,u^\star_m(x))^2(t-t_m)\bigr)}\bigg]\\
    &\quad \leq C_p\E\bigg[\e^{C_p\bigl(\abs{\drfstar(t_m,x,u^\star_m(x))}(t-t_m) + |\difstar(t_m,x,u^\star_m(x))(\beta(t)-\beta(t_m))| + \difstar(t_m,x,u^\star_m(x))^2(t-t_m)\bigr)}\bigg].
   \end{align*}
	The mappings $\drfstar$ and $\difstar$ defined by~\eqref{eq-afstar} are bounded on $[0,T]\times\overline{\mathcal{D}}\times\R$, see~\eqref{eq-boundafstar}, thus there exists $C_p(T)\in(0,\infty)$ such that
   \begin{align*}
    \E\left[\abs{\sup_{c\in[0,1]}\partial_s^{\alpha_s}\partial_\gamma^{\alpha_\gamma}\partial_v^{\alpha_v}\Phi\big(c\Ztil(t,x)+(1-c)\Ztil(t_m,x)\big)}^{2p}\right] 
    &\leq C_p\E\left[\e^{C_p(T)\bigl((t-t_m) + |\beta(t)-\beta(t_m)|\bigr)}\right]\\
    & \leq C_p(T)\E\left[\e^{C_p(T)\abs{\beta(t)-\beta(t_m)}}\right].
   \end{align*}
   Let $\mathcal{Z}\sim\mathcal{N}(0,1)$ be a standard real-valued Gaussian random variable. Then one has
\[
\E\left[\e^{C_p(T)|\beta(t)-\beta(t_m)|}\right]=\E\left[\e^{C_p(T)\sqrt{t-t_m}|\mathcal{Z}|}\right]\le \E\left[\e^{C_p(T)\sqrt{T}|\mathcal{Z}|}\right].
\]   
Recalling that all the exponential moments of the standard Gaussian random variable $\mathcal{Z}$ are finite, i.\,e. that $\E\left[\e^{c|\mathcal{Z}|}\right]<\infty$ for all $c\in(0,\infty)$, one obtains the upper bounds~\eqref{eq-derPhi}.

This concludes the proof of Proposition~\ref{prop-Ztil}.
\end{proof}

Proposition~\ref{prop-utilprops} provides several important results on the auxiliary process $\util$ which play a crucial role in the proof of Theorem~\ref{thm-mainthm}. Specifically, $\util$ is almost surely continuous and takes values in the domain $[-1,1]$. Moreover, a temporal regularity result in the mean-square sense is also provided.

\begin{proposition}\label{prop-utilprops}
    Let Assumptions~\ref{ass-ic},~\ref{ass-driftdiff}, and~\ref{ass-Phi} be satisfied. Consider the auxiliary process $\util$ defined by~\eqref{eq-auxproc} with time-step size $\dt=T/M$.
    
	First, almost surely, the mapping $(t,x)\in[0,T]\times\overline{\mathcal{D}}\mapsto \util(t,x)$ is continuous.
	In addition, $\util$ takes values in the domain $[-1,1]$ almost surely: one has 
    \begin{equation}\label{eq-dputil}
        \PP\bigg(\util(t,x)\in[-1,1],\quad \forall~t\in[0,T],~x\in\overline{\mathcal{D}}\bigg) = 1.
    \end{equation}
    Finally, for all $\varepsilon\in(0,\frac12)$, there exists $C_\varepsilon(T)\in(0,\infty)$ such that for all for all $m\in\{1,\ldots,M-1\}$ and all $t\in(t_m,t_{m+1})$ one has
    \begin{equation}\label{eq-temputil}
        \underset{x\in\overline{\mathcal{D}}}\sup~\mse{\util(t,x)-\util(t_m,x)}^2
        \leq C_\varepsilon(T)\frac{\abs{t-t_m}^{1-2\varepsilon}}{t_m^{1-2\varepsilon}}.
    \end{equation}
\end{proposition}
\begin{proof}
    {\bf Proof of the continuity property.}
    
    Let $\mathbf{t}\in[0,T]$ and $\mathbf{x}\in\overline{\mathcal{D}}$.
    
    Assume first that $\mathbf{t}\notin\{t_m;~0\le m\le M\}$ is not a time-grid point, i.\,e. that there exists $m\in\{0,\ldots,M-1\}$ such that $\mathbf{t}\in(t_m,t_{m+1})$.
    Recall that the mapping $(t,x)\in[t_m,t_{m+1}]\times\overline{\mathcal{D}}\mapsto \vtil_m(t,x)$ defined by~\eqref{eq-vtil} is almost surely continuous. Owing to Proposition~\ref{prop-dplt} and to Assumption~\ref{ass-Phi-dp}, one has almost surely
    \[
    \underset{t\in[t_m,t_{m+1}]}\sup~\underset{x\in\overline{\mathcal{D}}}\sup~|\vtil_m(t,x)|\le 1.
    \]
    Applying the expression~\eqref{utilexpr2} for $\util$ from Proposition~\ref{prop-utilexpr} and the property of the heat kernel from Lemma~\ref{la-propHK-bdd}, for all $t\in(t_m,t_{m+1})$ and $x\in\overline{\mathcal{D}}$, one has
    \begin{align*}
        \abs{\util(t,x)-\util(\mathbf{t},\mathbf{x})} &= \abs{\int_\mathcal{D}G_d(t-t_m,x,y)\vtil_m(t,y)\dd y - \int_\mathcal{D}G_d(\mathbf{t}-t_m,\mathbf{x},y)\vtil_m(\mathbf{t},y)\dd y} \\
        &\leq \abs{\int_\mathcal{D}\bigl(G_d(t-t_m,x,y) - G_d(\mathbf{t}-t_m,\mathbf{x},y)\bigr)\vtil_m(t,y)\dd y} \\
        &\quad + \abs{\int_\mathcal{D}G_d(\mathbf{t}-t_m,\mathbf{x},y)\bigl(\vtil_m(t,y)-\vtil_m(\mathbf{t},y)\bigr)\dd y} \\
        &\leq \sup_{y\in\overline{\mathcal{D}}}\abs{G_d(t-t_m,x,y)-G_d(\mathbf{t}-t_m,\mathbf{x},y)} \\
        &\quad + \sup_{y\in\overline{\mathcal{D}}}\abs{\vtil_m(t,y)-\vtil_m(\mathbf{t},y)}.
    \end{align*}
    From the joint continuity of the heat kernel $G_d$ on $(0,\infty)\times\overline{\mathcal{D}}^2$ and the almost sure joint continuity of $\vtil_m$ on $[t_m,t_{m+1}]\times\overline{\mathcal{D}}$, one obtains almost surely
    \begin{equation*}
        \lim_{(t,x)\to(\mathbf{t},\mathbf{x})}\abs{\util(t,x)-\util(\mathbf{t},\mathbf{x})}=0.
    \end{equation*}
    The proof of the continuity property is thus completed if $\mathbf{t}\notin\{t_m;~0\le m\le M\}$.

    Assume now that $\mathbf{t}\in\{t_m;~0\le m\le M\}$ is a time-grid point, i.\,e. that there exists $m\in\{0,\ldots,M-1\}$ such that $\mathbf{t}=t_m$. To prove the continuity of $\util$, one needs to consider left and right limits when $t\to \mathbf{t}=t_m$.
    
	On the one hand, consider the limit $t\to t_m^+$, for $m\neq M$.
	Applying the expression~\eqref{utilexpr2} for $\util$ from Proposition~\ref{prop-utilexpr} and the property of the heat kernel from Lemma~\ref{la-propHK-bdd}, for all $t\in(t_m,t_{m+1})$ and $x\in\overline{\mathcal{D}}$, one has
    \begin{align*}
        \abs{\util(t,x)-\util(t_m,\mathbf{x})} &= \abs{\int_\mathcal{D}G_d(t-t_m,x,y)\vtil_m(t,y)\dd y - u^\star_m(\mathbf{x})} \\
        &\leq \abs{\int_\mathcal{D}G_d(t-t_m,x,y)\bigl(\vtil_m(t,y)-\vtil_m(t_m,y)\bigr)\dd y} \\
        &\quad + \abs{\int_\mathcal{D}G_d(t-t_m,x,y)u^\star_m(y)\dd y - u^\star_m(x)} \\
        &\quad + \abs{u^\star_m(x)-u^\star_m(\mathbf{x})} \\
        &\leq \sup_{z\in\overline{\mathcal{D}}}\abs{\vtil_m(t,z)-\vtil_m(t_m,z)} \\
        &\quad + \sup_{z\in\overline{\mathcal{D}}}\abs{\int_\mathcal{D}G_d(t-t_m,z,y)u^\star_m(y)\dd y - u^\star_m(z)} \\
        &\quad + \abs{u^\star_m(x)-u^\star_m(\mathbf{x})}.
    \end{align*}
    Note that $u^\star_m$ is almost surely continuous on $\overline{\mathcal{D}}$, and that almost surely $(t,x)\in[t_m,t_{m+1}]\times\overline{\mathcal{D}}\mapsto \vtil_m(t,x)$ is continuous. As a result, for the first and third terms in the right-hand side above, one has almost surely 
    \[
    \lim_{(t,x)\to(t_m^+,\mathbf{x})}\Bigl(\sup_{z\in\overline{\mathcal{D}}}\abs{\vtil_m(t,z)-\vtil_m(t_m,z)}+\abs{u^\star_m(x)-u^\star_m(\mathbf{x})}\Bigr)=0.
    \]
    Finally, by the joint continuity of the solution to the deterministic heat equation with homogeneous Dirichlet boundary conditions, see e.g.~\cite[Theorem 5.1.11]{MR3012216}, one has almost surely
    \[
    \lim_{(t,x)\to(t_m^+,\mathbf{x})}\sup_{z\in\overline{\mathcal{D}}}\abs{\int_\mathcal{D}G_d(t-t_m,z,y)u^\star_m(y)\dd y - u^\star_m(z)}=0.
    \]
    As a consequence, one has
    \begin{equation*}
    \lim_{(t,x)\to(t_m^+,\mathbf{x})}\abs{\util(t,x)-\util(t_m,\mathbf{x})}=0.
    \end{equation*}
	On the other hand, consider the limit $t\to t_m^-$, for $m\neq 0$. Applying the expression~\eqref{utilexpr2} for $\util$ from Proposition~\ref{prop-utilexpr}, the identity~\eqref{eq-vtildplt} and the property of the heat kernel from Lemma~\ref{la-propHK-bdd}, for all $t\in(t_{m-1},t_{m})$ and $x\in\overline{\mathcal{D}}$ one has
    \begin{align*}
        \abs{\util(t,x)-\util(t_{m},\mathbf{x})} &= \abs{\int_\mathcal{D}G_d(t-t_{m-1},x,y)\vtil_{m-1}(t,y)\dd y - \int_\mathcal{D}G_d(\dt,\mathbf{x},y)\vtil_{m-1}(t_m,y)\dd y} \\
        &\leq \abs{\int_\mathcal{D}\bigl(G_d(t-t_{m-1},x,y) - G_d(t_m-t_{m-1},\mathbf{x},y)\bigr)\vtil_{m-1}(t,y)\dd y} \\
        &\quad + \abs{\int_\mathcal{D}G_d(\dt,\mathbf{x},y)\bigl(\vtil_{m-1}(t,y)-\vtil_{m-1}(t_m,y)\bigr)\dd y} \\
        &\leq \sup_{y\in\overline{\mathcal{D}}}\abs{G_d(t-t_{m-1},x,y)-G_d(t_m-t_m,\mathbf{x},y)} \\
        &\quad + \sup_{y\in\overline{\mathcal{D}}}\abs{\vtil_{m-1}(t,y)-\vtil_{m-1}(t_m,y)}.
    \end{align*}
    From the joint continuity of the heat kernel $G_d$ on $(0,\infty)\times\overline{\mathcal{D}}^2$ and the almost sure joint continuity of $\vtil_{m-1}$ on $[t_{m-1},t_{m}]\times\overline{\mathcal{D}}$, one obtains almost surely
    \begin{equation*}
        \lim_{(t,x)\to(t_m^-,\mathbf{x})}\abs{\util(t,x)-\util(t_m,\mathbf{x})}=0.
    \end{equation*}
    The proof of the continuity property is thus completed if $\mathbf{t}\in\{t_m;~0\le m\le M\}$.

	This concludes the proof of the almost sure continuity of $(t,x)\in[0,T]\times\overline{\mathcal{D}}\mapsto \util(t,x)$.

	{\bf Proof of the domain preservation property~\eqref{eq-dputil}.}
	Owing to the continuity property established above, it suffices to prove that for all $t\in[0,T]$ and all $x\in\overline{\mathcal{D}}$ one has
	\begin{equation}\label{eq:claimDPutil}
	\mathbb{P}\left(\util(t,x)\in[-1,1]\right)=1.
	\end{equation}
	Let $(t,x)\in[0,T]\times\overline{\mathcal{D}}$.
	
	On the one hand, assume that $t\in\{t_m;~0\le m\le M\}$ is a time-grid point. In that case, the claim~\eqref{eq:claimDPutil} follows from the identity~\eqref{utilexpr1} from Proposition~\ref{prop-utilexpr} and the domain preservation property~\eqref{eq-dpdplt} from Proposition~\ref{prop-dplt}.
	
	On the other hand, assume that $t\notin\{t_m;~0\le m\le M\}$ is not a time-grid point. Assume that $t\in(t_m,t_{m+1})$ for some $m\in\{0,\ldots,M-1\}$.
    Applying the expression~\eqref{utilexpr2} from Proposition~\ref{prop-utilexpr} and the boundedness of the integral of the heat kernel given by Lemma~\ref{la-propHK-bdd}, one has almost surely
    \begin{align*}
        \abs{\util(t,x)} &= \abs{\int_\mathcal{D}G_d(t-t_m,x,y)\vtil_m(t,y)\dd y} \\
        &\leq \sup_{y\in\overline{\mathcal{D}}}\abs{\vtil_m(t,y)}\cdot\int_\mathcal{D}G_d(t-t_m,x,y)\dd y \\
        &\leq \sup_{y\in\overline{\mathcal{D}}}\abs{\vtil_m(t,y)}.
    \end{align*}
	Recall that $\vtil_m$ is defined by~\eqref{eq-vtil} and is almost surely continuous on $[t_m,t_{m+1}]\times\overline{\mathcal{D}}$. Applying Assumption~\ref{ass-Phi-dp} on the mapping $\Phi$, one has almost surely
	\[
	\abs{\util(t,x)}\leq\sup_{y\in\overline{\mathcal{D}}}\abs{\vtil_m(t,y)}\le 1.
	\]
The claim~\eqref{eq:claimDPutil} is thus verified in the second case.
    
    This concludes the proof of the domain preservation property~\eqref{eq-dputil}.

	{\bf Proof of the temporal regularity property~\eqref{eq-temputil}.}

	Recall that $\Util(s,y)$ is defined by~\eqref{Utilbis} for all $s\in[0,T]$ and $y\in\overline{\mathcal{D}}$.

    Let $\varepsilon\in(0,\frac12)$ and $m\in\{1,\ldots,M-1\}$. Owing to the definition~\eqref{eq-auxproc} of $\util(t,x)$, for all $t\in(t_m,t_{m+1})$ and all $x\in\overline{\mathcal{D}}$, one has    
    \begin{equation*}
        \mse{\util(t,x) - \util(t_m,x)}^2 \leq 4\bigl(E_1(t,x)+E_2(t,x)+E_3(t,x)+E_4(t,x)\bigr),
    \end{equation*}
    where
    \begin{align*}
        E_1(t,x) &= \abs{\int_\mathcal{D}G_d(t,x,y)u_0(y)\dd y - \int_\mathcal{D}G_d(t_m,x,y)u_0(y)\dd y}^2, \\        
        E_2(t,x) &= \left\|\int_0^t\int_\mathcal{D}G_d(t-\ell(s),x,y)\drfstar\bigl(\Util(\ell(s),y)\bigr)\partial_\gamma\Phi\big(\Ztil(s,y)\big)\dd y\dd s\right. \\
        &\quad \left. - \int_0^{t_m}\int_\mathcal{D}G_d(t_m-\ell(s),x,y)\drfstar\bigl(\Util(\ell(s),y)\bigr)\partial_\gamma\Phi\big(\Ztil(s,y)\big)\dd y\dd s\right\|_{L^2(\Omega)}^2, \\
        E_3(t,x) &= \left\|\int_0^t\int_\mathcal{D}G_d(t-\ell(s),x,y)\difstar\bigl(\Util(\ell(s),y)\bigr)\partial_\gamma\Phi\big(\Ztil(s,y)\big)\dd y\dd\beta(s) \right.  \\
        &\quad \left. - \int_0^{t_m}\int_\mathcal{D}G_d(t_m-\ell(s),x,y)\difstar\bigl(\Util(\ell(s),y)\bigr)\partial_\gamma\Phi\big(\Ztil(s,y)\big)\dd y\dd\beta(s)\right\|_{L^2(\Omega)}^2, \\
        E_4(t,x) &= \left\|\int_0^t\int_\mathcal{D}G_d(t-\ell(s),x,y)\difstar\bigl(\Util(\ell(s),y)\bigr)^2\psi\big(\Ztil(s,y)\big)\dd y\dd s\right. \\
        &\quad \left. - \int_0^{t_m}\int_\mathcal{D}G_d(t_m-\ell(s),x,y)\difstar\bigl(\Util(\ell(s),y)\bigr)^2\psi\big(\Ztil(s,y)\big)\dd y\dd s\right\|_{L^2(\Omega)}^2.
    \end{align*}
    For the first term $E_1(t,x)$, recalling that $u_0(y)\in[-1,1]$ for all $y\in\overline{\mathcal{D}}$ owing to Assumption~\ref{ass-ic}, and applying the temporal regularity property of the heat kernel from Lemma~\ref{la-propHK-reg} with $\alpha=\frac12-\varepsilon$, one obtains for all $t\in(t_m,t_{m+1})$
    \begin{align*}
        \underset{x\in\overline{\mathcal{D}}}\sup~E_1(t,x) &=\underset{x\in\overline{\mathcal{D}}}\sup~ \abs{\int_\mathcal{D}\left(G_d(t,x,y) - G_d(t_m,x,y)\right)u_0(y)\dd y}^2 \\
        &\leq \underset{x\in\overline{\mathcal{D}}}\sup~\left(\int_\mathcal{D}\abs{G_d(t,x,y)-G_d(t_m,x,y)}\cdot\abs{u_0(y)}\dd y\right)^2 \\
        &\leq \underset{x\in\overline{\mathcal{D}}}\sup~\left(\int_\mathcal{D}\abs{G_d(t,x,y)-G_d(t_m,x,y)}\dd y\right)^2 \\
        &\leq C\frac{\abs{t-t_m}^{1-2\varepsilon}}{\abs{t_m}^{1-2\varepsilon}}.
    \end{align*}
	For the other terms $E_j(t,x)$ with $j=1,2,3$, one has the upper bounds
	\[
	E_j(t,x) \leq 2\bigl(E_{j,1}(t,x)+E_{j,2}(t,x)\bigr),
	\]
	where $E_{2,1}(t,x)$ and $E_{2,2}(t,x)$ are defined as
    \begin{align*}
        E_{2,1}(t,x) &= \mse{\int_{t_m}^t\int_\mathcal{D}G_d(t-\ell(s),x,y)\drfstar\bigl(\Util(\ell(s),y)\bigr)\partial_\gamma\Phi\big(\Ztil(s,y)\big)\dd y\dd s}^2, \\
        E_{2,2}(t,x) &= \mse{\int_0^{t_m}\int_\mathcal{D}\left(G_d(t-\ell(s),x,y) - G_d(t_m-\ell(s),x,y)\right)\drfstar\bigl(\Util(\ell(s),y)\bigr)\partial_\gamma\Phi\big(\Ztil(s,y)\big)\dd y\dd s}^2,
    \end{align*}
    $E_{3,1}(t,x)$ and $E_{3,2}(t,x)$ are defined as
    \begin{align*}
        E_{3,1}(t,x) &= \mse{\int_{t_m}^t\int_\mathcal{D}G_d(t-\ell(s),x,y)\difstar\bigl(\Util(\ell(s),y)\bigr)\partial_\gamma\Phi\big(\Ztil(s,y)\big)\dd y\dd\beta(s)}^2, \\
        E_{3,2}(t,x) &= \mse{\int_0^{t_m}\int_\mathcal{D}\left(G_d(t-\ell(s),x,y) - G_d(t_m-\ell(s),x,y)\right)\difstar\bigl(\Util(\ell(s),y)\bigr)\partial_\gamma\Phi\big(\Ztil(s,y)\big)\dd y\dd\beta(s)}^2,
    \end{align*}
    and $E_{4,1}(t,x)$ and $E_{4,2}(t,x)$ are defined as
    \begin{align*}
        E_{4,1}(t,x) &= \mse{\int_{t_m}^t\int_\mathcal{D}G_d(t-\ell(s),x,y)\difstar\bigl(\Util(\ell(s),y)\bigr)^2\psi\big(\Ztil(s,y)\big)\dd y\dd s}^2, \\
        E_{4,2}(t,x) &= \mse{\int_0^{t_m}\int_\mathcal{D}\left(G_d(t-\ell(s),x,y)-G_d(t_m-\ell(s),x,y)\right)\difstar\bigl(\Util(\ell(s),y)\bigr)^2\psi\big(\Ztil(s,y)\big)\dd y\dd s}^2.
    \end{align*}
    Applying the Cauchy--Schwarz inequality to the terms $E_{2,1}(t,x)$ and $E_{4,1}(t,x)$, and the It\^o isometry property to the term $E_{3,1}(t,x)$, one has
    \begin{align*}
        &E_{2,1}(t,x)+E_{3,1}(t,x)+E_{4,1}(t,x) \\
        &\quad \leq C(T)\int_{t_m}^t\mse{\int_\mathcal{D}G_d(t-\ell(s),x,y)\drfstar\bigl(\Util(\ell(s),y)\bigr)\partial_\gamma\Phi\big(\Ztil(s,y)\big)\dd y}^2\dd s \\
        &\qquad + \int_{t_m}^t\mse{\int_\mathcal{D}G_d(t-\ell(s),x,y)\difstar\bigl(\Util(\ell(s),y)\bigr)\partial_\gamma\Phi\big(\Ztil(s,y)\big)\dd y}^2\dd s \\
        &\qquad + C(T)\int_{t_m}^{t}\mse{\int_\mathcal{D}G_d(t-\ell(s),x,y)\difstar\bigl(\Util(\ell(s),y)\bigr)^2\psi\big(\Ztil(s,y)\big)\dd y}^2\dd s.
    \end{align*}
    Next, applying the Minkowski inequality and recalling that the heat kernel $G_d$ takes nonnegative values (see Lemma~\ref{la-propHK-bdd}), one obtains
    \begin{align*}
        &E_{2,1}(t,x)+E_{3,1}(t,x)+E_{4,1}(t,x) \\
        &\quad \leq C(T)\int_{t_m}^t\left(\int_\mathcal{D}G_d(t-\ell(s),x,y)\mse{\drfstar\bigl(\Util(\ell(s),y)\bigr)\partial_\gamma\Phi\big(\Ztil(s,y)\big)}\dd y\right)^2\dd s \\
        &\qquad + \int_{t_m}^t\left(\int_\mathcal{D}G_d(t-\ell(s),x,y)\mse{\difstar\bigl(\Util(\ell(s),y)\bigr)\partial_\gamma\Phi\big(\Ztil(s,y)\big)}\dd y\right)^2\dd s \\
        &\qquad + C(T)\int_{t_m}^t\left(\int_\mathcal{D}G_d(t-\ell(s),x,y)\mse{\difstar\bigl(\Util(\ell(s),y)\bigr)^2\psi\big(\Ztil(s,y)\big)}\dd y\right)^2\dd s.
    \end{align*}
    Recall that the mappings $\drfstar$ and $\difstar$ are bounded on $[0,T]\times\overline{\mathcal{D}}\times\R$, see equation~\eqref{eq-boundafstar}. Applying the moment bounds~\eqref{eq-derPhi} on the derivatives of $\Phi(\Ztil)$ from Proposition~\ref{prop-Ztil}, one thus obtains the upper bounds
    \begin{align*}
        E_{2,1}(t,x)+E_{3,1}(t,x)+E_{4,1}(t,x) &\leq C(T)\int_{t_m}^t\left(\int_\mathcal{D}G_d(t-\ell(s),x,y)\mse{\partial_\gamma\Phi\big(\Ztil(s,y)\big)}\dd y\right)^2\dd s \\
        &\quad + C(T)\int_{t_m}^t\left(\int_\mathcal{D}G_d(t-\ell(s),x,y)\mse{\psi\big(\Ztil(s,y)\big)}\dd y\right)^2\dd s\\
        &\le C(T)\int_{t_m}^t\left(\int_\mathcal{D}G_d(t-\ell(s),x,y)\dd y\right)^2\dd s.
    \end{align*}
    Finally, owing to the boundedness of the integral of the heat kernel from Lemma~\ref{la-propHK-bdd}, one obtains the following inequality: there exists $C(T)\in(0,\infty)$ such that for all $m\in\{1,\ldots,M-1\}$, $t\in(t_m,t_{m+1})$ and $x\in\overline{\mathcal{D}}$, one has
    \begin{align*}
        E_{2,1}(t,x)+E_{3,1}(t,x)+E_{4,1}(t,x)&\leq C(T)\abs{t-t_m}.
    \end{align*}
    The treatment of the terms $E_{j,1}(t,x)$ is thus completed.
    
    Applying the Cauchy--Schwarz inequality to the terms $E_{2,2}(t,x)$ and $E_{4,2}(t,x)$, and the It\^o isometry property to the term $E_{3,2}(t,x)$, one has
    \begin{align*}
        &E_{2,2}(t,x)+E_{3,2}(t,x)+E_{4,2}(t,x) \\
        &\quad \leq C(T)\int_0^{t_m}\mse{\int_\mathcal{D}\left(G_d(t-\ell(s),x,y) - G_d(t_m-\ell(s),x,y)\right)\drfstar\bigl(\Util(\ell(s),y)\bigr)\partial_\gamma\Phi\big(\Ztil(s,y)\big)\dd y}^2\dd s \\
        &\qquad + \int_0^{t_m}\mse{\int_\mathcal{D}\left(G_d(t-\ell(s),x,y) - G_d(t_m-\ell(s),x,y)\right)\difstar\bigl(\Util(\ell(s),y)\bigr)\partial_\gamma\Phi\big(\Ztil(s,y)\big)\dd y}^2\dd s \\
        &\qquad + C(T)\int_0^{t_m}\mse{\int_\mathcal{D}\left(G_d(t-\ell(s),x,y)-G_d(t_m-\ell(s),x,y)\right)\difstar\bigl(\Util(\ell(s),y)\bigr)^2\psi\big(\Ztil(s,y)\big)\dd y}^2\dd s.
    \end{align*}
	Next, applying the Minkowski inequality, one obtains
    \begin{align*}
        &E_{2,2}(t,x)+E_{3,2}(t,x)+E_{4,2}(t,x) \\
        &\quad \leq C(T)\int_0^{t_m}\left(\int_\mathcal{D}\abs{G_d(t-\ell(s),x,y)-G_d(t_m-\ell(s),x,y)}\mse{\drfstar\bigl(\Util(\ell(s),y)\bigr)\partial_\gamma\Phi\big(\Ztil(s,y)\big)}\dd y\right)^2\dd s \\
        &\qquad + \int_0^{t_m}\left(\int_\mathcal{D}\abs{G_d(t-\ell(s),x,y)-G_d(t_m-\ell(s),x,y)}\mse{\difstar\bigl(\Util(\ell(s),y)\bigr)\partial_\gamma\Phi\big(\Ztil(s,y)\big)}\dd y\right)^2\dd s \\ 
        &\qquad + C(T)\int_0^{t_m}\left(\int_\mathcal{D}\abs{G_d(t-\ell(s),x,y)-G_d(t_m-\ell(s),x,y)}\mse{\difstar\bigl(\Util(\ell(s),y)\bigr)^2\psi\big(\Ztil(s,y)\big)}\dd y\right)^2\dd s.
    \end{align*}
    Recall that the mappings $\drfstar$ and $\difstar$ are bounded on $[0,T]\times\overline{\mathcal{D}}\times\R$, see equation~\eqref{eq-boundafstar}. Applying the moment bounds~\eqref{eq-derPhi} on the derivatives of $\Phi(\Ztil)$ from Proposition~\ref{prop-Ztil}, one thus obtains the upper bounds
    \begin{align*}
        &E_{2,2}(t,x)+E_{3,2}(t,x)+E_{4,2}(t,x) \\
        &\quad \leq C(T)\int_0^{t_m}\left(\int_\mathcal{D}\abs{G_d(t-\ell(s),x,y)-G_d(t_m-\ell(s),x,y)}\mse{\partial_\gamma\Phi\big(\Ztil(s,y)\big)}\dd y\right)^2\dd s \\
        &\qquad + C(T)\int_0^{t_m}\left(\int_\mathcal{D}\abs{G_d(t-\ell(s),x,y)-G_d(t_m-\ell(s),x,y)}\mse{\psi\big(\Ztil(s,y)\big)}\dd y\right)^2\dd s \\
        &\quad \leq C(T)\int_0^{t_m}\left(\int_\mathcal{D}\abs{G_d(t-\ell(s),x,y)-G_d(t_m-\ell(s),x,y)}\dd y\right)^2\dd s.
    \end{align*}
	Finally, applying the heat kernel regularity property from Lemma~\ref{la-propHK-reg} with $\alpha=\frac12-\varepsilon\in(0,\frac12)$ and recalling that for all $s\in[0,T]$ one has $\ell(s)\leq s$, one obtains the following inequalities: there exists $C_\varepsilon(T)\in(0,\infty)$ such that for all $m\in\{1,\ldots,M-1\}$, $t\in(t_m,t_{m+1})$ and $x\in\overline{\mathcal{D}}$, one has
    \begin{align*}
        E_{2,2}(t,x)+E_{3,2}(t,x)+E_{4,2}(t,x) &\leq C(T)\abs{t-t_m}^{1-2\varepsilon}\int_0^{t_m}\abs{t_m-\ell(s)}^{-1+2\varepsilon}\dd s \\
        &\leq C(T)\abs{t-t_m}^{1-2\varepsilon}\int_0^{t_m}\abs{t_m-s}^{-1+2\varepsilon}\dd s \\
        &\leq C_\varepsilon(T)\abs{t-t_m}^{1-2\varepsilon}.
    \end{align*}
    
    Gathering the upper bounds for the terms $E_{1}(t,x)$ and $E_{j_1,j_2}(t,x)$ with $j_1=1,2,3$ and $j_2=1,2$, one obtains the following inequality:  there exists $C_\varepsilon(T)\in(0,\infty)$ such that for all $m\in\{1,\ldots,M-1\}$ and $t\in(t_m,t_{m+1})$, one has
    \begin{equation*}
        \underset{x\in\overline{\mathcal{D}}}\sup~\mse{\util(t,x) - \util(t_m,x)}^2\leq C_\varepsilon(T)\frac{\abs{t-t_m}^{1-2\varepsilon}}{t_m^{1-2\varepsilon}}.
    \end{equation*}
    The proof of the inequality~\eqref{eq-temputil} is thus completed.
    
    This concludes the proof of Proposition~\ref{prop-utilprops}.
\end{proof}

\subsection{Proof of Theorem~\ref{thm-mainthm}}\label{sec-mainproof}

Let the value of the final time $T\in(0,\infty)$ and of the arbitrary auxiliary parameter $\varepsilon\in(0,\frac12)$ be fixed.

Recall that $\bigl(\ustar(t,x)\bigr)_{t\in[0,T],x\in\overline{\mathcal{D}}}$ denotes the mild solution to~\eqref{eq-SPDEstar} and that $\bigl(\util(t,x)\bigr)_{t\in[0,T],x\in\overline{\mathcal{D}}}$ is defined by equation~\eqref{eq-auxproc}, and is related to the DPLT splitting scheme~\eqref{eq-dpltstar} by the almost sure equality~\eqref{utilexpr1} from Proposition~\ref{prop-utilexpr}. As a result, one has
\begin{align*}
\sup_{0\leq m\leq M}\sup_{x\in\overline{\mathcal{D}}} \mse{u^\star_m(x)-\ustar(t_m,x)}
&=\sup_{0\leq m\leq M}\sup_{x\in\overline{\mathcal{D}}} \mse{\util(t_m,x)-\ustar(t_m,x)}\\
&\le \sup_{t\in[0,T]}\sup_{x\in\overline{\mathcal{D}}} \mse{\util(t,x)-\ustar(t,x)},
\end{align*}
and below we thus prove the mean-square error estimates~\eqref{convTimeCont} stated below Theorem~\ref{thm-mainthm}.

For all $(t,x)\in[0,T]\times\overline{\mathcal{D}}$, define
\[
e(t,x)=\ustar(t,x)-\util(t,x).
\]
Note that by the definitions of $\ustar$ and $\util$ at time $t=0$ one has
\begin{equation*}
\ustar(0,\cdot)=\util(0,\cdot).
\end{equation*}
Furthermore, owing to Assumption~\ref{ass-Phi-Psi} satisfied by the integrator $\Phi$, one has
\[
\psi\big(\Ztil(\ell(t),x)\big) = \psi\bigl(0,0,\util(\ell(t,x)) \bigr)=0,\quad \forall~t\in[0,T],~x\in\overline{\mathcal{D}}.
\]
Finally, for all $s\in[0,T]$ and $y\in\overline{\mathcal{D}}$, set
\begin{equation}
\Ustar(s,y)=\bigl(s,y,\ustar(s,y)\bigr),
\end{equation}
and recall that $\Util(s,y)$ is defined by~\eqref{Utilbis}.

{\bf Decomposition of the error.}

For all $t\in[0,T]$ and $x\in\overline{\mathcal{D}}$, the error $e(t,x)$ is decomposed as
\begin{equation}\label{onemoreerror}
    \ustar(t,x) - \util(t,x) = e_1(t,x) + e_2(t,x),
\end{equation}
where the error terms $e_1(t,x)$ and $e_2(t,x)$ are defined as
\begin{align*}
    &e_1(t,x) = \int_0^t\int_\mathcal{D}\left[G_d(t-s,x,y)\drstar(\Ustar(s,y)) - G_d\left(t-\ell(s),x,y\right)\drfstar\bigl(\Util(\ell(s),y)\bigr)\partial_\gamma\Phi\big(\Ztil(s,y)\big)\right]\dd y\dd s \\
    &\quad + \int_0^t\int_\mathcal{D}\left[G_d(t-s,x,y)\distar(\Ustar(s,y)) - G_d\left(t-\ell(s),x,y\right)\difstar\bigl(\Util(\ell(s),y)\bigr)\partial_\gamma\Phi\big(\Ztil(s,y)\big)\right]\dd y\dd\beta(s), \\
    &e_2(t,x) = \int_0^t\int_\mathcal{D}G_d\left(t-\ell(s),x,y\right)\difstar\bigl(\Util(\ell(s),y)\bigr)^2\left[\psi\big(\Ztil(s,y)\big) - \psi\big(\Ztil\left(\ell(s),y\right)\big)\right]\dd y\dd s.
\end{align*}
Owing to Assumption~\ref{ass-Phi-dg} satisfied by the integrator $\Phi$, one has $\partial_\gamma\Phi(0,0,v) = \sigma(v)$ for $v\in[-1,1]$. Therefore the first error term $e_1(t,x)$ is further decomposed as
\begin{equation}\label{onemoreerror2}
    e_1(t,x) = e_{1,1}(t,x) + e_{1,2}(t,x) + e_{1,3}(t,x) + e_{1,4}(t,x) + e_{1,5}(t,x),
\end{equation}
where the auxiliary error terms $e_{1,j}(t,x)$ for $j=1,\ldots,5$ are defined as
\begin{align*}
    e_{1,1}(t,x) &= \int_0^t\int_\mathcal{D}G_d(t-s,x,y)\left[\drstar\left(s,y,\ustar(s,y)\right) - \drstar\left(s,y,\util(s,y)\right)\right]\dd y\dd s \\
    &\quad + \int_0^t\int_\mathcal{D}G_d(t-s,x,y)\left[\distar\left(s,y,\ustar(s,y)\right) - \distar\left(s,y,\util(s,y)\right)\right]\dd y\dd\beta(s), \\
    e_{1,2}(t,x) &= \int_0^t\int_\mathcal{D}G_d(t-s,x,y)\left[\drstar\left(s,y,\util(s,y)\right) - \drstar\left(\ell(s),y,\util(s,y)\right)\right]\dd y\dd s \\
    &\quad + \int_0^t\int_\mathcal{D}G_d(t-s,x,y)\left[\distar\left(s,y,\util(s,y)\right) - \distar\left(\ell(s),y,\util(s,y)\right)\right]\dd y\dd\beta(s), \\
    e_{1,3}(t,x) &= \int_0^t\int_\mathcal{D}G_d(t-s,x,y)\left[\drstar\left(\ell(s),y,\util(s,y)\right) - \drstar\left(\ell(s),y,\util\left(\ell(s),y\right)\right)\right]\dd y\dd s \\
    &\quad + \int_0^t\int_\mathcal{D}G_d(t-s,x,y)\left[\distar\left(\ell(s),y,\util(s,y)\right) - \distar\left(\ell(s),y,\util\left(\ell(s),y\right)\right)\right]\dd y\dd\beta(s), \\
    e_{1,4}(t,x) &= \int_0^t\int_\mathcal{D}\left[G_d(t-s,x,y) - G_d\left(t-\ell(s),x,y\right)\right]\drstar\left(\ell(s),y,\util\left(\ell(s),y\right)\right)\dd y\dd s \\
    &\quad + \int_0^t\int_\mathcal{D}\left[G_d(t-s,x,y) - G_d\left(t-\ell(s),x,y\right)\right]\distar\left(\ell(s),y,\util\left(\ell(s),y\right)\right)\dd y\dd\beta(s), \\
    e_{1,5}(t,x) &= \int_0^t\int_\mathcal{D}G_d\left(t-\ell(s),x,y\right)\drfstar\bigl(\Util(\ell(s),y)\bigr)\left[\partial_\gamma\Phi\big(\Ztil(\ell(s),y)\big) - \partial_\gamma\Phi\big(\Ztil\left(s,y\right)\big)\right]\dd y\dd s \\
    &\quad + \int_0^t\int_\mathcal{D}G_d\left(t-\ell(s),x,y\right)\difstar\bigl(\Util(\ell(s),y)\bigr)\left[\partial_\gamma\Phi\big(\Ztil(\ell(s),y)\big) - \partial_\gamma\Phi\big(\Ztil\left(s,y\right)\big)\right]\dd y\dd\beta(s).
\end{align*}

{\bf Treatment of the error term $e_{1,1}(t,x)$.}

Applying the Cauchy--Schwarz inequality, the It{\^o} isometry property, and the Minkowski inequality, and recalling that the heat kernel $G_d$ takes nonnegative values (see Lemma~\ref{la-propHK-bdd}), for the error term $e_{1,1}(t,x)$ one has
\begin{align*}
    \mse{e_{1,1}(t,x)}^2 &\leq C(T) \int_0^t\mse{\int_\mathcal{D}G_d(t-s,x,y)\left[\drstar\left(s,y,\ustar(s,y)\right) - \drstar\left(s,y,\util(s,y)\right)\right]\dd y}^2\dd s \\
    &\quad + \int_0^t\mse{\int_\mathcal{D}G_d(t-s,x,y)\left[\distar\left(s,y,\ustar(s,y)\right) - \distar\left(s,y,\util(s,y)\right)\right]\dd y}^2\dd s \\
    &\leq C(T)\int_0^t\left(\int_\mathcal{D}G_d(t-s,x,y)\mse{\drstar\left(s,y,\ustar(s,y)\right) - \drstar\left(s,y,\util(s,y)\right)}\dd y\right)^2\dd s \\
    &\quad + \int_0^t\left(\int_\mathcal{D}G_d(t-s,x,y)\mse{\distar\left(s,y,\ustar(s,y)\right) - \distar\left(s,y,\util(s,y)\right)}\dd y\right)^2\dd s.
\end{align*}
The mappings $\drstar$ and $\distar$ satisfy the Lipschitz continuity property~\eqref{eq-Lipstar} with respect to the third variable $u\in\R$. Owing to the boundedness of the integral of the heat kernel from Lemma~\ref{la-propHK-bdd}, one thus obtains the following upper bounds: there exists $C(T)\in(0,\infty)$ such that for all $t\in[0,T]$ and $x\in\overline{\mathcal{D}}$ one has
\begin{align*}
    \mse{e_{1,1}(t,x)}^2 &\leq C(T)\int_0^t\left(\int_\mathcal{D}G_d(t-s,x,y)\mse{\ustar(s,y) - \util(s,y)}\dd y\right)^2\dd s \\
    &\leq C(T)\int_0^t\sup_{z\in\overline{\mathcal{D}}}\mse{\ustar(s,z) - \util(s,z)}^2\left(\int_\mathcal{D}G_d(t-s,x,y)\dd y\right)^2\dd s \\
    &\leq C(T)\int_0^t\sup_{z\in\overline{\mathcal{D}}}\mse{\ustar(s,z) - \util(s,z)}^2\dd s.
\end{align*}

{\bf Treatment of the error term $e_{1,2}(t,x)$.}

Applying the Cauchy--Schwarz inequality, the It{\^o} isometry property, and the Minkowski inequality, for the error term $e_{1,2}(t,x)$ one has
\begin{align*}
    &\mse{e_{1,2}(t,x)}^2 \\
    &\quad \leq C(T)\int_0^t\mse{\int_\mathcal{D}G_d(t-s,x,y)\left[\drstar\left(s,y,\util(s,y)\right) - \drstar\left(\ell(s),y,\util(s,y)\right)\right]\dd y}^2\dd s \\
    &\qquad + \int_0^t\mse{\int_\mathcal{D}G_d(t-s,x,y)\left[\distar\left(s,y,\util(s,y)\right) - \distar\left(\ell(s),y,\util(s,y)\right)\right]\dd y}^2\dd s \\
    &\quad \leq C(T)\int_0^t\left(\int_\mathcal{D}G_d(t-s,x,y)\mse{\drstar\left(s,y,\util(s,y)\right) - \drstar\left(\ell(s),y,\util(s,y)\right)}\dd y\right)^2\dd s \\
    &\qquad + \int_0^t\left(\int_\mathcal{D}G_d(t-s,x,y)\mse{\distar\left(s,y,\util(s,y)\right) - \distar\left(\ell(s),y,\util(s,y)\right)}\dd y\right)^2\dd s.
\end{align*}
The mappings $\drstar$ and $\distar$ satisfy the H\"older continuity property given by Remark~\ref{rmk-driftdiffstar-regt} with respect to the first variable $t\in[0,T]$. Owing to the boundedness of the integral of the heat kernel from Lemma~\ref{la-propHK-bdd}, one thus obtains the following upper bounds: there exists $C(T)\in(0,\infty)$ such that for all $t\in[0,T]$ and $x\in\overline{\mathcal{D}}$ one has
\begin{align*}
    \mse{e_{1,2}(t,x)}^2 &\leq C(T)\int_0^t\left(\int_\mathcal{D}G_d(t-s,x,y)\abs{s-\ell(s)}^{\frac{1}{2}}\dd y\right)^2\dd s \\
    &\leq C(T)\dt\int_0^t\left(\int_\mathcal{D}G_d(t-s,x,y)\dd y\right)^2\dd s \\
    &\leq C(T)\dt.
\end{align*}

{\bf Treatment of the error term $e_{1,3}(t,x)$.}

Applying the Cauchy--Schwarz inequality, the It{\^o} isometry property, and the Minkowski inequality, for the error term $e_{1,3}(t,x)$ one has
\begin{align*}
	&\mse{e_{1,3}(t,x)}^2 \\
    &\quad\leq C(T)\int_0^t\mse{\int_\mathcal{D}G_d(t-s,x,y)\left[\drstar\left(\ell(s),y,\util(s,y)\right) - \drstar\left(\ell(s),y,\util\left(\ell(s),y\right)\right)\right]\dd y}^2\dd s \\
    &\qquad + \int_0^t\mse{\int_\mathcal{D}G_d(t-s,x,y)\left[\distar\left(\ell(s),y,\util(s,y)\right) - \distar\left(\ell(s),y,\util\left(\ell(s),y\right)\right)\right]\dd y}^2\dd s \\
    &\quad\leq C(T)\int_0^t\left(\int_\mathcal{D}G_d(t-s,x,y)\mse{\drstar\left(\ell(s),y,\util(s,y)\right) - \drstar\left(\ell(s),y,\util\left(\ell(s),y\right)\right)}\dd y\right)^2\dd s \\
    &\qquad + \int_0^t\left(\int_\mathcal{D}G_d(t-s,x,y)\mse{\distar\left(\ell(s),y,\util(s,y)\right) - \distar\left(\ell(s),y,\util\left(\ell(s),y\right)\right)}\dd y\right)^2\dd s.
\end{align*}
The mappings $\drstar$ and $\distar$ satisfy the Lipschitz continuity property~\eqref{eq-Lipstar} with respect to the third variable $u\in\R$. Owing to the boundedness of the integral of the heat kernel from Lemma~\ref{la-propHK-bdd}, one obtains
\begin{align*}
    \mse{e_{1,3}(t,x)}^2 &\leq C(T)\int_0^t\left(\int_\mathcal{D}G_d(t-s,x,y)\mse{\util(s,y) - \util\left(\ell(s),y\right)}\dd y\right)^2\dd s \\
    &\leq C(T)\int_0^t\sup_{z\in\overline{\mathcal{D}}}\mse{\util(s,z) - \util\left(\ell(s),z\right)}^2\left(\int_\mathcal{D}G_d(t-s,x,y)\dd y\right)^2\dd s \\
    &\leq C(T)\int_0^t\sup_{z\in\overline{\mathcal{D}}}\mse{\util(s,z) - \util\left(\ell(s),z\right)}^2\dd s.
\end{align*}
To obtain an upper bound on the right-hand side of the inequality above, one would like to apply the inequality~\eqref{eq-temputil} from Proposition~\ref{prop-utilprops}, however it is valid only if $\ell(s)\ge \dt$. To deal with this issue, one writes
\begin{align*}
\mse{e_{1,3}(t,x)}^2&\le \int_0^{t\wedge\dt}\sup_{z\in\overline{\mathcal{D}}}\mse{\util(s,z) - \util\left(\ell(s),z\right)}^2\dd s\\
&\quad +\mathds{1}_{t>\dt}\int_\dt^t\sup_{z\in\overline{\mathcal{D}}}\mse{\util(s,z) - \util\left(\ell(s),z\right)}^2\dd s.
\end{align*}
On the one hand, recall that the auxiliary process $\util$ takes values in the domain $[-1,1]$ almost surely, see the property~\eqref{eq-dputil} from Proposition~\ref{prop-utilprops}. As a result, applying the Minkowski inequality, one obtains
\begin{align*}
    \int_0^{t\wedge\dt}\sup_{z\in\overline{\mathcal{D}}}\mse{\util(s,z) - \util\left(\ell(s),z\right)}^2\dd s &\leq \int_0^{t\wedge\dt}\sup_{z\in\overline{\mathcal{D}}}\left(\mse{\util(s,z)} + \mse{\util(\ell(s),z)}\right)^2\dd s \\
    &\leq 2(t\wedge\dt) \\
    &\leq 2\dt.
\end{align*}
On the other hand, assume that $t\in[\dt,T]$. Applying the temporal regularity property~\eqref{eq-temputil} from Proposition~\ref{prop-utilprops} for the auxiliary process $\util$, and the lower bound $\ell(s)\geq s-\dt$, one obtains
\begin{align*}
    \int_\dt^t\sup_{z\in\overline{\mathcal{D}}}\mse{\util(s,z) - \util\left(\ell(s),z\right)}^2\dd s &\leq C_\varepsilon(T)\int_\dt^t\frac{\abs{s-\ell(s)}^{1-2\varepsilon}}{\abs{\ell(s)}^{1-2\varepsilon}} \dd s \\
    &\leq C_\varepsilon(T)\dt^{1-2\varepsilon}\int_\dt^t\abs{\ell(s)}^{-1+2\varepsilon}\dd s \\
    &\leq C_\varepsilon(T)\dt^{1-2\varepsilon}\int_\dt^t\abs{s-\dt}^{-1+2\varepsilon}\dd s \\
    &\leq C_\varepsilon(T)\dt^{1-2\varepsilon}.
\end{align*}
Therefore for the error term $e_{1,3}(t,x)$ one obtains the following upper bound: for all $\varepsilon\in(0,1/2)$, there exists $C_\varepsilon(T)\in(0,\infty)$ such that for all $t\in[0,T]$ and $x\in\overline{\mathcal{D}}$ one has
\begin{equation*}
    \mse{e_{1,3}(t,x)}^2 \leq C_\varepsilon(T)\dt^{1-2\varepsilon}.
\end{equation*}

{\bf Treatment of the error term $e_{1,4}(t,x)$.}

Recall the notation~\eqref{Utilbis} for $\Util(\ell(s),y)$. Applying the Cauchy--Schwarz inequality, the It{\^o} isometry property, and the Minkowski inequality, one has
\begin{align*}
    \mse{e_{1,4}(t,x)}^2 &\leq C(T)\int_0^t\mse{\int_\mathcal{D}\left[G_d(t-s,x,y) - G_d\left(t-\ell(s),x,y\right)\right]\drstar\bigl(\Util(\ell(s),y)\bigr)\dd y}^2\dd s \\
    &\quad + \int_0^t\mse{\int_\mathcal{D}\left[G_d(t-s,x,y) - G_d\left(t-\ell(s),x,y\right)\right]\distar\bigl(\Util(\ell(s),y)\bigr)\dd y}^2\dd s \\
    &\leq C(T)\int_0^t\left(\int_\mathcal{D}\abs{G_d(t-s,x,y) - G_d\left(t-\ell(s),x,y\right)}\mse{\drstar\bigl(\Util(\ell(s),y)\bigr)}\dd y\right)^2\dd s \\
    &\quad + \int_0^t\left(\int_\mathcal{D}\abs{G_d(t-s,x,y) - G_d\left(t-\ell(s),x,y\right)}\mse{\distar\bigl(\Util(\ell(s),y)\bigr)}\dd y\right)^2\dd s.
\end{align*}
Note that the mappings $\drstar$ and $\distar$ defined by~\eqref{eq-starcoeff} are bounded on $[0,T]\times\overline{\mathcal{D}}\times\R$, which can be seen by treating the cases $u\in[-1,1]$ and $u\notin[-1,1]$ separately.
Applying the temporal regularity property of the heat kernel from Lemma~\ref{la-propHK-reg} with $\alpha=\frac12-\varepsilon$, one obtains the following upper bounds: for all $\varepsilon\in(0,1/2)$, there exists $C_\varepsilon(T)\in(0,\infty)$ such that for all $t\in[0,T]$ and $x\in\overline{\mathcal{D}}$ one has
\begin{align*}
    \mse{e_{1,4}(t,x)}^2 &\leq C(T)\int_0^t\left(\int_\mathcal{D}\abs{G_d(t-s,x,y) - G_d\left(t-\ell(s),x,y\right)}\dd y\right)^2\dd s \\
    &\leq C(T)\int_0^t\frac{\abs{s-\ell(s)}^{1-2\varepsilon}}{\abs{t-s}^{1-2\varepsilon}}\dd s \\
    &\leq C(T)\dt^{1-2\varepsilon}\int_0^t\abs{t-s}^{-1+2\varepsilon}\dd s \\
    &\leq C_\varepsilon(T)\dt^{1-2\varepsilon}.
\end{align*}

{\bf Treatment of the error terms $e_{1,5}(t,x)$ and $e_2(t,x)$.}

The error terms $e_{1,5}(t)$ and $e_2(t,x)$ are treated together. Recall the notation~\eqref{Utilbis} for $\Util(\ell(s),y)$. Applying the Cauchy--Schwarz inequality, the It{\^o} isometry property and the Minkowski inequality, one has
\begin{align*}
    &\mse{e_{1,5}(t,x)}^2+\mse{e_2(t,x)}^2 \\
    &\quad \leq C(T)\int_0^t\mse{\int_\mathcal{D}G_d\left(t-\ell(s),x,y\right)\drfstar\bigl(\Util(\ell(s),y)\bigr)\left[\partial_\gamma\Phi\big(\Ztil(\ell(s),y)\big) - \partial_\gamma\Phi\big(\Ztil\left(s,y\right)\big)\right]\dd y}^2\dd s \\
    &\qquad + \int_0^t\mse{\int_\mathcal{D}G_d\left(t-\ell(s),x,y\right)\difstar\bigl(\Util(\ell(s),y)\bigr)\left[\partial_\gamma\Phi\big(\Ztil(\ell(s),y)\big) - \partial_\gamma\Phi\big(\Ztil\left(s,y\right)\big)\right]\dd y}^2\dd s \\
    &\qquad + C(T)\int_0^t\mse{\int_\mathcal{D}G_d\left(t-\ell(s),x,y\right)\difstar\bigl(\Util(\ell(s),y)\bigr)^2\left[\psi\big(\Ztil(s,y)\big) - \psi\big(\Ztil\left(\ell(s),y\right)\big)\right]\dd y}^2\dd s \\
    &\quad \leq C(T)\int_0^t\left(\int_\mathcal{D}G_d\left(t-\ell(s),x,y\right)\mse{\drfstar\bigl(\Util(\ell(s),y)\bigr)\left[\partial_\gamma\Phi\big(\Ztil(\ell(s),y)\big) - \partial_\gamma\Phi\big(\Ztil\left(s,y\right)\big)\right]}\dd y\right)^2\dd s \\
    &\qquad + \int_0^t\left(\int_\mathcal{D}G_d\left(t-\ell(s),x,y\right)\mse{\difstar\bigl(\Util(\ell(s),y)\bigr))\left[\partial_\gamma\Phi\big(\Ztil(\ell(s),y)\big) - \partial_\gamma\Phi\big(\Ztil\left(s,y\right)\big)\right]}\dd y\right)^2\dd s \\
    &\qquad + C(T)\int_0^t\left(\int_\mathcal{D}G_d\left(t-\ell(s),x,y\right)\mse{\difstar\bigl(\Util(\ell(s),y)\bigr)^2\left[\psi\big(\Ztil(s,y)\big) - \psi\big(\Ztil\left(\ell(s),y\right)\big)\right]}\dd y\right)^2\dd s.
\end{align*}
Owing to the boundedness of the mappings $\drfstar$ and $\difstar$ on $[0,T]\times\overline{\mathcal{D}}\times\R$ from equation~\eqref{eq-boundafstar} and the boundedness of the integral of the heat kernel from Lemma~\ref{la-propHK-bdd}, one has
\begin{align*}
    &\mse{e_{1,5}(t,x)}^2+\mse{e_2(t,x)}^2 \\
    &\quad \leq C(T)\int_0^t\left(\int_\mathcal{D}G_d\left(t-\ell(s),x,y\right)\mse{\partial_\gamma\Phi\big(\Ztil(\ell(s),y)\big) - \partial_\gamma\Phi\big(\Ztil\left(s,y\right)\big)}\dd y\right)^2\dd s \\
    &\qquad + C(T)\int_0^t\left(\int_\mathcal{D}G_d\left(t-\ell(s),x,y\right)\mse{\psi\big(\Ztil(s,y)\big) - \psi\big(\Ztil\left(\ell(s),y\right)\big)}\dd y\right)^2\dd s \\
    &\quad \leq C(T)\int_0^t\sup_{z\in\overline{\mathcal{D}}}\mse{\partial_\gamma\Phi\big(\Ztil(\ell(s),z)\big) - \partial_\gamma\Phi\big(\Ztil\left(s,z\right)\big)}^2\dd s \\
    &\qquad + C(T)\int_0^t\sup_{z\in\overline{\mathcal{D}}}\mse{\psi\big(\Ztil(s,z)\big) - \psi\big(\Ztil\left(\ell(s),z\right)\big)}^2\dd s.
\end{align*}
Applying the mean value theorem in $\R^3$, the Cauchy--Schwarz inequality, and the moment bounds~\eqref{eq-derPhi} on the derivatives of $\Phi(\Ztil)$ from Proposition~\ref{prop-Ztil}, one obtains
\begin{align*}
    &\mse{e_{1,5}(t,x)}^2+\mse{e_2(t,x)}^2 \\ 
    &\quad \leq C(T)\int_0^t\sup_{z\in\overline{\mathcal{D}}}\E\left[\sup_{c\in[0,1]}\norm{\nabla\partial_\gamma\Phi\big(c\Ztil(s,z) + (1-c)\Ztil(\ell(s),z)\big)}^2\norm{\Ztil(s,z) - \Ztil(\ell(s),z)}^2\right]\dd s \\
    &\qquad + C(T)\int_0^t\sup_{z\in\overline{\mathcal{D}}}\E\left[\sup_{c\in[0,1]}\norm{\nabla\psi\big(c\Ztil(s,z) + (1-c)\Ztil(\ell(s),z)\big)}^2\norm{\Ztil(s,z) - \Ztil(\ell(s),z)}^2\right]\dd s \\
    &\quad \leq C(T)\int_0^t\sup_{z\in\overline{\mathcal{D}}}\E\left[\sup_{c\in[0,1]}\norm{\nabla\partial_\gamma\Phi\big(c\Ztil(s,z) + (1-c)\Ztil(\ell(s),z)\big)}^4\right]^{\frac12}\E\left[\norm{\Ztil(s,z) - \Ztil(\ell(s),z)}^4\right]^{\frac12}\dd s \\
    &\qquad + C(T)\int_0^t\sup_{z\in\overline{\mathcal{D}}}\E\left[\sup_{c\in[0,1]}\norm{\nabla\psi\big(c\Ztil(s,z) + (1-c)\Ztil(\ell(s),z)\big)}^4\right]^{\frac12}\E\left[\norm{\Ztil(s,z) - \Ztil(\ell(s),z)}^4\right]^{\frac12}\dd s \\
    &\quad \leq C(T)\int_0^t\sup_{z\in\overline{\mathcal{D}}}\E\left[\norm{\Ztil(s,z) - \Ztil(\ell(s),z)}^4\right]^{\frac12}\dd s.
\end{align*}
Finally, applying the temporal regularity property~\eqref{eq-regZtil} from Proposition~\ref{prop-Ztil} for $\Ztil$, one obtains the following upper bound: there exists $C(T)\in(0,\infty)$ such that for all $t\in[0,T]$ and $x\in\overline{\mathcal{D}}$ one has
\begin{align*}
    \mse{e_{1,5}(t,x)}^2+\mse{e_2(t,x)}^2 &\leq C(T)\int_0^t\abs{s-\ell(s)}\dd s \leq C(T)\tau.
\end{align*}

{\bf Conclusion.}

Recalling the decompositions~\eqref{onemoreerror} of the error $e(t,x)=\ustar(t,x)-\util(t,x)$ and~\eqref{onemoreerror2} of the error term $e_{1}(t,x)$, gathering the upper bounds obtained above, one obtains the following inequality: for all $\varepsilon\in(0,1/2)$, there exists $C_\varepsilon(T)\in(0,\infty)$ such that for all $t\in[0,T]$ one has
\begin{equation*}
    \sup_{x\in\overline{\mathcal{D}}}\mse{\ustar(t,x)-\util(t,x)}^2 \leq C_\varepsilon(T)\left(\dt^{1-2\varepsilon} + \int_0^t\sup_{x\in\overline{\mathcal{D}}}\mse{\ustar(s,x)-\util(s,x)}^2\dd s\right).
\end{equation*}
Applying the Gr\"{o}nwall inequality, one thus obtains the following result: for all $\varepsilon\in(0,1/2)$, there exists $C_\varepsilon(T)\in(0,\infty)$ such that
\begin{equation}\label{eq-timecontresult}
    \underset{t\in[0,T]}\sup~\sup_{x\in\overline{\mathcal{D}}}\mse{\ustar(t,x)-\util(t,x)}^2 \leq C_\varepsilon(T)\dt^{1-2\varepsilon}.
\end{equation}
This concludes the proof of the inequality~\eqref{convTimeCont} for the error $e(t,x)=\ustar(t,x)-\util(t,x)$. As explained at the beginning of this proof, since $\util(t_m,\cdot)=u^\star_m$ at the time-grid points $t_m$, this implies the strong error estimates~\eqref{convDiscont}. This concludes the proof of Theorem~\ref{thm-mainthm}.

\subsection{Proof of Proposition~\ref{prop-exuniqdp}}\label{sec-proofexuniqdp}

As a consequence of Theorem~\ref{thm-mainthm}, more precisely of the stronger result~\eqref{convTimeCont}, we prove below that the mild solution $\ustar$ to the SPDE~\eqref{eq-SPDEstar} is a solution to the original SPDE~\eqref{eq-SPDE} in the sense of Definition~\ref{defi-mild}. This shows the existence of a mild solution to~\eqref{eq-SPDE}. The uniqueness follows from a straightforward argument, for instance noting that any mild solution to~\eqref{eq-SPDE} is a mild solution to the SPDE~\eqref{eq-SPDEstar}, which admits a unique solution.

Let $\varepsilon\in(0,1/2)$ be fixed. For all $n\in\N$, set $M_n=2^n$, and consider the time-step size $\dt_n=T/M_n=T/2^n$. Let $\util_n$ be the auxiliary process defined by~\eqref{eq-auxproc} with time-step size $\dt_n$. Owing to the strong error estimates~\eqref{convTimeCont}, for all $n\in\N$, one has
\[
\underset{t\in[0,T]}\sup~\sup_{x\in\overline{\mathcal{D}}}\mse{\ustar(t,x)-\util_n(t,x)} \leq C_\varepsilon(T)\rho^n,
\]
with $\rho=2^{-(\frac12-\varepsilon)}$. As a result, one has $\sum_{n\in\N}\rho^n<\infty$.

Let $t\in[0,T]$ and $x\in\overline{\mathcal{D}}$. Owing to the Markov inequality and the Borel--Cantelli lemma, one has the almost sure convergence
\[
\util_n(t,x)\underset{n\to\infty}\to \ustar(t,x).
\]
Due to the domain preservation property~\eqref{eq-dputil} from Proposition~\ref{prop-utilprops} for the auxiliary processes $\util_n$, one has $\util_n(t,x)\in[-1,1]$ almost surely, and due to the almost sure convergence above one has $\ustar(t,x)\in[-1,1]$ almost surely.

More precisely, one has proved that for all $t\in[0,T]$ and $x\in\overline{\mathcal{D}}$ one has
\[
\PP\left(\ustar(t,x)\in[-1,1]\right)=1.
\]
Since the mapping $(t,x)\in[0,T]\times\overline{\mathcal{D}}\mapsto \ustar(t,x)$ is almost surely continuous, one obtains the stronger property
\[
\PP\bigg(\ustar(t,x)\in[-1,1],\quad \forall~t\in[0,T],~ x\in\overline{\mathcal{D}}\bigg) = 1,
\]
meaning that almost surely one has
\[
\underset{t\in[0,T]}\sup~\underset{x\in\overline{\mathcal{D}}}\sup~|\ustar(t,x)|\le 1.
\]
Finally, recalling that the drift and diffusion coefficients $\drstar$ and $\distar$ defined by~\eqref{eq-starcoeff} are extensions of $\dr$ and $\di$, using the mild formulation of the SPDE~\eqref{eq-SPDEstar} and applying the property above, for all $t\in(0,T]$ one has
\begin{align*}
    \ustar(t,x) &= \int_\mathcal{D}G_d(t,x,y)u_0(y)\dd y + \int_0^t\int_\mathcal{D}G_d(t-s,x,y)\drstar\left(s,y,\ustar(s,y)\right)\dd y\dd s \\
    &\quad+ \int_0^t\int_\mathcal{D}G_d(t-s,x,y)\distar\left(s,y,\ustar(s,y)\right)\dd y\dd\beta(s)\\
    &=\int_\mathcal{D}G_d(t,x,y)u_0(y)\dd y + \int_0^t\int_\mathcal{D}G_d(t-s,x,y)\dr\left(s,y,\ustar(s,y)\right)\dd y\dd s \\
    &\quad+ \int_0^t\int_\mathcal{D}G_d(t-s,x,y)\di\left(s,y,\ustar(s,y)\right)\dd y\dd\beta(s).
\end{align*}
Combining the properties above shows that $\ustar$ is indeed a mild solution to the original SPDE~\eqref{eq-SPDE} in the sense of Definition~\ref{defi-mild}.
Applying the arguments presented above concludes the proof of the existence and uniqueness of a mild solution to~\eqref{eq-SPDE} in the sense of Definition~\ref{defi-mild}.

Note also that if $u$ and $\ustar$ are the unique mild solutions to~\eqref{eq-SPDE} and~\eqref{eq-SPDEstar} respectively, then one has
\begin{equation}\label{eq-uustar}
\PP\bigg(u(t,x)=\ustar(t,x),\quad \forall~t\in[0,T], ~x\in\overline{\mathcal{D}}\bigg) = 1.
\end{equation}

\subsection{Proof of Corollary~\ref{lecorollaire}}\label{sec-prooflecoro}

Let $\varepsilon\in(0,1/2)$.

Recall from Section~\ref{subMain} that the numerical solutions $u_m$ given by the DPLT splitting scheme~\eqref{eq-dplt} and $u^\star_m$ given by the auxiliary version~\eqref{eq-dpltstar} coincide: almost surely for all $m\in\{0,\ldots,M\}$ and all $x\in\overline{\mathcal{D}}$ one has $u^\star_m(x)=u_m(x)$. In addition, due to the property~\eqref{eq-uustar} established in the proof of Proposition~\ref{prop-exuniqdp} above, almost surely for all $m\in\{0,\ldots,M\}$ and all $x\in\overline{\mathcal{D}}$ one has $\ustar(t_m,x)=u(t_m,x)$.

Therefore, as a straightforward application of the strong error estimates~\eqref{convDiscont} from Theorem~\ref{thm-mainthm} one obtains
\begin{equation*}
    \sup_{0\leq m\leq M}\sup_{x\in\overline{\mathcal{D}}}\left(\E\left[\abs{u_m(x) - u(t_m,x)}^2\right]\right)^{\frac{1}{2}} = \sup_{0\leq m\leq M}\sup_{x\in\overline{\mathcal{D}}}\left(\E\left[\abs{u^\star_m(x) - \ustar(t_m,x)}^2\right]\right)^{\frac{1}{2}}\leq C_\varepsilon(T)\dt^{\frac{1}{2}-\varepsilon}.
\end{equation*} 
The proof of Corollary~\ref{lecorollaire} is thus completed.

\section{Generalisations}\label{sec:genZ}
In this section, we briefly describe three possible generalisations of the domain preserving Lie--Trotter scheme~\eqref{eq-dplt} and of the SPDE~\eqref{eq-SPDE}. In Subsection~\ref{sec:Strato} we consider SPDEs where the noise is interpreted in the Stratonovich sense}.
In Subsection~\ref{sec:mulNoise}, we consider a version of~\eqref{eq-SPDE} driven by multiple noises and the adaptation of the scheme~\eqref{eq-dplt} in that setting.
In Subsection~\ref{sec:syst}, we consider systems of SPDEs (with two components). In order to be concise, we only provide brief descriptions with numerical experiments and the detailed analysis is omitted. These generalisations are presented separately, but can certainly be combined.

\subsection{Generalisation to Stratonovich SPDE}\label{sec:Strato}

In this subsection we consider the Stratonovich version of the SPDE~\eqref{eq-SPDE}, and generalise the DPLT splitting scheme~\eqref{eq-dplt} to this setting.

We consider the following class of Stratonovich SPDEs:
\begin{equation}\label{eq-SPDEstrato}
\begin{cases}
    \dd u(t,x) = \left(\Delta u(t,x) + \dr(t,x,u(t,x))\right)\dd t+ \di(t,x,u(t,x))\circ\dd\beta(t), ~ t\in(0,T], x\in\mathcal{D}, \\
    u(t,x) = 0, \quad t\in[0,T], x\in\partial\mathcal{D}, \\
    u(0,x) = u_0(x), \quad x\in\overline{\mathcal{D}},
\end{cases}
\end{equation}
where we assume that the initial condition $u_0$ satisfies Assumption~\ref{ass-ic} and that the drift and diffusion coefficients $\dr$ and $\di$ satisfy Assumption~\ref{ass-driftdiff}.

For all $(t,x,u)\in [0,\infty)\times\overline{\mathcal{D}}\times[-1,1]$, set
\[
\overline{\dr}(t,x,u)=\dr(t,x,u)+\frac12\partial_u\di(t,x,u)\di(t,x,u).
\]
The Stratonovich SPDE~\eqref{eq-SPDEstrato} is equivalent to the It{\^o} SPDE
\begin{equation}\label{eq-SPDEstratoito}
\dd u(t,x) = \left(\Delta u(t,x) + \overline{\dr}(t,x,u(t,x))\right)\dd t+ \di(t,x,u(t,x))\dd\beta(t),\quad t\in(0,T], x\in\mathcal{D},
\end{equation}
with the same boundary and initial conditions as in~\eqref{eq-SPDEstrato}. Owing to the factorisations of the coefficients $\dr$ and $\di$ given by Lemma~\ref{la-driftdiffdecomp}, for all  $(t,x,u)\in [0,\infty)\times\overline{\mathcal{D}}\times[-1,1]$ one has
\[
\dr(t,x,u)=\drf(t,x,u)\sigma(u),\quad \di(t,x,u)=\dif(t,x,u)\sigma(u)\quad\text{and}\quad \overline{\dr}(t,x,u)=\overline{\drf}(t,x,u)\sigma(u),
\]
where the function $\overline{\drf}$ is defined by
\begin{equation}\label{eq-dfrstrato}
\overline{\drf}(t,x,u)=\drf(t,x,u)+\frac12\partial_u\di(t,x,u)\dif(t,x,u).
\end{equation}
In addition to $\dr$ and $\di$ satisfying Assumption~\ref{ass-driftdiff}, we also assume that they are such that $\overline{\dr}$ and $\di$ satisfy Assumption~\ref{ass-driftdiff}. Therefore the Stratonovich SPDE~\eqref{eq-SPDEstrato} has a unique mild solution $u$. More precisely, its equivalent It\^o formulation~\eqref{eq-SPDEstratoito} admits a unique mild solution in the sense of Definition~\ref{defi-mild}, owing to Proposition~\ref{prop-exuniqdp}. In particular, one has $u(t,x)\in[-1,1]$ for all $t\in[0,T]$ and $x\in\overline{\mathcal{D}}$ almost surely.

To obtain a domain preserving time integrator for the Stratonovich SPDE~\eqref{eq-SPDEstrato}, it suffices to apply the DPLT splitting scheme to its equivalent formulation~\eqref{eq-SPDEstratoito}. Specifically, given an integrator $\Phi:[0,\infty)\times\R\times[-1,1]$ satisfying Assumption~\ref{ass-Phi}, we introduce a time integrator for the Stratonovich SPDE~\eqref{eq-SPDEstrato} as follows: for all for $m\in\{0,1,\ldots,M-1\}$ and $x\in\overline{\mathcal{D}}$, set
\begin{equation}\label{eq-dpltstrato}
\begin{cases}
u_{m+1}(x)=\int_{\mathcal D}G_d(\dt,x,y)\Phi\big(\Ztil_{m+1}(y)\big)\dd y, \\
\Ztil_{m+1}(y)=\left(\dif(t_m,y,u_m(y))^2\dt,\overline{\drf}(t_m,y,u_m(y))\dt+\dif(t_m,y,u_m(y))\dbt_m,u_m(y)\right),
\end{cases}
\end{equation}
where $\overline{\drf}$ is given by~\eqref{eq-dfrstrato}.

Interpreting the version~\eqref{eq-dpltstrato} of the DPLT splitting scheme for the Stratonovich SPDE~\eqref{eq-SPDEstrato} as the DPLT splitting scheme~\eqref{eq-dplt} for the equivalent It\^o formulation, applying Proposition~\ref{prop-dplt} and Corollary~\ref{lecorollaire}, the following result is obtained.
\begin{corollary}\label{cor-resstrato}
	Let $T\in(0,\infty)$. Consider the mild solution $\bigl(u(t,\cdot))_{t\in[0,T]}$ to the Stratonovich SPDE~\eqref{eq-SPDEstrato} and the numerical solution $\bigl(u_m\bigr)_{0\le m\le M}$ given by the DPLT splitting scheme~\eqref{eq-dpltstrato} with time-step size $\dt=T/M$.
    Under Assumptions~\ref{ass-ic}, ~\ref{ass-driftdiff}, ~\ref{ass-Phi}, and the supplementary conditions on $\dr$ and $\di$ above, the DPLT splitting scheme is domain preserving: for any time-step size $\dt=T/M$, one has
        \begin{equation*}
        \PP\left(u_m(x)\in[-1,1],\forall m\in\{0,1,\ldots,M\},\forall x\in\overline{\mathcal{D}}\right) = 1.
    \end{equation*}
     Furthermore, the mean-square error converges to $0$ at order $\frac12-$: for all $T\in(0,\infty)$ and $\varepsilon\in(0,\frac12)$, there exists $C_\varepsilon(T)\in(0,\infty)$ such that, for all $\dt=T/M$, one has
         \begin{equation*}
        \sup_{0\leq m\leq M}\sup_{x\in\overline{\mathcal{D}}} \mse{u_m(x)-u(t_m,x)}
        \leq C_\varepsilon(T)\dt^{\frac12-\varepsilon}.
    \end{equation*}
\end{corollary}

Let us illustrate the domain preservation of the numerical scheme~\eqref{eq-dpltstrato} when applied to the SPDE~\eqref{eq-SPDEstrato} on the time interval $[0,20]$, and with initial value $u_0(x)=\sin(2\pi x)$ for $x\in[0,1]$. The spatial discretisation uses a mesh size of $h=2^{-8}$. We apply the integrators {\upshape EM}, {\upshape sEM}, and {\upshape SEXP} to the equivalent It\^o formulation~\eqref{eq-SPDEstratoito}, and the {\upshape DPLT} splitting scheme from~\eqref{eq-dpltstrato} with time-step size $\dt=2^{-3}$.
For several choices of the coefficients $\dr$ and $\di$ (or $\drf$ and $\dif$), we compute $100$ realisations and compute
the proportion of these realisations which remain in the domain $[-1,1]$.
Table~\ref{tabDPStrato} illustrates the fact that all realisations of the DPLT splitting scheme~\eqref{eq-dpltstrato} take values in the domain $[-1,1]$. On the contrary, one
observes that for the classical time integrators some realisations exit the domain $[-1,1]$.

\begin{table}[h]
\begin{center}
\begin{tabular}{|c | c | c | c| c| c|}
  \hline
  $\drf(t,x,u)$ & $\dif(t,x,u)$ & {\upshape DPLT} & {\upshape EM} & {\upshape sEM} & {\upshape SEXP} \\
  \specialrule{.15em}{.05em}{.05em}
  $-u$ & $2(1-u^2)$  & $100/100$  & $0/100$ & $34/100$ & $100/100$\\
  \hline
  $\sin(x\pi/2)$ & $2\cos(u\pi/2)$ & $100/100$ & $0/100$ & $55/100$ & $99/100$\\
  \hline
  $u$ & $u$ & $100/100$ & $0/100$ & $100/100$ & $100/100$\\
  \hline
  $\sin(u)+\cos(x)$ & $2\exp(u)$ & $100/100$ & $0/100$ & $4/100$ & $93/100$\\
  \hline
\end{tabular}
\end{center}
\caption{Proportion of realisations in the domain $[-1,1]$ for $100$ simulated sample paths for the time integrators:
DPLT splitting scheme~\eqref{eq-dpltstrato} ({\upshape DPLT}), Euler--Maruyama scheme ({\upshape EM}), semi-implicit Euler--Maruyama scheme ({\upshape sEM}),
and stochastic exponential Euler scheme ({\upshape SEXP}).}
\label{tabDPStrato}
\end{table}

We now provide a numerical illustration of the mean-square convergence with order $\frac12-$ obtained in Corollary~\ref{cor-resstrato}.
For this purpose we consider the Stratonovich SPDE~\eqref{eq-SPDEstrato} with initial value
$u_0(x)=\sin(2\pi x)$ for $x\in[0,1]$.
The spatial discretisation is performed with a finite difference method with mesh size $h=2^{-8}$.
For the temporal discretisation, we apply the integrators {\upshape sEM}, {\upshape SEXP} to the equivalent It\^o formulation~\eqref{eq-SPDEstratoito},
and the {\upshape DPLT} splitting scheme from~\eqref{eq-dpltstrato}.
Like in Section~\ref{subNumExp}, for the applications of the integrators {\upshape sEM} and {\upshape SEXP},
the coefficients $\overline{\drf}$ and $\dif$ are replaced with extensions $\overline{\drf}^\star$ and $\difstar$ defined on $[0,\infty)\times\overline{\mathcal{D}}\times \R$. For the drift and diffusion coefficients $\dr=\drf\sigma$ and $\di=\dif\sigma$, we consider $\drf(t,x,u)=\sin(u)e^{3u}$ and $\dif(t,x,u)=2e^u$.
In Figure~\ref{fig-strongStrato}, we plot the maximum mean-square errors for the time-step sizes $\dt=2^{-4},\ldots,2^{-15}$ on the time domain $[0,1]$.
We use the {\upshape DPLT} splitting scheme~\eqref{eq-dpltstrato} with time-step size $\dt^{\text{ref}}=2^{-16}$ as the reference solution, and take $200$ samples to approximate the expectations. We observe the rate of convergence $\frac12$ for the DPLT splitting scheme~\eqref{eq-dpltstrato} as predicted by Corollary~\ref{cor-resstrato}.
The comparison with the classical integrators {\upshape sEM} and {\upshape SEXP} justifies that the proposed DPLT splitting scheme~\eqref{eq-dpltm} is consistent with the SPDE~\eqref{eq-SPDEm}.

\begin{figure}[h]
    \centering
    \includegraphics[width=0.5\textwidth]{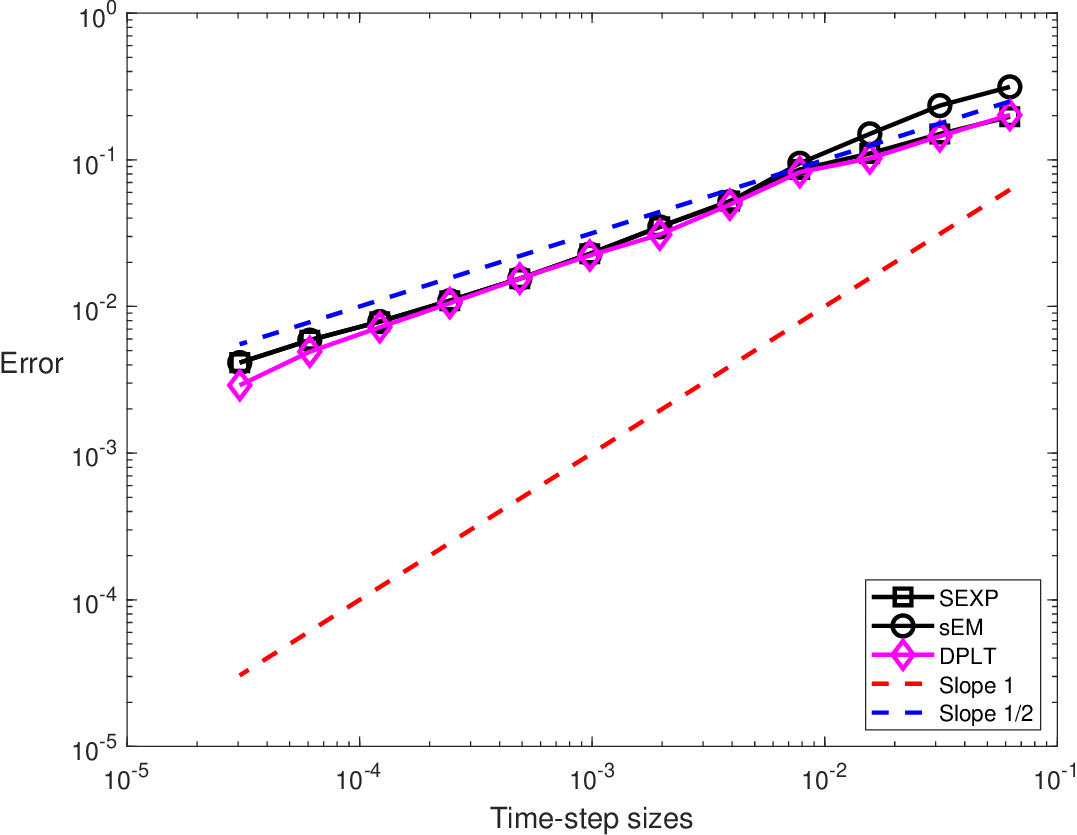}
    \caption{Maximum mean-square errors for the {\upshape sEM}, {\upshape SEXP}, and {\upshape DPLT} time integrators for the Stratonovich SPDE~\eqref{eq-SPDEstrato} on the time interval $[0,1]$ in one-dimensional space, with space-mesh size $h=2^{-8}$.}
   \label{fig-strongStrato}
\end{figure}

\subsection{Generalisation to multidimensional noise}\label{sec:mulNoise}
In this subsection, we generalise the results of this article for the SPDE~\eqref{eq-SPDE} to It\^o SPDEs driven by multi-dimensional noise.
Let $K\in\mathbb N$ and $(\beta(t))_{t\geq0}$ be a standard $\mathbb R^K$-dimensional Brownian motion, with $\beta(t)=\bigl(\beta_1(t),\ldots,\beta_K(t)\bigr)$. Let also $\di_1,\ldots,\di_K\colon[0,\infty)\times\overline{\mathcal{D}}\times[-1,1]\to \R$ be functions which satisfy the same conditions as the diffusion coefficient $\di$ in~\eqref{eq-SPDE}. We consider SPDEs (interpreted in the Itô sense) of the form:
\begin{equation}\label{eq-SPDEm}
\left\lbrace
\begin{aligned}
    &\dd u(t,x) = \bigl( \Delta u(t,x) + \dr (t,x,u(t,x)) \bigr)\dd t+g(t,x,u(t,x))\cdot\dd\beta(t), \quad t\in(0,T], x\in\mathcal{D}, \\
    &u(t,x) = 0, \quad t\in[0,T], x\in\partial\mathcal{D}, \\
    &u(0,x) = u_0(x), \quad x\in\overline{\mathcal{D}},
\end{aligned}
\right.
\end{equation}
where
$$
g(t,x,u(t,x))\cdot\dd\beta(t)=\sum_{k=1}^K\di_k(t,x,u(t,x))\dd\beta_k(t)
$$
for $g=(g_1,\ldots,g_K)$.

Under Assumption~\ref{ass-ic} for the initial value $u_0$ and Assumption~\ref{ass-driftdiff} for the drift coefficient $\dr=\drf\sigma$ and the diffusion coefficients $\di_k$,
for $k=1,\ldots,K$, Proposition~\ref{prop-exuniqdp} can be generalised: the SPDE~\eqref{eq-SPDEm} admits a unique mild solution such that almost surely $u(t,x)\in[-1,1]$ for all $t\in[0,T]$ and $x\in\overline{\mathcal{D}}$.

The definition of the DPLT splitting scheme is then adapted as follows. Since the mappings $\di_1,\ldots,\di_K$ satisfy the conditions in Assumption~\ref{ass-driftdiff}, one can factorise $\di_k=\dif_k\sigma$ for all $k=1,\ldots,K$ by Lemma~\ref{la-driftdiffdecomp}. Given an integrator $\Phi\colon[0,\infty)\times\mathbb R\times[-1,1]\to\mathbb R$ satisfying Assumption~\ref{ass-Phi}, the DPLT splitting scheme applied to the SPDE~\eqref{eq-SPDEm} driven by multidimensional noise is defined as follows: for all $m\in\{0,1,\ldots,M-1\}$ and $x\in\overline{\mathcal D}$, set
\begin{equation}\label{eq-dpltm}
\begin{cases}
\displaystyle u_{m+1}(x)=\int_{\mathcal D}G_d(\dt,x,y)\Phi\big(\Ztil_{m+1}(y)\big)\dd y, \\
\Ztil_{m+1}(y)=\left(\| \dif(t_m,y,u_m(y)) \|^2\dt,\drf(t_m,y,u_m(y))\dt+\dif_k(t_m,y,u_m(y))\cdot\dbt_{m},u_m(y)\right),
\end{cases}
\end{equation}
where we recall that $\displaystyle\| \dif(t_m,y,u_m(y)) \|^2=\sum_{k=1}^K\dif_k(t_m,y,u_m(y))^2$
for $\dif=(\dif_1,\ldots,\dif_K)$. Here, we set $\dbt_{m}=(\dbt_{1,m},\ldots,\dbt_{K,m})$ with the Brownian increments denoted as $\dbt_{k,m}=\dbt_k(t_{m+1})-\dbt_k(t_m)$
for all $k\in\{1,\ldots,K\}$ and $m\in\{0,1,\ldots,M-1\}$.

From the definition of the time integrator~\eqref{eq-dpltm}, it is straightforward to adapt the proof of Proposition~\ref{prop-dplt}, and to show that it is domain preserving,
for any time-step size $\dt=T/M$: almost surely for all $m\in\{0,\ldots,M\}$ and $x\in\overline{\mathcal{D}}$ one has $u_m(x)\in[-1,1]$. Moreover, the mean-square convergence result with order $\frac12-$ given in Corollary~\ref{lecorollaire} can also be generalised. Numerical experiments below illustrate the above properties of the DPLT splitting scheme~\eqref{eq-dpltm}.

We first illustrate the domain preservation of the numerical scheme~\eqref{eq-dpltm} when applied to the SPDE~\eqref{eq-SPDEm} with $K=2$, on the time interval $[0,20]$, and with initial value $u_0(x)=\sin(4\pi x)$ for $x\in[0,1]$.
We fix the finite-difference mesh size to $h=2^{-8}$ and apply the numerical schemes from Subsection~\ref{subNumExp} with time-step size $\dt=2^{-3}$.
For several choices of the coefficients $\dr$, $\di_1$, and $\di_2$, we compute $100$ realisations and compute the proportion of these realisations which remain in the domain $[-1,1]$.
Table~\ref{tabDPm} illustrates the fact that all realisations of the DPLT splitting scheme~\eqref{eq-dpltm} take values in the domain $[-1,1]$. On the contrary, one 
observes that for the classical time integrators some realisations exit the domain $[-1,1]$.

\begin{table}[h]
\begin{center}
\begin{tabular}{|c | c | c | c | c| c| c|}
  \hline
  $\drf(t,x,u)$ & $\dif_1(t,x,u)$ & $\dif_2(t,x,u)$ & {\upshape DPLT} & {\upshape EM} & {\upshape sEM} & {\upshape SEXP} \\
  \specialrule{.15em}{.05em}{.05em}
  $u+\cos(x)$ & $u+2\sin(t)$ &  $u$  & $100/100$  & $0/100$ & $70/100$ & $100/100$\\
  \hline
  $u$ & $u$ & $u$ & $100/100$ & $0/100$ & $100/100$ & $100/100$\\
  \hline
  $\sin(x\pi/2)$ & $2\cos(u\pi/2)$ & $\sin(5u\pi/2)$ & $100/100$ & $0/100$ & $32/100$ & $95/100$\\
  \hline
  $-u$ & $2(1-u^2)$ & $\exp(u)$ & $100/100$ & $0/100$ & $13/100$ & $96/100$\\
  \hline
\end{tabular}
\end{center}
\caption{Proportion of realisations in the domain $[-1,1]$ for $100$ simulated sample paths for the time integrators:
DPLT splitting scheme~\eqref{eq-dpltm} ({\upshape DPLT}), Euler--Maruyama scheme ({\upshape EM}), semi-implicit Euler--Maruyama scheme ({\upshape sEM}),
and stochastic exponential Euler scheme ({\upshape SEXP}).}
\label{tabDPm}
\end{table}

We now verify that the DPTL splitting scheme~\eqref{eq-dpltm} converges in the mean-square sense with order $\frac12-$, as predicted by the generalization of Corollary~\ref{lecorollaire} to the multidimensional noise setting mentioned above. We consider the SPDE~\eqref{eq-SPDEm} with $K=2$ and with initial value
$u_0(x)=\sin(2\pi x)$ for $x\in[0,1]$. We discretise the problem in space using standard finite differences
with mesh size $h=2^{-8}$, and we apply the time integrators {\upshape sEM}, {\upshape SEXP} adapted to the multidimensional noise setting,
and the {\upshape DPLT} scheme~\eqref{eq-dpltm}. As in Subsection~\ref{subNumExp}, the classical integrators {\upshape sEM} and {\upshape SEXP} are applied with the extensions $\drstar$, $\distar_1$ and $\distar_2$ of the coefficients $\dr$, $\di_1$ and $\di_2$ given by equation~\eqref{eq-starcoeff}, to ensure their mean-square convergence with order $\frac12-$ as the coefficients satisfy appropriate global Lipschitz continuity properties with respect to the third variable $u$.
For the coefficients $\drf$, $\dif_1$, and $\dif_2$ associated by Lemma~\ref{la-driftdiffdecomp} to $\dr$, $\di_1$ and $\di_2$, we choose $\drf(t,x,u) = \abs{u}$, $\dif_1(t,x,u) = 2\left(u^2+\abs{x-0.5}\right)$, and $\dif_2(t,x,u) = 2\left(u+\sqrt{t}\right)$. In Figure~\ref{fig-strong2BM}, the maximum of the mean-square errors is displayed for
the time-step sizes $\dt=2^{-4},\ldots,2^{-15}$ on the time interval $[0,1]$. The reference solution is computed by applying the DPLT splitting scheme~\eqref{eq-dpltm}
with time-step size $\dt^{\text{ref}}=2^{-16}$. We use $200$ samples to approximate the expectations.
We observe the rate of convergence $\frac12$ for the DPLT splitting scheme~\eqref{eq-dpltm}, as predicted by the expected generalisation of Corollary~\ref{lecorollaire}.
The comparison with the classical integrators {\upshape sEM} and {\upshape SEXP} justifies that the proposed DPLT splitting scheme~\eqref{eq-dpltm} is consistent with the SPDE~\eqref{eq-SPDEm}.

\begin{figure}[h]
    \centering
    \includegraphics[width=0.5\textwidth]{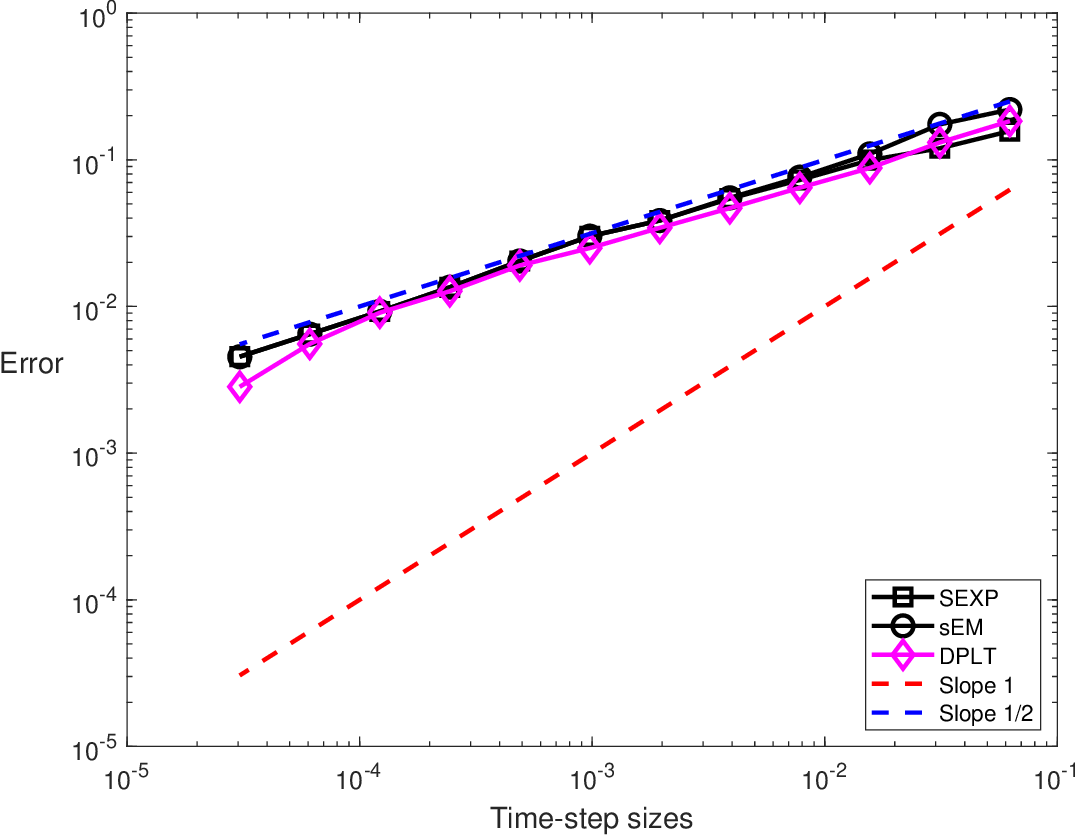}
    \caption{Maximum mean-square errors for the {\upshape sEM}, {\upshape SEXP}, and {\upshape DPLT} time integrators on the time interval $[0,1]$ for the SPDE~\eqref{eq-SPDEm} in one-dimensional space, with a two-dimensional noise ($K=2$).}
   \label{fig-strong2BM}
\end{figure}

\subsection{Generalisation to systems}\label{sec:syst}

In the following subsection, we explain how to extend the definition of the DPLT splitting scheme~\eqref{eq-dplt} to systems of SPDEs instead of scalar SPDEs~\eqref{eq-SPDE}. For simplicity we consider systems of dimension two, however the generalization to arbitrary dimension is straightforward. Let $(\beta_1(t))_{t\geq0}$ and $(\beta_2(t))_{t\geq0}$
be two (equal or independent) real-valued Brownian motions. We consider systems of SPDEs of the following form, where the unknowns are stochastic processes $\bigl(u_1(t,x)\bigr)_{t\in[0,T],x\in\overline{\mathcal{D}}}$ and $\bigl(u_2(t,x)\bigr)_{t\in[0,T],x\in\overline{\mathcal{D}}}$:
\begin{equation}\label{eq-SPDEsyst}
\begin{cases}
    \dd u_1(t,x) = \left(\Delta u_1(t,x) + \dr_1(t,x,u_1(t,x),u_2(t,x))\right)\dd t \\
    \qquad\qquad\quad + \di_1(t,x,u_1(t,x),u_2(t,x))\dd\beta_1(t), \quad t\in(0,T], x\in\overline{\mathcal{D}}, \\
    \dd u_2(t,x) = \left(\Delta u_2(t,x) + \dr_2(t,x,u_1(t,x),u_2(t,x))\right)\dd t \\
    \qquad\qquad\quad + \di_2(t,x,u_1(t,x),u_2(t,x))\dd\beta_2(t), \quad t\in(0,T], x\in\overline{\mathcal{D}}, \\
    u_1(t,x) = u_2(t,x) = 0, \quad t\in[0,T], x\in\partial\mathcal{D}, \\
    u_1(0,x) = u_{1,0}(x), \quad u_2(0,x) = u_{2,0}(x), \quad x\in\overline{\mathcal{D}}.
\end{cases}
\end{equation}
The initial values $u_{1,0}$ and $u_{2,0}$ satisfy Assumption~\ref{ass-ic}. The drift and diffusion coefficients $\dr_1,\dr_2,\di_1,\di_2\colon[0,\infty)\times\overline{\mathcal{D}}\times[-1,1]^2\to\R$ are continuous mappings which satisfy variants of Assumptions~\ref{ass-driftdiff-regt} and~\ref{ass-driftdiff-regu}, and the following variant of Assumption~\ref{ass-driftdiff-vanish}: for all $t\in[0,\infty)$, $x\in\overline{\mathcal{D}}$ and $u\in[-1,1]$, one has
\[
\dr_1(t,x,\pm 1,u)=\di_1(t,x,\pm 1,u)=0\quad\text{and}\quad\dr_2(t,x,u,\pm 1)=\di_2(t,x,u,\pm 1)=0.
\]
Under the assumptions above, a version of Proposition~\ref{prop-exuniqdp} is valid:  there exists a unique mild solution to~\eqref{eq-SPDEsyst} such that almost surely $u(t,x)=\bigl(u_1(t,x),u_2(t,x)\bigr)\in[-1,1]^2$ for all $t\in[0,T]$ and $x\in\overline{\mathcal{D}}$.

The definition of the DPLT splitting scheme~\eqref{eq-dplt} is then adapted in the following way. Owing to the condition above, by Lemma~\ref{la-driftdiffdecomp}, the drift and diffusion coefficients $\dr_1,\dr_2,\di_1,\di_2\colon[0,\infty)\times\overline{\mathcal{D}}\times[-1,1]^2\to\R$ can be factorised as follows: there exist mappings $\drf_1,\drf_2,\dif_1,\dif_2\colon[0,\infty)\times\overline{\mathcal{D}}\times[-1,1]^2\to\R$ such that for all $(t,x)\in[0,\infty)\times\overline{\mathcal{D}}$ and $(u_1,u_2)\in[-1,1]^2$ one has
\begin{align*}
&    \dr_1(t,x,u_1,u_2) = \drf_1(t,x,u_1,u_2)\sigma(u_1) \quad \text{and} \quad \dr_2(t,x,u_1,u_2) = \drf_2(t,x,u_1,u_2)\sigma(u_2),\\
&    \di_1(t,x,u_1,u_2) = \dif_1(t,x,u_1,u_2)\sigma(u_1) \quad \text{and} \quad \di_2(t,x,u_1,u_2) = \dif_2(t,x,u_1,u_2)\sigma(u_2),
\end{align*}
where we recall that $\sigma(u) = (u-1)(u+1)$ for all $u\in[-1,1]$. Given an integrator $\Phi:[0,\infty)\times\R\times[-1,1]$ which satisfies Assumption~\ref{ass-Phi}, the DPLT splitting scheme applied to the two-dimensional system of SPDEs~\eqref{eq-SPDEsyst} is defined as follows: denoting $u_m(x)=(u_{m,1}(x),u_{m,2}(x))$, for all $m\in\{0,1,\ldots,M-1\}$ and $x\in\overline{\mathcal{D}}$, set
\begin{equation}\label{eq-dpltsyst}
\begin{cases}
    \displaystyle u_{m+1,j}(x) = \int_\mathcal{D}G_d(\dt,x,y)\Phi\left(\Ztil_{m+1,j}(y)\right)\dd y, \\
    \Ztil_{m+1,j}(y)=\left(\dif_j(\UMY)^2\dt,\drf_j(\UMY)\dt + \dif_j(\UMY)\delta\beta_{m,j},u_{m,j}(y)\right),
\end{cases}
\end{equation}
where $\UMY=(t_m,y,u_{m,1}(y),u_{m,2}(y))$ and the Brownian increments are defined as $\delta\beta_{m,j} = \beta_j(t_{m+1})-\beta_j(t_m)$ for all $j=1,2$ and $m\in\{0,1,\ldots,M-1\}$.

From the definition of the time integrator~\eqref{eq-dpltsyst}, a straightforward adaptation of the proof of Proposition~\ref{prop-dplt} shows that it is domain preserving for any time-step size $\dt=T/M$: almost surely for all $m\in\{0,\ldots,M\}$ and $x\in\overline{\mathcal{D}}$, one has $u_m(x)=\bigl(u_{m,1}(x),u_{m,2}(x)\bigr)\in[-1,1]^2$. Moreover, the mean-square convergence result with order $\frac12-$ given in Corollary~\ref{lecorollaire} can also be generalised. Numerical experiments below illustrate the above properties of the DPLT splitting scheme~\eqref{eq-dpltsyst}.

We first illustrate the domain preservation of the DPLT splitting scheme~\eqref{eq-dpltsyst} when applied to the two-dimensional system of SPDEs~\eqref{eq-SPDEsyst}
with two independent Brownian motions, on the time interval $[0,20]$ with initial values $u_{1,0}(x)=\sin(4\pi x)$ and $u_{2,0}(x)=\sin(2\pi x)$ for $x\in[0,1]$.
We fix the finite-difference mesh size to $h=2^{-8}$ and apply the above numerical schemes with time-step size $\dt=2^{-3}$.
For several choices of the coefficients of the system of SPDEs, we compute $100$ realisations and compute the proportion of these realisations which remain in the domain $[-1,1]^2$, precisely the proportions of these realisations such that $u_j$ remains in $[-1,1]$ for $j=1,2$.
Table~\ref{tabDPsyst} illustrates the fact that all realisations of the DPLT splitting scheme~\eqref{eq-dpltsyst} take values in the domain $[-1,1]^2$. On the contrary, one observes that for the classical time integrators some realisations exit the domain $[-1,1]^2$.

\begin{table}[h]
\begin{center}
\resizebox{\columnwidth}{!}{%
\begin{tabular}{|c | c | c | c| c| c|}
  \hline
  $(\drf_1,\drf_2)(t,x,u_1,u_2)$ & $(\dif_1,\dif_2)(t,x,u_1,u_2)$ & {\upshape DPLT} ($u_1,u_2$) & {\upshape EM} ($u_1,u_2$)& {\upshape sEM} ($u_1,u_2$)& {\upshape SEXP} ($u_1,u_2$)\\
  \specialrule{.15em}{.05em}{.05em}
  $(u_1^2+u_2^2,\exp(u_1+u_2))$ & $(u_1u_2,2(\cos(u_1)+\sin(u_2)))$ & $100/100, 100/100$  & $0/100, 0/100$ & $99/100, 31/100$ & $100/100, 96/100$\\
  \hline
  $(x^2+t,u_2)$ & $(u_1+\sin(t),\exp(u_2)\exp(x))$ & $100/100, 100/100$ & $0/100, 0/100$ & $0/100, 40/100$ & $60/100, 100/100$\\
  \hline
  $(u_2,u_1)$ & $(u_1,u_2)$ & $100/100, 100/100$ & $0/100, 0/100$ & $100/100, 100/100$ & $100/100, 100/100$\\
  \hline
\end{tabular}
}
\end{center}
\caption{Proportion of realisations in the domain $[-1,1]^2$ for $100$ simulated sample paths for the time integrators:
DPLT splitting scheme~\eqref{eq-dpltsyst} ({\upshape DPLT}), Euler--Maruyama scheme ({\upshape EM}), semi-implicit Euler--Maruyama scheme ({\upshape sEM}),
and stochastic exponential Euler scheme ({\upshape SEXP}). The SPDE~\eqref{eq-SPDEsyst} has two independent Brownian motions $\beta_1$ and $\beta_2$.}
\label{tabDPsyst}
\end{table}

We repeat these numerical experiments but consider the two-dimensional system of SPDEs~\eqref{eq-SPDEsyst} with equal Brownian motions $\beta_1=\beta_2$.
The results are presented in Table~\ref{tabDPsyst1}. The conclusions are the same as for Table~\ref{tabDPsyst}.

\begin{table}[h]
\begin{center}
\resizebox{\columnwidth}{!}{%
\begin{tabular}{|c | c | c | c| c| c|}
  \hline
  $(\drf_1,\drf_2)(t,x,u_1,u_2)$ & $(\dif_1,\dif_2)(t,x,u_1,u_2)$ & {\upshape DPLT} ($u_1,u_2$) & {\upshape EM} ($u_1,u_2$)& {\upshape sEM} ($u_1,u_2$)& {\upshape SEXP} ($u_1,u_2$)\\
  \specialrule{.15em}{.05em}{.05em}
  $(u_1^2+u_2^2,\exp(u_1+u_2))$ & $(u_1u_2,2(\cos(u_1)+\sin(u_2)))$ & $100/100, 100/100$  & $0/100, 0/100$ & $99/100, 34/100$ & $100/100, 98/100$\\
  \hline
  $(x^2+t,u_2)$ & $(u_1+\sin(t),\exp(u_2)\exp(x))$ & $100/100, 100/100$ & $0/100, 0/100$ & $0/100, 44/100$ & $61/100, 100/100$\\
  \hline
  $(u_2,u_1)$ & $(u_1,u_2)$ & $100/100, 100/100$ & $0/100, 0/100$ & $100/100, 100/100$ & $100/100, 100/100$\\
  \hline
\end{tabular}
}
\end{center}
\caption{Proportion of realisations in the domain $[-1,1]^2$ for $100$ simulated sample paths for the time integrators:
DPLT splitting scheme~\eqref{eq-dpltsyst} ({\upshape DPLT}), Euler--Maruyama scheme ({\upshape EM}), semi-implicit Euler--Maruyama scheme ({\upshape sEM}),
and stochastic exponential Euler scheme ({\upshape SEXP}). The SPDE~\eqref{eq-SPDEsyst} has equal Brownian motions $\beta_1=\beta_2$.}
\label{tabDPsyst1}
\end{table}

We now verify that the DPLT splitting scheme~\eqref{eq-dpltsyst} converges with order $\frac12-$. We test two cases: the Brownian motions $\beta_1$ and $\beta_2$ are either independent or equal.
We consider the system of SPDEs~\eqref{eq-SPDEsyst} with initial conditions $u_{1,0}(x)=\sin(2\pi x)$ and $u_{2,0}(x)=\sin(2\pi x)$ for $x\in[0,1]$ and discretise in space using finite differences with mesh size $h=2^{-8}$. 
We apply the time integrators {\upshape sEM}, {\upshape SEXP}, and the {\upshape DPLT} splitting scheme~\eqref{eq-dpltsyst} on the time interval $[0,1]$.
As in Subsection~\ref{subNumExp}, the classical integrators {\upshape sEM} and {\upshape SEXP} are applied with appropriate extensions $\drstar_j$ and $\distar_j$ of the coefficients $\dr_j$ and $\di_j$ given by a variant of equation~\eqref{eq-starcoeff}, to ensure their mean-square convergence with order $\frac12-$ as the coefficients satisfy appropriate global Lipschitz continuity properties in the third and fourth variables $u_1$ and $u_2$.
The following drift and diffusion coefficients are considered: 
\begin{align*}
    \drf_1(t,x,u_1,u_2)=\sin(u_1)\sin(u_2) &\text{ and } \drf_2(t,x,u_1,u_2)=\cos(u_1)\sin(u_2) \\
    \dif_1(t,x,u_1,u_2)=3\sin(u_1)\cos(u_2) &\text{ and } \dif_2(t,x,u_1,u_2)=5\cos(u_1)\cos(u_2).
\end{align*}
Figure~\ref{fig-strongsyst} displays the maximum mean-square errors for the time-step sizes $\dt=2^{-4},\ldots,2^{-15}$. 
The reference solution is the {\upshape DPLT} splitting scheme~\eqref{eq-dpltsyst} with time-step size $\dt^{\text{ref}}=2^{-16}$, and we use $200$ samples to approximate the expectations. 
Both for independent Brownian motions $\beta_1$ and $\beta_2$ (left) and for equal Brownian motions $\beta_1=\beta_2$ (right), we observe the rate of convergence $\frac12$ agrees with the theoretical rate of $\frac{1}{2}$, as predicted by the expected generalisation of Corollary~\ref{lecorollaire}. The comparison with the classical integrators {\upshape sEM} and {\upshape SEXP} justifies that the proposed DPLT splitting scheme~\eqref{eq-dpltm} is consistent with the SPDE~\eqref{eq-SPDEm}.

\begin{figure}[h]
    \centering
    \begin{subfigure}{.45\textwidth}
      \centering
      \includegraphics[width=\textwidth]{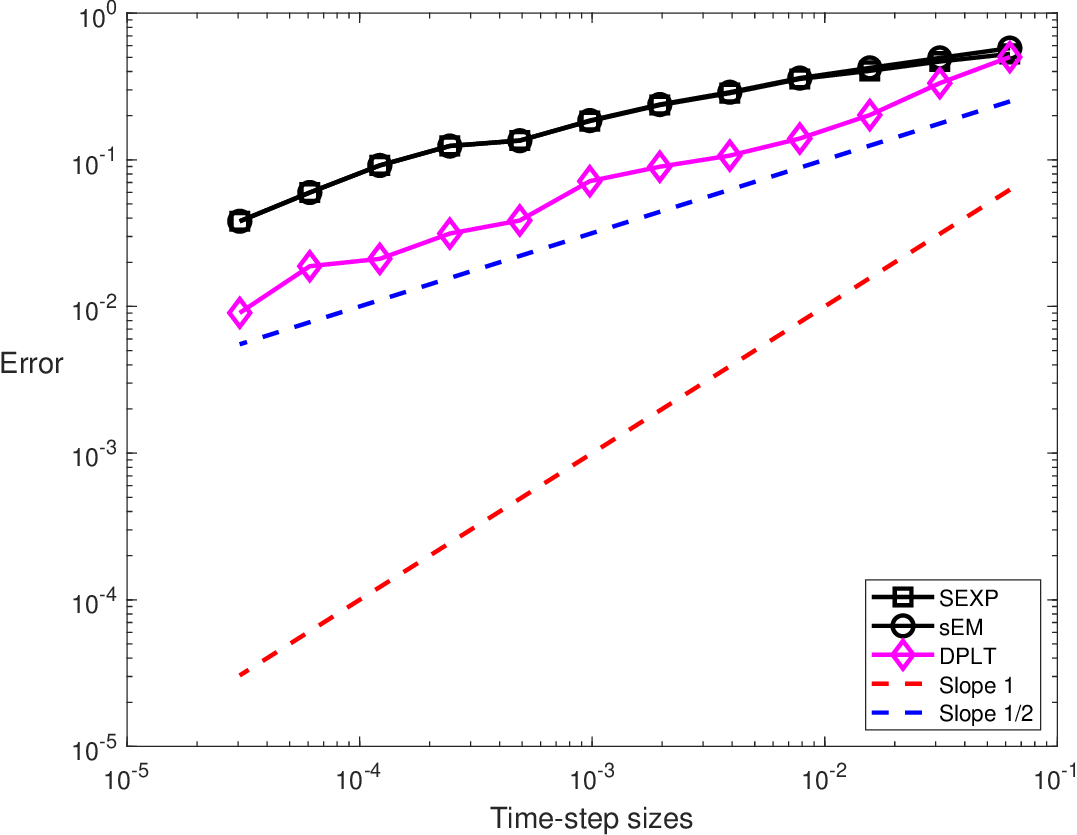}
      \caption{Independent Brownian motions $\beta_1$ and $\beta_2$}
    \end{subfigure}
    \begin{subfigure}{.45\textwidth}
      \centering
      \includegraphics[width=\textwidth]{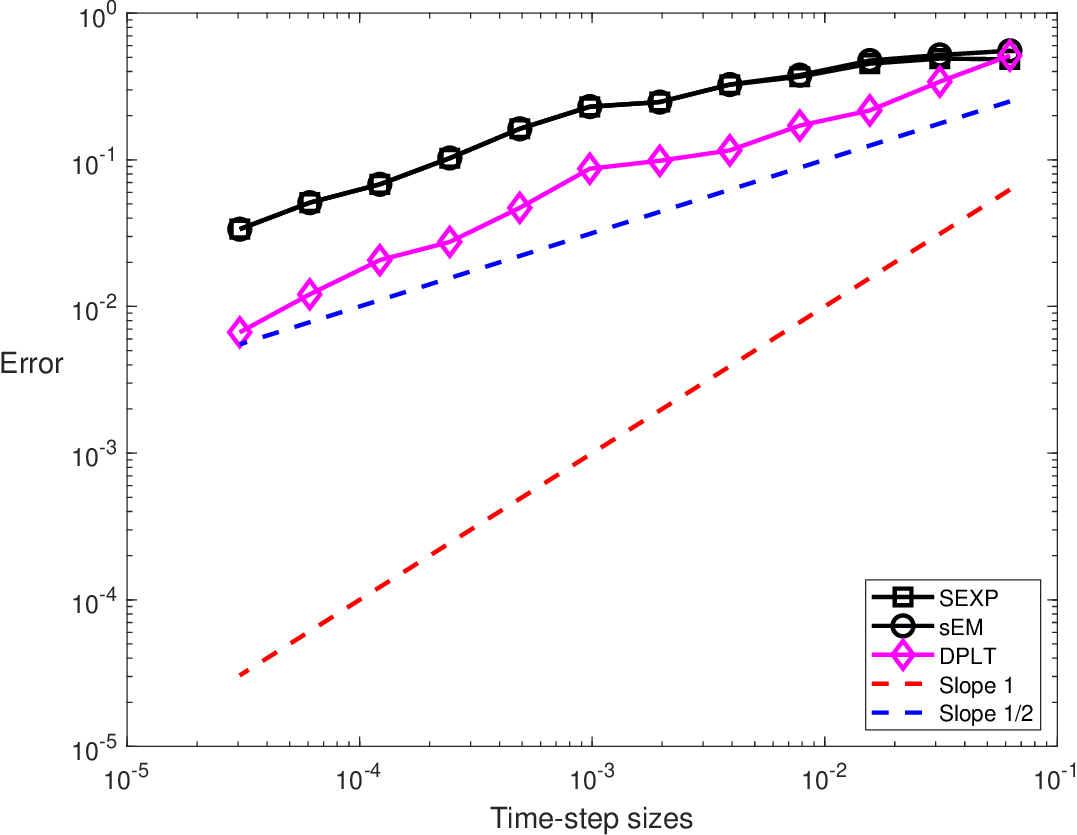}
      \caption{Equal Brownian motions $\beta_1=\beta_2$}
    \end{subfigure}%
    \caption{Maximum mean-square errors for the {\upshape sEM}, {\upshape SEXP}, and {\upshape DPLT} time integrators on the time interval $[0,1]$ for the system of SPDEs~\eqref{eq-SPDEsyst} in one-dimensional space, with space-mesh size $h=2^{-8}$.}
   \label{fig-strongsyst}
\end{figure}

\bibliographystyle{plain}
\bibliography{labib}

\section{Acknowledgements}
This work was initiated thanks to the support of the SFVE-A program.
The work of DC and GC was partially supported by the Swedish Research Council (VR) (project nr. $2024-04536$).
The computations were performed on resources provided by
the National Academic Infrastructure for Supercomputing in Sweden (NAISS) at Vera, Chalmers e-Commons
at Chalmers University of Technology and partially funded by the Swedish Research Council
through grant agreement no. 2022-06725.

\end{document}